\documentclass[12pt]{amsart}

\usepackage{amsmath,amsfonts,amssymb,mathabx,latexsym,mathtools}
\usepackage{enumitem}
\usepackage[usenames, dvipsnames]{xcolor}
\usepackage{hyperref}
\usepackage{ytableau}
\usepackage{centernot}
\usepackage{tikz-cd}
\usepackage{tikz}
\usetikzlibrary{decorations.markings}
\usepackage{stmaryrd }
\usepackage{wrapfig}

\usetikzlibrary{calc, shapes, backgrounds,arrows,positioning,plotmarks}
\tikzset{>=stealth',
  head/.style = {fill = white, text=black},
  plaque/.style = {draw, rectangle, minimum size = 10mm}, 
  pil/.style={->,thick},
  junct/.style = {draw,circle,inner sep=0.5pt,outer sep=0pt, fill=black}
  }

\newtheorem{theorem}{Theorem}[section]
\newtheorem{lemma}[theorem]{Lemma}

\newtheorem{proposition}[theorem]{Proposition}
\newtheorem{corollary}[theorem]{Corollary}

\theoremstyle{definition}

\newtheorem{examplex}[theorem]{Example}
\newenvironment{example}
  {\pushQED{\qed}\examplex}
  {\popQED\endexamplex}

\theoremstyle{remark}
\newtheorem{remark}[theorem]{Remark}

\numberwithin{equation}{section}

\usepackage[colorinlistoftodos]{todonotes}

\newcommand{\id}{\ensuremath{\mathrm{id}}}

\newcommand{\Schub}{\mathfrak{S}}
\newcommand{\Groth}{\mathfrak{G}}

\newcommand{\SymGp}{S}
\newcommand{\Mute}{{\sf Mute}}

\newcommand{\mydef}[1]{{\bf #1}}

\newcommand{\bpd}[1]{{\sf BPD}(#1)}
\newcommand{\mbpd}[1]{{\sf MBPD}(#1)}
\newcommand{\ncbpd}[1]{{\sf NCBPD}(#1)}
\newcommand{\hbpd}[1]{{\sf HBPD}(#1)}

\newcommand{\code}{{\sf c}}
\newcommand{\des}{{\sf des}}
\newcommand{\mc}{{\sf mc}}

\newcommand{\asmtransition}{{\sf t}}

\newcommand{\wt}[1]{{\tt wt}({#1})}

\newcommand{\yd}[1]{{\mathbb Y}(#1)}

\newcommand{\asm}{{\sf ASM}}

\newcommand{\pipes}[1]{{\sf Pipes}(#1)}
\newcommand{\rpipes}[1]{{\sf RPipes}(#1)}
\newcommand{\mpipes}[1]{{\sf MPipes}(#1)}

\newcommand{\paths}{{\tt Paths}}

\newcommand{\FSSYT}{{\sf FSYT}}

\newcommand{\FSet}{\overline{{\sf FSYT}}}

\newcommand{\SatFSet}{{\overline{{\sf FSYT}}_S}}

\newcommand{\DT}{{\sf DT}}

\newcommand{\RWT}{{\sf RWT}}

\newcommand{\mkey}[1]{\kappa(#1)}

\newcommand{\demprod}[1]{\partial(#1)}

\newcommand{\flatten}{{\tt flatten}}

\newcommand{\asmtobpd}{\Phi}

\newcommand{\flaggedtovexmarkedbpd}{\overline{\gamma}}
\newcommand{\FSSYTtovexbpd}{\gamma}
\newcommand{\HeckeBPDtodectab}{\Omega}
\newcommand{\vexBPDtodectab}{\Gamma}

\newcommand{\restrictBPD}{{\tt res}}
\newcommand{\completebpd}{{\tt comp}}

\begin{document}


\title{Bumpless Pipe Dreams and Alternating Sign Matrices}

\author[A. Weigandt]{Anna Weigandt}
\address[AW]{Department of Mathematics, University of Michigan, Ann Arbor, MI 48109}
\email{weigandt@umich.edu}


\date{\today}



\begin{abstract}
In their work on the infinite flag variety, Lam, Lee, and Shimozono (2018) introduced objects called bumpless pipe dreams and used them to give a formula for double Schubert polynomials. We extend this formula to the setting of K-theory, giving an expression for double Grothendieck polynomials as a sum over a larger class of bumpless pipe dreams.  Our proof relies on techniques found in an unpublished manuscript of Lascoux (2002).  Lascoux showed how to write double Grothendieck polynomials as a sum over alternating sign matrices.
We explain how to view the Lam-Lee-Shimozono formula as a disguised special case of Lascoux's alternating sign matrix formula. 
  
Knutson, Miller, and Yong (2009) gave a tableau formula for  vexillary Grothendieck polynomials.  We recover this formula by showing vexillary marked bumpless pipe dreams and flagged set-valued tableaux are in weight preserving bijection. Finally, we give a bijection between Hecke bumpless pipe dreams and decreasing tableaux.  The restriction of this bijection to Edelman-Greene bumpless pipe dreams solves a problem of Lam, Lee, and Shimozono.

\end{abstract}

\maketitle

\section{Introduction}
\label{s:intro}

Lascoux and Sch\"utzenberger \cite{Lascoux.Schutzenberger} introduced double Grothendieck polynomials, which represent classes in the equivariant K-theory of the complete flag variety.  The initial definition was in terms of divided difference operators, but in the intervening years, authors have put forward many combinatorial models to study Grothendieck polynomials (see, e.g.,\ \cite{Fomin.Kirillov,Knutson.Miller,Lenart.Robinson.Sottile}).  One purpose of this article is to shine a spotlight on a lesser known formula of Lascoux \cite{Lascoux:ice}.  Lascoux's formula realizes each Grothendieck polynomial as a weighted sum over alternating sign matrices (ASMs). Our goal is to make explicit connections between Lascoux's formula and subsequent work in the literature.

In the context of back stable Schubert calculus, Lam, Lee, and Shimozono \cite{Lam.Lee.Shimozono} introduced bumpless pipe dreams (BPDs) and used them to give a formula for double Schubert polynomials.
We extend this formula to the K-theoretic setting and give a bumpless pipe dream formula for double Grothendieck polynomials.  This formula is closely related to Lascoux's formula in terms of ASMs.  There is a natural, weight preserving bijection between  bumpless pipe dreams and ASMs.  Indeed, bumpless pipe dreams are transparently in bijection with the osculating lattice paths of statistical mechanics.  
  We  connect Lascoux's ASM formula to bumpless pipe dreams by observing that the  key of an ASM is the same as the Demazure product of its corresponding bumpless pipe dream (see Theorem~\ref{theorem:demkey}).

The proof of the K-theoretic bumpless pipe dream formula  follows the outline set out in \cite{Lascoux:ice} in terms of ASMs.  These techniques translate in a natural way to BPDs.  We provide additional details which were omitted by Lascoux.  
The main ingredient  is  a transition formula for double Grothendieck polynomials (see Theorem~\ref{thm:transition}).  Lascoux stated this formula without proof.
We provide one in Appendix~\ref{appendix:transitionproof}.  
The bumpless pipe dream formula for Grothedieck polynomials follows by observing the weights on bumpless pipe dreams are compatible with transition (see Proposition~\ref{prop:asmtransition} and Lemma~\ref{lemma:bpdtransitionweights}).

For the remainder of the introduction, we proceed with a summary of our main results.

\subsection{Bumpless pipe dreams}
  Start with the six tiles pictured below.
\begin{equation}
\label{eqn:sixtiles}
\raisebox{-.5em}{
	\begin{tikzpicture}[x=1.5em,y=1.5em]
	\draw[color=black, thick](0,1)rectangle(1,2);
	\draw[thick,rounded corners,color=blue] (.5,1)--(.5,1.5)--(1,1.5);
	\end{tikzpicture}
	\hspace{2em}
	\begin{tikzpicture}[x=1.5em,y=1.5em]
	\draw[color=black, thick](0,1)rectangle(1,2);
	\draw[thick,rounded corners,color=blue] (.5,2)--(.5,1.5)--(0,1.5);
	\end{tikzpicture}
	\hspace{2em}
	\begin{tikzpicture}[x=1.5em,y=1.5em]
	\draw[color=black, thick](0,1)rectangle(1,2);
	\draw[thick,rounded corners,color=blue] (0,1.5)--(1,1.5);
	\draw[thick,rounded corners,color=blue] (.5,1)--(.5,2);
	\end{tikzpicture}
	\hspace{2em}
	\begin{tikzpicture}[x=1.5em,y=1.5em]
	\draw[color=black, thick](0,1)rectangle(1,2);
	\end{tikzpicture}
	\hspace{2em}
	\begin{tikzpicture}[x=1.5em,y=1.5em]
	\draw[color=black, thick](0,1)rectangle(1,2);
	\draw[thick,rounded corners,color=blue] (0,1.5)--(1,1.5);
	\end{tikzpicture}
	\hspace{2em}
	\begin{tikzpicture}[x=1.5em,y=1.5em]
	\draw[color=black, thick](0,1)rectangle(1,2);
	\draw[thick,rounded corners,color=blue] (.5,1)--(.5,2);
	\end{tikzpicture}}
\end{equation}
These tiles are the building blocks for a network of pipes. We interpret each ``plus'' tile as a place where two pipes cross, one pipe running horizontally and the other vertically.  
A  {\bf bumpless pipe dream} is a tiling of the $n\times n$ grid with the tiles in (\ref{eqn:sixtiles}) so that
\begin{enumerate}
	\item there are $n$ total pipes,
	\item each pipe starts vertically at the bottom edge of the grid, and
	\item pipes end  horizontally at the right edge of the grid.
\end{enumerate}
Write $\bpd{n}$ for the set of $n\times n$  bumpless pipe dreams.

Bumpless pipe dreams are in transparent bijection with certain  \emph{osculating lattice paths}  from statistical mechanics (see, e.g.,\ \cite{Behrend}).  Indeed, osculating lattice paths have nearly the same definition, though the crossing tile is often represented as an ``osculating'' tile \raisebox{-.2em}{
	\begin{tikzpicture}[x=1em,y=1em]
	\draw[color=black, thick](0,1)rectangle(1,2);
	\draw[thick,rounded corners, color=blue] (.5,2)--(.5,1.5)--(0,1.5);
	\draw[thick,rounded corners, color=blue] (1,1.5)--(.5,1.5)--(.5,1);
	\end{tikzpicture}  
}.  In the terminology of Lam-Lee-Shimozono, this tile is called a ``bumping'' tile.  This is  the reason for calling these new pipe dreams ``bumpless.''

The bijection with osculating lattice paths yields a  bijection between bumpless pipe dreams and alternating sign matrices\footnote{This map was described  in \cite[Figure~1]{BousquetMelou.Habsieger}.  The connection to osculating paths is explained in \cite{brak1997osculating}.  See also \cite[Section~4]{Behrend}.}.  Here, we prefer the pipe dream interpretation of these objects,  which considers each BPD to be a planar history of a (possibly non-reduced) word in the symmetric group $\SymGp_n$.  Thus, our use of the crossing tile is crucial.

  Given $\mathcal P\in \bpd{n}$, we write $\demprod{\mathcal P}$ for the \emph{Demazure product} of the \emph{column reading word} of $\mathcal P$ (see Section~\ref{section:planar}).  Roughly, $\demprod{\mathcal P}$ is the permutation obtained by tracing out the paths of each pipe.  Furthermore, if a pair of pipes crosses more than once, we ignore all crossings after the first.  See Figure~\ref{figure:demprod}. 
\begin{figure}[h]
\begin{tikzpicture}[x=1.5em,y=1.5em]
\draw[step=1,gray!40, thin] (0,0) grid (7,7);
\draw[color=black, thick](0,0)rectangle(7,7);
\draw[thick,rounded corners, color=blue](.5,0)--(.5,1.5)--(2.5,1.5)--(2.5,5.5)--(7,5.5);
\draw[thick,rounded corners, color=blue](1.5,0)--(1.5,2.5)--(5.5,2.5)--(5.5,6.5)--(7,6.5);
\draw[thick,rounded corners, color=blue](2.5,0)--(2.5,.5)--(7,.5);
\draw[thick,rounded corners, color=blue](3.5,0)--(3.5,3.5)--(7,3.5);
\draw[thick,rounded corners, color=blue](4.5,0)--(4.5,4.5)--(7,4.5);
\draw[thick,rounded corners, color=blue](5.5,0)--(5.5,1.5)--(7,1.5);
\draw[thick,rounded corners, color=blue](6.5,0)--(6.5,2.5)--(7,2.5);
\draw[](.5,-.5) node {1};
\draw[](1.5,-.5) node {2};
\draw[](2.5,-.5) node {3};
\draw[](3.5,-.5) node {4};
\draw[](4.5,-.5) node {5};
\draw[](5.5,-.5) node {6};
\draw[](6.5,-.5) node {7};
\draw[](7.5,6.5) node {5};
\draw[](7.5,5.5) node {2};
\draw[](7.5,4.5) node {4};
\draw[](7.5,3.5) node {1};
\draw[](7.5,2.5) node {7};
\draw[](7.5,1.5) node {6};
\draw[](7.5,.5) node {3};
\end{tikzpicture}
\hspace{4em}
\begin{tikzpicture}[x=1.5em,y=1.5em]
\draw[step=1,gray!40, thin] (0,0) grid (7,7);
\draw[color=black, thick](0,0)rectangle(7,7);
\draw[thick,rounded corners, color=blue](.5,0)--(.5,1.5)--(2.5,1.5)--(2.5,5.5)--(7,5.5);
\draw[thick,rounded corners, color=blue](1.5,0)--(1.5,2.5)--(5.5,2.5)--(5.5,6.5)--(7,6.5);
\draw[thick,rounded corners, color=blue](2.5,0)--(2.5,.5)--(7,.5);
\draw[thick,rounded corners, color=blue](3.5,0)--(3.5,3.5)--(7,3.5);
\draw[thick,rounded corners, color=blue](4.5,0)--(4.5,4.5)--(7,4.5);
\draw[thick,rounded corners, color=blue](5.5,0)--(5.5,1.5)--(7,1.5);
\draw[thick,rounded corners, color=blue](6.5,0)--(6.5,2.5)--(7,2.5);

\filldraw[color=gray!40, fill=white, thin](2,2)rectangle(3,3);
\draw[color=blue,thick,rounded corners] (2.5,1.9)--(2.5,2.5)--(3.1,2.5);
\draw[color=blue,thick,rounded corners] (2.5,3.1)--(2.5,2.5)--(1.9,2.5);

\filldraw[color=gray!40, fill=white, thin](5,3)rectangle(6,4);
\draw[thick,color=blue](5.5,0)--(5.5,1);
\draw[thick,color=blue](5,.5)--(6,.5);

\filldraw[color=gray!40, fill=white, thin](5,4)rectangle(6,5);
\draw[color=blue,thick,rounded corners] (5.5,2.9)--(5.5,3.5)--(6.1,3.5);
\draw[color=blue,thick,rounded corners] (5.5,4.1)--(5.5,3.5)--(4.9,3.5);

\draw[color=blue,thick,rounded corners] (5.5,3.9)--(5.5,4.5)--(6.1,4.5);
\draw[color=blue,thick,rounded corners] (5.5,5.1)--(5.5,4.5)--(4.9,4.5);
\draw[](.5,-.5) node {1};
\draw[](1.5,-.5) node {2};
\draw[](2.5,-.5) node {3};
\draw[](3.5,-.5) node {4};
\draw[](4.5,-.5) node {5};
\draw[](5.5,-.5) node {6};
\draw[](6.5,-.5) node {7};
\draw[](7.5,6.5) node {5};
\draw[](7.5,5.5) node {2};
\draw[](7.5,4.5) node {4};
\draw[](7.5,3.5) node {1};
\draw[](7.5,2.5) node {7};
\draw[](7.5,1.5) node {6};
\draw[](7.5,.5) node {3};
\end{tikzpicture} 
\caption{If $\mathcal P$ is the BPD pictured on the left, $\demprod{\mathcal P} =5241763$.  The diagram on the right indicates which crossings should be ignored.}
\label{figure:demprod}
\end{figure}
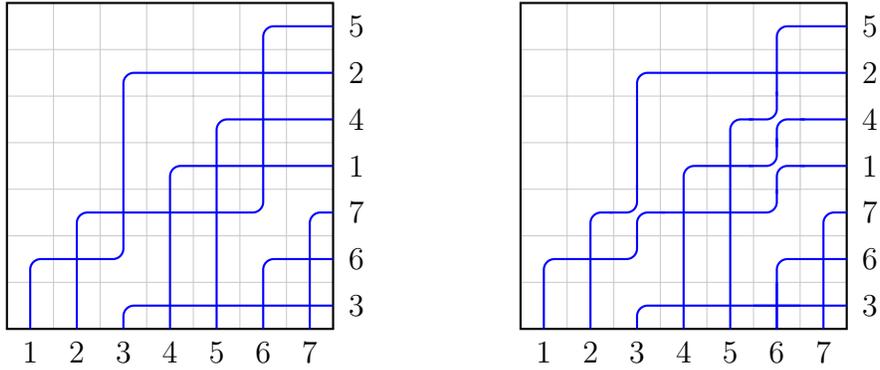

Fix $w\in \SymGp_n$ and let \[\pipes{w}=\{\mathcal P\in \bpd{n}:\demprod{\mathcal P}=w\}.\] A bumpless pipe dream is  \mydef{reduced} if any pair of pipes crosses at most once.
 Write \[\rpipes{w}=\{\mathcal P\in \pipes{w}:\mathcal P \enspace \text{is reduced}\}.\]  

Lam, Lee, and Shimozono showed that $\rpipes{w}$ is connected by \emph{droop moves} on BPDs.  We define \emph{K-theoretic droops} and show that $\pipes{w}$ is connected by droop moves combined with K-theoretic droops (see Proposition~\ref{prop:droops}).

\subsection{The bumpless pipe dream formula for $\beta$-double Grothendieck polynomials}

We work in the setting of the $\beta$-double Grothendieck polynomials of \cite{Fomin.Kirillov}.  See Section~\ref{subsec:betagroth} for this definition.  These polynomials represent classes in connective K-theory \cite{Hudson}.  Certain specializations recover (double) Schubert and Grothendieck polynomials.  
In this section, we introduce the bumpless pipe dream formula for $\beta$-double Grothendieck polynomials.

Let
\[D(\mathcal P):=\{(i,j): \mathcal P \enspace \text{has a} \, 
\raisebox{-.2em}{
	\begin{tikzpicture}[x=1em,y=1em]
	\draw[color=black, thick](0,1)rectangle(1,2);
	\end{tikzpicture}  
}
\, \text{tile in row $i$ and column $j$}\} \] 
and
\[U(\mathcal P):=\{(i,j): \mathcal P \enspace \text{has a} \, 
\raisebox{-.2em}{
	\begin{tikzpicture}[x=1em,y=1em]
	\draw[color=black, thick](0,1)rectangle(1,2);
	\draw[thick,rounded corners, color=blue] (.5,2)--(.5,1.5)--(0,1.5);
	\end{tikzpicture}  
}
\, \text{tile in row $i$ and column $j$}\}.  \] 
Define $x_i\oplus y_j=x_i+y_j+\beta x_iy_j$.  We associate to $\mathcal P$ the weight
\[\wt{\mathcal P}=\left( \prod_{(i,j)\in D(\mathcal P)}\beta(x_i\oplus y_j)\right) \left(\prod_{(i,j)\in U(\mathcal P)} 1+\beta(x_i\oplus y_j) \right).\]

\begin{theorem}
\label{thm:main}
The $\beta$-double Grothendieck polynomial for $w\in \SymGp_{n}$ is the weighted sum
\[\Groth^{(\beta)}_w(\mathbf x;\mathbf y)=\beta^{-\ell(w)}\sum_{\mathcal P\in \pipes{w}}\wt{\mathcal P}.\]
\end{theorem}
We prove Theorem~\ref{thm:main} in Section~\ref{section:pfmain}.
\begin{example}
 The elements of $\pipes{2143}$ are pictured below.
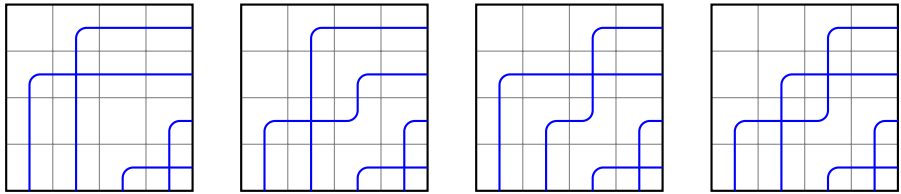
\[
\begin{tikzpicture}[x=1.5em,y=1.5em]
\draw[step=1,gray, thin] (0,1) grid (4,5);
\draw[color=black, thick](0,1)rectangle(4,5);
\draw[thick,rounded corners, color=blue] (.5,1)--(.5,3.5)--(4,3.5);
\draw[thick,rounded corners, color=blue] (1.5,1)--(1.5,4.5)--(4,4.5);
\draw[thick,rounded corners, color=blue] (2.5,1)--(2.5,1.5)--(4,1.5);
\draw[thick,rounded corners, color=blue] (3.5,1)--(3.5,2.5)--(4,2.5);
\end{tikzpicture}
\hspace{1.5em}
\begin{tikzpicture}[x=1.5em,y=1.5em]
\draw[step=1,gray, thin] (0,1) grid (4,5);
\draw[color=black, thick](0,1)rectangle(4,5);
\draw[thick,rounded corners, color=blue] (.5,1)--(.5,2.5)--(2.5,2.5)--(2.5,3.5)--(4,3.5);
\draw[thick,rounded corners, color=blue] (1.5,1)--(1.5,4.5)--(4,4.5);
\draw[thick,rounded corners, color=blue] (2.5,1)--(2.5,1.5)--(4,1.5);
\draw[thick,rounded corners, color=blue] (3.5,1)--(3.5,2.5)--(4,2.5);
\end{tikzpicture}
\hspace{1.5em}
\begin{tikzpicture}[x=1.5em,y=1.5em]
\draw[step=1,gray, thin] (0,1) grid (4,5);
\draw[color=black, thick](0,1)rectangle(4,5);
\draw[thick,rounded corners, color=blue] (.5,1)--(.5,3.5)--(4,3.5);
\draw[thick,rounded corners, color=blue] (1.5,1)--(1.5,2.5)--(2.5,2.5)--(2.5,4.5)--(4,4.5);
\draw[thick,rounded corners, color=blue] (2.5,1)--(2.5,1.5)--(4,1.5);
\draw[thick,rounded corners, color=blue] (3.5,1)--(3.5,2.5)--(4,2.5);
\end{tikzpicture}
\hspace{1.5em}
\begin{tikzpicture}[x=1.5em,y=1.5em]
\draw[step=1,gray, thin] (0,1) grid (4,5);
\draw[color=black, thick](0,1)rectangle(4,5);
\draw[thick,rounded corners, color=blue] (.5,1)--(.5,2.5)--(2.5,2.5)--(2.5,4.5)--(4,4.5);
\draw[thick,rounded corners, color=blue] (1.5,1)--(1.5,3.5)--(4,3.5);
\draw[thick,rounded corners, color=blue] (2.5,1)--(2.5,1.5)--(4,1.5);
\draw[thick,rounded corners, color=blue] (3.5,1)--(3.5,2.5)--(4,2.5);
\end{tikzpicture}\]
Thus, applying Theorem~\ref{thm:main}, we see that
\begin{align*}
\Groth^{(\beta)}_{2143}(\mathbf x;\mathbf y)&=(x_1\oplus y_1)(x_3\oplus y_3)+(x_1\oplus y_1)(x_2\oplus y_1)(1+\beta (x_3\oplus y_3))\\ &\quad +(x_1\oplus y_1)(x_1\oplus y_2)(1+\beta(x_3\oplus y_3))\\
& \quad +\beta(x_1\oplus y_1)(x_1 \oplus y_2) (x_2\oplus y_1)(1+\beta(x_3\oplus y_3)). \qedhere
\end{align*}
\end{example}

As a corollary to Theorem~\ref{thm:main}, we obtain the following.
\begin{theorem}[\cite{Lam.Lee.Shimozono}]
	\label{theorem:bumplessSchubert}
	The double Schubert polynomial for $w\in \SymGp_n$ is a sum over reduced bumpless pipe dreams:
	\[\Schub_w(\mathbf x;\mathbf y)=	\sum_{\mathcal P\in {\sf RPipes}(w)} \prod_{(i,j)\in D(\mathcal P)}(x_i-y_j).\]
\end{theorem}

\begin{example}
Pictured below are the elements of ${\sf RPipes}(2143)$.
	
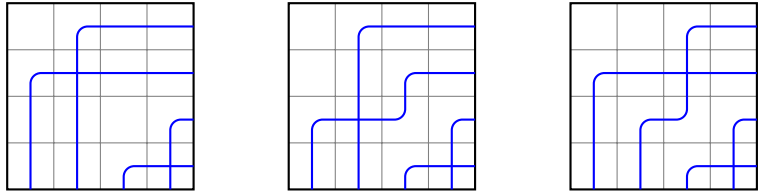
\[
	\begin{tikzpicture}[x=1.5em,y=1.5em]
	\draw[step=1,gray, thin] (0,1) grid (4,5);
	\draw[color=black, thick](0,1)rectangle(4,5);
	\draw[thick,rounded corners, color=blue] (.5,1)--(.5,3.5)--(4,3.5);
	\draw[thick,rounded corners, color=blue] (1.5,1)--(1.5,4.5)--(4,4.5);
	\draw[thick,rounded corners, color=blue] (2.5,1)--(2.5,1.5)--(4,1.5);
	\draw[thick,rounded corners, color=blue] (3.5,1)--(3.5,2.5)--(4,2.5);
	\end{tikzpicture}
	\hspace{3em}
	\begin{tikzpicture}[x=1.5em,y=1.5em]
	\draw[step=1,gray, thin] (0,1) grid (4,5);
	\draw[color=black, thick](0,1)rectangle(4,5);
	\draw[thick,rounded corners, color=blue] (.5,1)--(.5,2.5)--(2.5,2.5)--(2.5,3.5)--(4,3.5);
	\draw[thick,rounded corners, color=blue] (1.5,1)--(1.5,4.5)--(4,4.5);
	\draw[thick,rounded corners, color=blue] (2.5,1)--(2.5,1.5)--(4,1.5);
	\draw[thick,rounded corners, color=blue] (3.5,1)--(3.5,2.5)--(4,2.5);
	\end{tikzpicture}
	\hspace{3em}
	\begin{tikzpicture}[x=1.5em,y=1.5em]
	\draw[step=1,gray, thin] (0,1) grid (4,5);
	\draw[color=black, thick](0,1)rectangle(4,5);
	\draw[thick,rounded corners, color=blue] (.5,1)--(.5,3.5)--(4,3.5);
	\draw[thick,rounded corners, color=blue] (1.5,1)--(1.5,2.5)--(2.5,2.5)--(2.5,4.5)--(4,4.5);
	\draw[thick,rounded corners, color=blue] (2.5,1)--(2.5,1.5)--(4,1.5);
	\draw[thick,rounded corners, color=blue] (3.5,1)--(3.5,2.5)--(4,2.5);
	\end{tikzpicture}\]
	Therefore,
	$		\mathfrak S_{2143}(
	\mathbf x;\mathbf y)= (x_1-y_1)(x_3-y_3)+(x_1-y_1)(x_2-y_1)+(x_1-y_1)(x_1-y_2). 
	$
\end{example}

We will also study \mydef{marked bumpless pipe dreams}, that is \[\mbpd{n}:=\{(\mathcal P,\mathcal S):\mathcal P\in \bpd{n} \enspace \text{and} \enspace \mathcal S\subseteq U(\mathcal P)\}.\]  
Graphically, we represent  $(\mathcal P,\mathcal S)$ by drawing $\mathcal P$ as usual and shading each cell which belongs to $\mathcal S$.
Write \[\mpipes{w}=\{(\mathcal P,\mathcal S)\in \mbpd{n}:\mathcal P\in \pipes{w}\}.\]
As an immediate corollary to Theorem~\ref{thm:main}, we have
\begin{corollary}
\label{cor:main}
The $\beta$-double Grothendieck polynomial is the weighted sum
\[\Groth^{(\beta)}_w(\mathbf x;\mathbf y)=\sum_{(\mathcal P,\mathcal S)\in \mpipes{w}}\beta^{|D(\mathcal P)|+|\mathcal S|-\ell(w)}\left(\prod_{(i,j)\in D(\mathcal P)\cup \mathcal  S}(x_i\oplus y_j)\right).\]
\end{corollary}

For the reader familiar with pipe dreams (also known as RC-graphs) in the sense of \cite{Bergeron.Billey,Fomin.Kirillov.1996,Knutson.Miller}, the BPD formulas for double Schubert and Grothendieck polynomials are genuinely different from the pipe dream formulas. In small examples, such as $w=132$, we see obstructions to finding a weight preserving bijection which explains the equality of these expressions (see Example~\ref{ex:132}). However,  after specializing the $\mathbf y$ variables to $0$, it should, in principle, be possible to find a weight preserving bijection from pipe dreams to marked BPDs.  Finding a direct bijection is expected to be difficult.  See Section~\ref{section:comparisons} for further discussion.

\subsection{Vexillary bumpless pipe dreams and flagged tableaux}

 A \mydef{partition} is a weakly decreasing sequence of nonnegative integers $\lambda=(\lambda_1,\,\lambda_2,\,\ldots,\,\lambda_k)$.  Write \[\yd{\lambda}:=\{(i,j):1\leq j\leq \lambda_i \enspace \text{and} \enspace 1\leq i\leq k\}.\]  A \mydef{tableau of shape $\lambda$} associates a positive integer to each $(i,j)\in \yd{\lambda}$.  We will discuss \emph{semistandard tableaux}, \emph{flagged tableaux}, and \emph{set-valued tableaux} in this section.  See Section~\ref{section:partandtab} for these definitions.

In Section~\ref{s:vex}, we restrict our attention to bumpless pipe dreams for \emph{vexillary permutations}, i.e.,\ those permutations which avoid the pattern 2143.
We show $\pipes{v}=\rpipes{v}$ if and only if $v$ is vexillary (see Lemma~\ref{lemma:vexnoncrossing}).  
Furthermore, if $v$ is vexillary, $\pipes{v}$ is in transparent bijection with the sets of non-intersecting lattice paths studied by Kreiman \cite{Kreiman}.  

 As a consequence of Kreiman's work, we obtain a bijection \[\FSSYTtovexbpd:\FSSYT(v)\rightarrow \pipes{v}\] where   $\FSSYT(v)$ is the set of flagged semistandard tableaux for $v$.  
Explicitly,  the map takes a tableau $T$ of shape $\mu^{(v)}$ to the unique $\mathcal P\in \pipes{v}$ so that 
\begin{equation}
\label{eqn:bpdmaptabmapdiagram}
D(\mathcal P)=\{(T(i,j),T(i,j)+j-i):(i,j)\in\yd{\mu^{(v)}}\}.
\end{equation}
In particular, $\pipes{v}$ is in  bijection with certain sets of \emph{excited Young diagrams} \cite{Ikeda.Naruse}, as well as the (reduced) \emph{diagonal pipe dreams} of \cite{Knutson.Miller.Yong}.

Continuing Kreiman's lattice path story, we define a map
\[\flaggedtovexmarkedbpd:\FSet(v)\rightarrow \mpipes{v}\]
 where $\FSet(v)$ is the set of flagged set-valued tableaux for $v$. 
If $\mathbf T\in \FSet(v)$, the minimum element of each cell of $\mathbf T$ determines the diagram of the underlying BPD in the same way as (\ref{eqn:bpdmaptabmapdiagram}).  The other elements indicate which upward elbow tiles are marked.  We make this precise in Section~\ref{section:tabpipes}.  See Figure~\ref{figure:markedbpdvex} for an example.
\begin{theorem}
\label{theorem:flaggedtovexmarkedbpdbij}
Fix $v\in \SymGp_{n}$ so that $v$ is vexillary.  The map \[\flaggedtovexmarkedbpd:\FSet(v)\rightarrow \mpipes{v}\] is a weight preserving bijection.
\end{theorem}

\begin{figure}
\[	\raisebox{6em}{
\begin{ytableau}
1&1\\
2 3\\
4
\end{ytableau}}\hspace{4em}\begin{tikzpicture}[x=1.5em,y=1.5em]
	\draw[step=1,gray, thin] (0,0) grid (5,5);
	\draw[color=black, thick](0,0)rectangle(5,5);
	\draw[color=black,fill=lightgray, thick](1,2)rectangle(2,3);
	\draw[thick,rounded corners, color=blue] (.5,0)--(.5,2.5)--(1.5,2.5)--(1.5,3.5)--(2.5,3.5)--(2.5,4.5)--(5,4.5);
	\draw[thick,rounded corners, color=blue] (2.5,0)--(2.5,2.5)--(5,2.5);
	\draw[thick,rounded corners, color=blue] (1.5,0)--(1.5,.5)--(5,.5);
	\draw[thick,rounded corners, color=blue] (3.5,0)--(3.5,3.5)--(5,3.5);
	\draw[thick,rounded corners, color=blue] (4.5,0)--(4.5,1.5)--(5,1.5);
	\end{tikzpicture} 
\]
\caption{ Pictured on the left is a  set-valued tableau in $\FSet(14352)$.  On the right is its corresponding  marked bumpless pipe dream.}
\label{figure:markedbpdvex}
\end{figure}
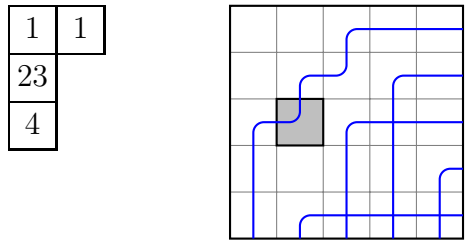
This bijection  provides a new combinatorial proof of the formula for vexillary double Grothendieck polynomials given in \cite{Knutson.Miller.Yong}.  See Theorem~\ref{thm:KMY}.

\subsection{Decreasing tableaux and Hecke bumpless pipe dreams}

Write $\DT(n-1)$ for the set of \emph{decreasing tableaux} with entries in $\{1,2,\ldots,n-1\}$.  Each decreasing tableau has a \emph{column reading word} formed by reading its labels within columns from bottom to top, starting at the left (see Figure~\ref{figure:hecketab}).

Say that $\mathcal P\in \bpd{n}$ is a \mydef{Hecke} BPD if $D(\mathcal P)$ forms a northwest justified partition shape. Write $\hbpd{n}$ for the set of Hecke BPDs in $\bpd{n}$.  If $D(\mathcal P)$ corresponds to the partition $\lambda$, say $\mathcal P$ has \mydef{shape} $\lambda$.  If $\mathcal P$ is Hecke and reduced, then it is an \mydef{Edelman-Greene} BPD, as defined in \cite{Lam.Lee.Shimozono}.

We define a (shape preserving) map from $\hbpd{n}$ to $\DT(n-1)$ as follows.  First, notice that within any diagonal of the grid, $\mathcal P$ has as many crossing tiles as blank tiles (see Lemma~\ref{lemma:uniquelydet}).  We pair the first blank tile within a diagonal with the last crossing tile, the next with the second to last, and so on.  With this convention, if a blank tile sits in row $i$ and its paired crossing tile in row $i'$, we assign the corresponding cell the value $i'-i$.
Write \[\HeckeBPDtodectab:\hbpd{n}\rightarrow \DT(n-1)\]
for this map.  See Figure~\ref{figure:hecketab} for an example.

\begin{theorem}
\label{thm:hecketabs}
The  map 
$\HeckeBPDtodectab:\hbpd{n}\rightarrow \DT(n-1)$
is a shape preserving bijection. Furthermore, if $\mathcal P\in \pipes{w}$, then the reading word of $\HeckeBPDtodectab(\mathcal P)$ is a Hecke word for $w$.
\end{theorem}

\begin{figure}
\[
	\raisebox{9em}{
\begin{ytableau}
6&5&4&2&1\\
4&3&2\\
3&2&1\\
1
\end{ytableau}} \hspace{4em}
\begin{tikzpicture}[x=1.5em,y=1.5em]
	\draw[step=1,gray, thin] (0,0) grid (7,7);
	\draw[color=black, thick](0,0)rectangle(7,7);
	\draw[thick,rounded corners, color=blue] (.5,0)--(.5,2.5)--(7,2.5);
	\draw[thick,rounded corners, color=blue] (1.5,0)--(1.5,3.5)--(5.5,3.5)--(5.5,6.5)--(7,6.5);
	\draw[thick,rounded corners, color=blue] (2.5,0)--(2.5,1.5)--(7,1.5);
	\draw[thick,rounded corners, color=blue] (3.5,0)--(3.5,5.5)--(7,5.5);
	\draw[thick,rounded corners, color=blue] (4.5,0)--(4.5,4.5)--(7,4.5);
	\draw[thick,rounded corners, color=blue] (5.5,0)--(5.5,.5)--(7,.5);
	\draw[thick,rounded corners, color=blue] (6.5,0)--(6.5,3.5)--(7,3.5);
	\end{tikzpicture} \]
\caption{The decreasing tableau  on the left corresponds to the Hecke BPD $\mathcal P$ on the right.
The tableau has reading word $\mathbf a=(1,3,4,6,2,3,5,1,2,4,2,1)$.
The reader may verify that $\demprod{\mathbf a}=5427136=\demprod{\mathcal P}$. }
\label{figure:hecketab}
\end{figure}
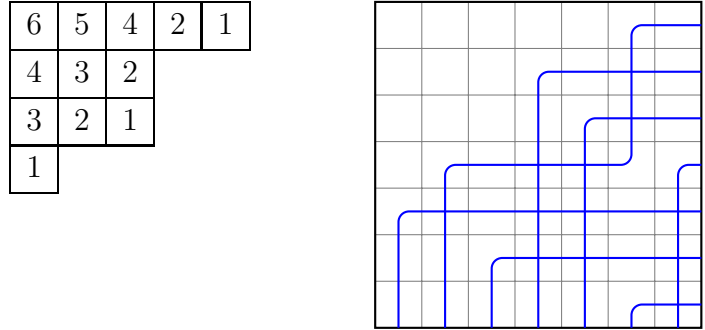

The name Hecke BPD is meant to emphasize the relationship with the Hecke insertion of \cite{BKSTY.factorsequ}.  Each $w\in \SymGp_n$ has an associated formal power series $G_w(x_1,\,x_2,\,\ldots)$, known as a \emph{stable Grothendieck polynomial}.  Buch \cite{Buch} proved we can write $G_w$  as a finite linear combination of stable Grothendieck polynomials corresponding to partitions:
\[G_w=\sum_{\lambda}a_{w\, \lambda}G_\lambda.\]  In particular, the coefficients are  integers.  Lascoux \cite{Lascoux:transition} showed these coefficients alternate in sign with degree.  Buch, Kresch, Shimozono,  Tamvakis, and  Yong proved that $|a_{w\,\lambda}|$ counts increasing tableaux of shape $\lambda$ whose reading words are Hecke words for $w$.
As a consequence of Theorem~\ref{thm:hecketabs} and \cite[Theorem~1]{BKSTY.factorsequ},
$|a_{w\,\lambda}|$ also counts the number of Hecke BPDs of shape $\lambda$ in $\pipes{w}$.

Restricting the map $\HeckeBPDtodectab$ to Edelman-Greene BPDs provides a solution to  \cite[Problem 5.19]{Lam.Lee.Shimozono}. This question was phrased in terms of increasing tableaux. However, we find BPDs to be naturally compatible with decreasing tableaux, so we focus on these objects instead. This is merely a convention shift (see \cite[Section~3.8]{BKSTY.factorsequ}).  A previous bijection from Edelman-Greene BPDs to (increasing) reduced word tableaux  was given by Fan, Guo, and Sun in \cite{fan2018bumpless}.  The statement of their bijection relies on Edelman-Greene insertion and the Lascoux-Sch\"utzenberger transition tree. Our bijection has the advantage that it is more direct to state. 

\subsection{Organization}
 In Section~\ref{s:background}, we recall the necessary background on  reduced words in the symmetric group and Grothendieck polynomials. 
In Section~\ref{section:ice} we review relevant objects from the literature which are in natural bijection with ASMs (and hence BPDs).  We collect certain facts for use in subsequent sections.

In the next two sections, we translate Lascoux's work on ASMs into the language of BPDs.  
Section~\ref{section:key} discusses keys of ASMs, Lascoux's inflation moves, and K-theoretic droops. In particular, we show Lascoux's notion of the key of an ASM is compatible with taking the Demazure product of the associated BPD.  In Section~\ref{section:transitionBPD}, we show that BPDs satisfy the transition equations for Grothendieck polynomials.  This allows us to prove Theorem~\ref{thm:main}.  The foundation of the proofs in Section~\ref{section:transitionBPD} was laid out by Lascoux, but we provide additional details to make this discussion self-contained.  We prove transition in Appendix~\ref{appendix:transitionproof}.

In Section~\ref{section:comparisons}, we recall (ordinary) pipe dreams (see \cite{Bergeron.Billey,Fomin.Kirillov.1996,Fomin.Kirillov,Knutson.Miller}) and compare the pipe dream formulas for Schubert and Grothendieck polynomials to the BPD formulas.
Section~\ref{s:vex} discusses vexillary Grothendieck polynomials.  Finally, in Section~\ref{section:hecketabs}, we show Hecke BPDs are in shape preserving  bijection with decreasing tableaux.

\section{Background}
\label{s:background}

\subsection{The symmetric group}

In this section, we review  necessary definitions related to the symmetric group.  We refer the reader to \cite{Manivel} for further background.
Write $\mathbb N=\{0,1,2,\ldots\}$ and $\mathbb P=\{1,2,\ldots\}$.  Define $[n]=\{1,2,\ldots,n\}$. 

Let $\SymGp_n$ be the \mydef{symmetric group on $n$ letters}, i.e.,\ the group of bijections from $[n]$ to itself.  We often represent permutations in one-line notation, writing $w=w_1 \, w_2 \, \ldots\, w_n$ where $w_i:=w(i)$ for all $i\in[n]$. We will also use cycle notation.  We write  $w_0=n\,n-1\,\ldots\,1$ for the {\bf longest permutation}.

The \mydef{Rothe diagram} of $w$ is the set
\begin{equation}
\label{eqn:rothe}
D(w):=\{(i,j):w_i>j,w^{-1}_j>i \enspace \text{for all} \enspace i,j\in[n]\}.
\end{equation}
The \mydef{Coxeter length} of $w$ is $\ell(w):=|D(w)|$.
The \mydef{Lehmer code} of $w$ is \[\code_w=(\code_w(1),\,\code_w(2),\,\ldots,\,\code_w(n))\] where $\code_w(i)=|\{j:(i,j)\in D(w)\}|$.  The map $w\mapsto \code_w$ defines a bijection from $\SymGp_n$ to \[\{(c_1,\,c_2,\,\ldots,\,c_n)\in \mathbb N^{n}: 0\leq c_i\leq n-i \enspace \text{for all} \enspace i\in[n]\}.\] 
The Lehmer code is well behaved under the natural inclusion $\iota:\SymGp_{n}\rightarrow \SymGp_{n+1}$; we have $\code_{\iota(w)}=(\code_w(1),\ldots,\code_w(n),0)$.  As such, we conflate these finite codes with sequences taking entries in $\mathbb N$ where all but finitely many terms are $0$.  This allows us to extend the map $w\mapsto \code_w$ to  $\SymGp_{\infty}$.

A permutation $v\in \SymGp_n$ is \mydef{vexillary} if it avoids the pattern 2143, i.e.,\ if there are no indices $1\leq i_1<i_2<i_3<i_4\leq n$ so that $v_{i_2}<v_{i_1}<v_{i_4}<v_{i_3}$.  A permutation is \mydef{dominant} if it avoids the pattern 132, i.e.,\ there are no positions $1\leq i_1<i_2<i_3\leq n$ so that $v_{i_1}<v_{i_3}<v_{i_2}$.  A permutation is dominant if and only if its code is a weakly decreasing sequence. 
\subsection{Words in the symmetric group}  In this section, we follow \cite{Knutson.Miller.subword}.
 Write $s_i$ for the transposition $(i\; i+1)$ and $t_{i\,j}$ for $(i \, j)$. The $s_i$'s satisfy \mydef{braid relations} 
\[s_is_{i+1}s_i=s_{i+1}s_is_{i+1} \enspace \text{for all} \enspace i\in[n-2]\]
and \mydef{commutation relations}
\[s_is_j=s_js_i \enspace \text{if} \enspace |i-j|>1.\]
Furthermore, $s_i^2=\id$ for each $i\in [n-1]$.

A \mydef{word} for $w$ is a tuple $\mathbf a=(a_1,\ldots,a_k)$ such that $w=s_{a_1}\cdots s_{a_k}$. This word is \mydef{reduced} if $k=\ell(w)$.
The \mydef{Demazure algebra} is the free $\mathbb Z$-module generated by $\{e_w:w\in \SymGp_n\}$ with multiplication defined by 
\begin{equation}
e_we_{s_i}=
\begin{cases}
e_{ws_i} & \text{if} \enspace \ell(ws_i)>\ell(w) \enspace \text{and}\\
e_{w} & \text{if} \enspace \ell(ws_i)<\ell(w).
\end{cases}
\end{equation}
Write $e_i:=e_{s_i}$.  The $e_i$'s satisfy the same braid and commutation relations as the $s_i$'s:
\[e_ie_{i+1}e_i=e_{i+1}e_ie_{i+1} \enspace \text{for all} \enspace i\in[n-2]\enspace  \text{and} \enspace e_ie_j=e_je_i \enspace \text{if} \enspace |i-j|>1.\]
 Additionally, $e_i^2=e_i$. 

Given a word $\mathbf a=(a_1,\ldots,a_k)$ with letters in $[n-1]$, the \mydef{Demazure product}  $\demprod{\mathbf a}$ is the (unique) element of $\SymGp_n$ so that $e_{a_1}\cdots e_{a_k}=e_{\demprod{\mathbf a}}$. If $\demprod{\mathbf a}=w$, we say that $\mathbf a$ is a \mydef{Hecke word}  for $w$.
 Using the relations among the $e_i$'s, any word $\mathbf a$ can be simplified to a reduced word $\mathbf a'$ such that $\demprod{\mathbf a}=\demprod{\mathbf a'}$.

\subsection{Planar histories}
\label{section:planar}

A \mydef{northeast planar history} $\mathcal P$ is a configuration of $n$ pseudo-lines, constrained to a rectangular region which satisfy the following properties:
\begin{enumerate}
\item each path starts at the bottom of the rectangle,
\item each path ends at the right edge of the rectangle, 
\item paths move in a northward or eastward direction at all times, and
\item paths may cross at a point, but do not travel concurrently.
\end{enumerate}
Bumpless pipe dreams are special cases of northeast planar histories.

Suppose $\mathcal P$ has $k$ crossings.  We obtain a word from $\mathcal P$ as follows.  Order the crossings from left to right, breaking ties within columns by starting at the bottom and moving upwards.  Label each crossing by counting the number of paths which pass weakly northeast of the crossing and then subtracting one.  Write $a_i$ for the label of the $i$th crossing.  Then $\mathbf a_{\mathcal P}=(a_1,\ldots,a_k)$ is the \mydef{word} of the planar history.
We define $w_{\mathcal P}=s_{a_1}\cdots s_{a_k}$.  If $\mathbf a_{\mathcal P}$ is a reduced word, we say $\mathcal P$ is \mydef{reduced}.  

To  compute $w_{\mathcal P}$ graphically, label the paths  at the bottom edge of the rectangle from left to right with the numbers $1,\,2,\,\ldots,\, n$.  Extend these labels across crossings using the picture below.
\begin{equation}
\label{eqn:pipecrossred}
\raisebox{-2.75em}{\begin{tikzpicture}[x=1.5em,y=1.5em]
	\draw[thick,rounded corners, color=blue] (-1,2)--(1,2);
	\draw[thick,rounded corners, color=blue] (0,1)--(0,3);
\draw (0,.5) node [align=left]{$j$};
\draw (1.5,2) node [align=left]{$i$};
\draw (-1.5,2) node [align=left]{$i$};
\draw (0,3.5) node [align=left]{$j$};
\end{tikzpicture}}
\end{equation}
  Then $w_{\mathcal P}(i)$ is the final label of the path which ends in the $i$th row.

Additionally, we may consider the Demazure product of $\mathbf a_{\mathcal P}$. For brevity, we write  $\demprod{\mathcal P}:=\demprod{ \mathbf a_{\mathcal P}}$.
We may compute $\demprod{\mathcal P}$ graphically in a similar way as before.  However, at crossings, if $i<j$, the labels change as pictured below.
\begin{equation}
\label{eqn:pipecross}
\raisebox{-2.75em}{\begin{tikzpicture}[x=1.5em,y=1.5em]
	\draw[thick,rounded corners, color=blue] (-1,2)--(1,2);
	\draw[thick,rounded corners, color=blue] (0,1)--(0,3);
\draw (0,.5) node [align=left]{$j$};
\draw (1.5,2) node [align=left]{$i$};
\draw (-1.5,2) node [align=left]{$i$};
\draw (0,3.5) node [align=left]{$j$};
\end{tikzpicture}
\hspace{4em}
\begin{tikzpicture}[x=1.5em,y=1.5em]
	\draw[thick,rounded corners, color=blue] (-1,2)--(1,2);
	\draw[thick,rounded corners, color=blue] (0,1)--(0,3);
\draw (0,.5) node [align=left]{$i$};
\draw (1.5,2) node [align=left]{$i$};
\draw (-1.5,2) node [align=left]{$j$};
\draw (0,3.5) node [align=left]{$j$};
\end{tikzpicture}}
\end{equation}
Then $\demprod{\mathcal P}(i)$ is the label of the path in the $i$th row at the right edge of the rectangle.  Notice that the second configuration in (\ref{eqn:pipecross}) occurs if and only if paths $i$ and $j$ have previously crossed, i.e.,\ the planar history is not reduced.  In particular, $\mathcal P$ is reduced if and only if $\demprod{\mathcal P}=w_\mathcal P$.

\begin{lemma}
\label{lemma:planarmoves}
Applying any of the local moves pictured below to $\mathcal P$ produces a new planar history $\mathcal P'$ so that $w_{\mathcal P}=w_{\mathcal P'}$ and $\demprod{\mathcal P}=\demprod{\mathcal P'}$.
\begin{equation}
\label{eqn:planarmoves1}
\raisebox{-1.25em}{\begin{tikzpicture}[x=.75em,y=.75em]
\draw[thick] (2,0)--(2,4);
\draw[thick,rounded corners, color=blue] (1,1)--(1,3)--(3,3);
\end{tikzpicture}
\hspace{1em} \raisebox{1.3em}{$\leftrightarrow$} \hspace{1em}
\begin{tikzpicture}[x=.75em,y=.75em]
\draw[thick] (2,0)--(2,4);
\draw[thick,rounded corners, color=blue] (1,1)--(3,1)--(3,3);
\end{tikzpicture}
\hspace{2em}
\raisebox{.7em}{
\begin{tikzpicture}[x=.75em,y=.75em]
\draw[thick] (0,2)--(4,2);
\draw[thick,rounded corners, color=blue] (1,1)--(1,3)--(3,3);
\end{tikzpicture}
\hspace{1em} \raisebox{.6em}{$\leftrightarrow$} \hspace{1em}
\begin{tikzpicture}[x=.75em,y=.75em]
\draw[thick] (0,2)--(4,2);
\draw[thick,rounded corners, color=blue] (1,1)--(3,1)--(3,3);
\end{tikzpicture}}
\hspace{2em}
\begin{tikzpicture}[x=.75em,y=.75em]
\draw[thick] (2,0)--(2,4);
\draw[thick] (0,2)--(4,2);
\draw[thick,rounded corners, color=blue] (1,1)--(1,3)--(3,3);
\end{tikzpicture}
\hspace{1em} \raisebox{1.3em}{$\leftrightarrow$} \hspace{1em}
\begin{tikzpicture}[x=.75em,y=.75em]
\draw[thick] (2,0)--(2,4);
\draw[thick] (0,2)--(4,2);
\draw[thick,rounded corners, color=blue] (1,1)--(3,1)--(3,3);
\end{tikzpicture}}
\end{equation}
Furthermore, if $\mathcal P'$ was obtained from $\mathcal P$ by one of the moves below, then $\demprod{\mathcal P}=\demprod{\mathcal P'}$.
\begin{equation}
\label{eqn:planarmoves2}
\raisebox{-1.25em}{\begin{tikzpicture}[x=.75em,y=.75em]
\draw[thick,rounded corners, color=blue] (0,1)--(1,1)--(1,3)--(4,3);
\draw[thick,rounded corners, color=blue] (1,0)--(1,1)--(3,1)--(3,4);
\end{tikzpicture}
\hspace{1em}
\raisebox{1.3em}{$\leftrightarrow$} \hspace{1em}
\begin{tikzpicture}[x=.75em,y=.75em]
\draw[thick,rounded corners, color=blue] (0,1)--(3,1)--(3,4);
\draw[thick,rounded corners, color=blue] (1,0)--(1,3)--(4,3);
\end{tikzpicture}
\hspace{1em} \raisebox{1.3em}{$\leftrightarrow$} \hspace{1em}
\begin{tikzpicture}[x=.75em,y=.75em]
\draw[thick,rounded corners, color=blue] (1,0)--(1,3)--(3,3)--(3,4);
\draw[thick,rounded corners, color=blue] (0,1)--(3,1)--(3,3)--(4,3);
\end{tikzpicture}}
\end{equation}
\end{lemma}
\begin{proof}
This is an immediate application of the labeling schemes from (\ref{eqn:pipecrossred}) and (\ref{eqn:pipecross}).  For the replacements 
\[\begin{tikzpicture}[x=.75em,y=.75em]
\draw[thick] (2,0)--(2,4);
\draw[thick] (0,2)--(4,2);
\draw[thick,rounded corners, color=blue] (1,1)--(1,3)--(3,3);
\end{tikzpicture}
\hspace{1em} \raisebox{1.3em}{$\leftrightarrow$} \hspace{1em}
\begin{tikzpicture}[x=.75em,y=.75em]
\draw[thick] (2,0)--(2,4);
\draw[thick] (0,2)--(4,2);
\draw[thick,rounded corners, color=blue] (1,1)--(3,1)--(3,3);
\end{tikzpicture}\] 
there are six cases to check when computing $\demprod{\mathcal P}$.  The verification is straightforward, so we leave this task to the reader.
\end{proof}

The \mydef{Rothe} BPD for $w$ is the (unique) BPD which has downward elbow tiles in positions $(i,w(i))$ for all $i\in[n]$ and no upward elbow tiles.  In other words, all of its pipes are hooks which bend exactly once.

\begin{lemma}
\label{lemma:permword}
If $\mathcal P$ is the Rothe BPD for $w$, then $\mathbf a_\mathcal P$ is  a reduced word and $\demprod{\mathcal P}=w_{\mathcal P}=w$.
\end{lemma}
\begin{proof}

By assumption, all pipes of $\mathcal P$ are hooks.  As such, pairwise, they cross at most once. Therefore, $\mathcal P$ is reduced (and so $\demprod{\mathcal P}=w_{\mathcal P}$).
 Starting from the bottom edge of the grid, the pipe labeled $i$ travels upwards and then turns right in cell $(w^{-1}(i),i)$.  Finally it hits the right edge of the grid in row $w^{-1}(i)$.  Therefore, $\demprod{\mathcal P}(w^{-1}(i))=i$ for all $i\in[n]$ which implies $\demprod{\mathcal P}=w$.
\end{proof}

\subsection{$\beta$-double Grothendieck polynomials}
\label{subsec:betagroth}

We now recall the $\beta$-double Grothendieck polynomials of \cite{Fomin.Kirillov}.
Let $R=\mathbb Z[\beta][y_1,\ldots,y_n]$. We write  $\mathbb Z[\beta][\mathbf x;\mathbf y]:=R[x_1,\ldots,x_n]$.
The symmetric group $\SymGp_n$ acts on $\mathbb Z[\beta][\mathbf x;\mathbf y]$ by \[w\cdot f=f(x_{w_1},\,x_{w_2},\,\ldots,\,x_{w_n};\mathbf y).\] We define an operator $\pi_i$ which acts on $\mathbb Z[\beta][\mathbf x;\mathbf y]$ by  
\[\pi_i(f)=\frac{(1+\beta x_{i+1})f-(1+\beta x_i)s_i\cdot f}{x_i-x_{i+1}}.\]
The operators $\pi_i$ satisfy the same braid and commutation relations as the simple reflections in $\SymGp_n$,  as well as the following Leibniz rule:
\begin{equation}
\label{eqn:leibniz}
\pi_i(fg)=\pi_i(f)g+(s_i\cdot f)(\pi_i(g)+\beta g).
\end{equation}
Notice that $\pi_i(f)$ is symmetric in $x_i$ and $x_{i+1}$.  Furthermore, for any $g$, if $f=s_i\cdot f$ then $\pi_i(fg)=f\pi_i(g)$;  setting $g=1$ yields $\pi_i(f)=-\beta f$. Thus, $\pi_i^2=-\beta\pi_i$.

The \mydef{$\beta$-double Grothendieck polynomials} are defined as follows.  For the longest permutation, we set \[\Groth^{(\beta)}_{w_0}(\mathbf x;\mathbf y)=\prod_{1<i+j\leq n}(x_i\oplus y_j),\] where $x_i\oplus y_j=x_i+y_j+\beta x_iy_j$. If $w\in \SymGp_n$ with $w_i>w_{i+1}$,  define 
\begin{equation}
\label{eq:betagrothdef}
\Groth^{(\beta)}_{ws_i}(\mathbf x;\mathbf y)=\pi_i(\Groth^{(\beta)}_{w}(\mathbf x;\mathbf y)).
\end{equation}
Notice if $w_i<w_{i+1}$ then $\Groth^{(\beta)}_w(\mathbf x;\mathbf y)$ is symmetric in $x_i$ and $x_{i+1}$.  Therefore,
\begin{equation}
\label{eqn:dividedbetaterm}
\pi_i(\Groth^{(\beta)}_w(\mathbf x;\mathbf y))=-\beta\Groth^{(\beta)}_w(\mathbf x;\mathbf y).
\end{equation}

Specializing $\beta=-1$, in $\Groth^{(\beta)}_w(\mathbf x;\mathbf y)$ yields the \mydef{double Grothendieck polynomials}.  We may further specialize each $y_i$ to $0$ to obtain the (single) \mydef{Grothendieck polynomial}.
We get the \mydef{double Schubert polynomials} from $\beta$-double Grothendieck polynomials by setting $\beta$ to $0$ and replacing each $y_i$ with $-y_i$, i.e.,\ $\Schub_w(\mathbf x;\mathbf y):=\Groth_w^{(0)}(\mathbf x;-\mathbf y)$.  Likewise, the (single) \mydef{Schubert polynomials} are obtained by setting $\beta$ and each of the $y_i$'s equal to $0$.

\subsection{Transition equations}
\label{section:introtransition}
In this section, we recall Lascoux's transition equations for double Grothendieck polynomials.  Note that the statement of transition in \cite{Lascoux:ice} differs from our conventions here by a change of variables.

 Say $(i,j)$ is a \mydef{pivot} of $(a,b)$ in $w$ if:
\begin{enumerate}
\item $w(i)=j$,
\item $i<a$ and $j<b$, and
\item if $(i',j')\in [i,a]\times [j,b]-\{(i,j),(a,b)\}$ then $w(i')\neq j'$.
\end{enumerate}

 The permutation $w$ has a \mydef{descent} in position $i$ if $w_i>w_{i+1}$.  Write \[\des(w)=\max(\{0\}\cup \{i:w_i>w_{i+1}\}).\]  Notice $\des(w)=0$ if and only if $w=\id$.

Fix a permutation $w\in \SymGp_n$ such that $w\neq \id$. Let $a=\des(w)$ and set \[b=\max\{j:(a,j)\in D(w)\},\] i.e.,\ $(a,b)$ is the rightmost cell in the last row of $D(w)$. We call $(a,b)$ the \mydef{maximal corner} of $w$ and denote it $\mc(w)$.  For the identity, the maximal corner is undefined.
 Let \[\phi(w)=\{i:(i,j) \enspace \text{is a pivot of} \enspace \mc(w) \enspace \text{in} \enspace w\}.\]
If $I=\{i_1,\ldots, i_k\}$ with $1\leq i_1<i_2<\cdots<i_k<a$, write $c^{(a)}_{I}$ for the cycle $(a\, i_k\, i_{k-1}\, \ldots\, i_1)$.
Fix $I\subseteq \phi(w)$ and  write $b'=w^{-1}(b)$. We define \[w_I=w t_{a \,b'} c^{(a)}_{I}.\] 
Notice $\ell(w_I)=\ell(w)+|I|-1$ (see Lemma~\ref{lemma:transitiondeflation}).

\begin{theorem}
\label{thm:transition}
Keeping the above notation,
\[\Groth^{(\beta)}_w(\mathbf x;\mathbf y)=(x_a\oplus y_b)\Groth^{(\beta)}_{w_\emptyset}(\mathbf x;\mathbf y)+(1+\beta( x_a\oplus y_b))\sum_{I\subseteq \phi(w):I\neq\emptyset}\beta^{|I|-1}\Groth^{(\beta)}_{w_I}(\mathbf x;\mathbf y).\]
\end{theorem}

By making the appropriate specializations of Theorem~\ref{thm:transition},
one immediately recovers transition formulas for Schubert and Grothendieck polynomials.  We state the version for double Schubert polynomials below.

\begin{corollary} 
\label{cor:transtiondoubleschubert}
Given $w\in \SymGp_{n}$ with $\mc(w)=(a,b)$,
\[ \Schub_w(\mathbf x;\mathbf y)=(x_a-y_b)\Schub_{w_\emptyset}+\sum_{i\in \phi(w)}\Schub_{w_{\{i\}}}(\mathbf x;\mathbf y).\]
\end{corollary}

The statement of Theorem~\ref{thm:transition} is well known to experts (especially the case $\beta=-1$).   For completeness, we include a proof in  Appendix~\ref{appendix:transitionproof}. The proof for single Grothendieck polynomials can be found in \cite{Lascoux:transition} and \cite[Corollary~3.10]{Lenart:Monk}. 
Corollary~\ref{cor:transtiondoubleschubert} can be obtained from the double version of Monk's rule in \cite[Proposition~4.1]{Kohnert.Veigneau}.  See \cite{Knutson.Yong} for a graphical description of transition in terms of Rothe diagrams.

\subsection{Partitions and tableaux}
\label{section:partandtab}

We will frequently refer to objects positioned in an $n\times n$ grid or the infinite grid $\mathbb P\times \mathbb P$.  When represented graphically, the cell $(1,1)$ is positioned in the northwest corner of the grid. We write $(i,j)$ for the cell which is in the $i$th row from the top and the $j$th column from the left.

Recall, a partition is a weakly decreasing tuple of nonnegative integers $\lambda=(\lambda_1,\lambda_2,\ldots,\lambda_k)$.  We will often conflate $\lambda$ with the subset of cells in the $\mathbb P\times \mathbb P$ grid \[\yd{\lambda}=\{(i,j): i\in [k] \enspace\text{and}\enspace j\in[\lambda_i]\}.\]  Write $|\lambda|=|\yd{\lambda}|$.

We now discuss tableaux.  We refer the reader to \cite{Fulton.YT} for a reference.
A \mydef{tableau of shape $\lambda$} is a filling of $\yd{\lambda}$ with elements of $\mathbb P$, i.e.,\ a function $T:\yd{\lambda}\rightarrow \mathbb P$. We say $T$ is \mydef{semistandard} if its entries weakly increase along rows (reading left to right) and strictly increase along columns (reading top to bottom).  Explicitly, 
\begin{equation}
\label{eq:ssdef}
T(i,j)\leq T(i,j+1) \enspace \text{and} \enspace T(i,j)<T(i+1,j)
\end{equation}
 whenever both sides of these inequalities are defined. 
Given a tuple $\mathbf f=(f_1,\,\ldots,\,f_k)$, we say $T$ is \mydef{flagged by $\mathbf f$} if $T(i,j)\leq f_i$ for all $(i,j)\in\yd{\lambda}$.  Write $\FSSYT(\lambda,\mathbf f)$ for the set of semistandard tableaux of shape $\lambda$ which are flagged by $\mathbf f$.

A tableau is \mydef{decreasing} if its entries strictly decrease along rows and columns.  Write $\DT(n)$ for the set of decreasing tableaux which have entries in $[n]$.
Let $\DT(\lambda,n)$ be the set decreasing tableaux of shape $\lambda$.
  Given $T\in \DT(n)$, we form its {\bf column reading word} by reading its entries along columns from bottom to top, working from left to right (see Figure~\ref{figure:hecketab}).  Say $T\in\DT(n)$ is a {\bf reduced word tableau} if its column reading word is reduced.  Write $\RWT(n)$ for set the reduced word tableaux in $\DT(n)$.  Likewise, we write $\RWT(\lambda,n)$ for the reduced word tableaux of shape $\lambda$ with entries in $[n]$.

 A \mydef{set-valued} tableau $\mathbf T$ is a filling of $\yd{\lambda}$ with nonempty, finite subsets of $\mathbb P$.  Write \[|\mathbf T|=\sum_{(i,j)\in \yd{\lambda}}|\mathbf T(i,j)|\] for the total number of entries in $\mathbf T$.

  We may produce a tableau $T$ from a set-valued tableau $\mathbf T$ by selecting a single entry from each cell.
In particular, define $\flatten(\mathbf T)$ by $\flatten(\mathbf T)(i,j)=\min(\mathbf T(i,j))$ for all $(i,j)\in \yd{\lambda}$.  Equivalently, this is the map from set-valued tableaux to ordinary tableaux which forgets all but the smallest entry in each cell.

  We apply the adjective semistandard  to $\mathbf T$ if each (ordinary) tableau which can obtained by restricting entries in $\mathbf T$ is itself semistandard.  
In the same way, the definition of a flagged tableau extends  to the setting of set-valued tableaux.
Let $\FSet(\lambda,\mathbf f)$ be the set-valued semistandard tableaux of shape $\lambda$ which are flagged by $\mathbf f$.

\section{Bumpless pipe dreams through ice}
\label{section:ice}

 As previously discussed, bumpless pipe dreams are in transparent bijection with osculating lattice paths.  These in turn, are in bijection with numerous objects from statistical mechanics.  See \cite{Propp} for a survey on alternating sign matrices and related objects.
In this section, we review the bijections between bumpless pipe dreams, alternating sign matrices,  square ice configurations, and corner sum matrices.  We also collect related facts.

\subsection{Alternating sign matrices  and Rothe diagrams}

An \mydef{alternating sign matrix} (ASM) is a square matrix with entries in $\{-1,0,1\}$ so that within each row and  column, the nonzero entries alternate in sign and sum to 1.  Write $\asm(n)$ for the set of ASMs of size $n$.

An ASM with no negative entries is called a {\bf permutation matrix}.  The subset of permutation matrices may be identified with $\SymGp_n$ via the map
$\SymGp_n\hookrightarrow {\sf ASM}(n)$ which takes $w\in \SymGp_n$ to the matrix which has 1's in positions $(i,w_i)$ for each $i\in [n]$ and zeros elsewhere.

We often represent $A\in \asm(n)$  in an $n\times n$ grid by plotting
\begin{enumerate}
\item a filled dot $\bullet$ in row $i$ and column $j$ if $A_{i\,j}=1$ and
\item an open dot $\circ$ in row $i$ and column $j$ if $A_{i\,j}=-1$.
\end{enumerate}  
Call this the \mydef{graph} of $A$.

 We now describe the \emph{Rothe diagram} of an ASM.  Start from the graph of  $A\in \asm(n)$.  For each filled dot, send out a line segment to its right and below it.  These segments terminate when they reach an open dot or the edge of the grid. Call these  {\bf defining line segments} for $A$.
The set of the coordinates of cells which do not contain any defining line segments  is the {\bf Rothe diagram}\footnote{In \cite{weigandt2017prism}, the author studied Rothe diagrams of ASMs using a different convention.  There, \emph{negative inversions} were included in the diagram.  This alternative definition allows for a natural generalization of Fulton's \emph{essential set} to ASMs.} of $A$, denoted $D(A)$.  See Figure~\ref{figure:rothe} for an example.  
Let \[N(A)=\{(i,j):A_{i\,j}=-1\}.\] 
The set $D(A)$ corresponds to the \emph{positive inversions} of $A$ as defined in \cite{robbins1986determinants} while $N(A)$ is the set of \emph{negative inversions}.
\begin{figure}	
	\begin{center}
		\raisebox{1\height}{$\left(\begin{array}{cccc}
			0&0&1&0\\
			0&1&0&0\\
			1&0&-1&1\\
			0&0&1&0
			\end{array}\right)$} \hspace {4em}
		\begin{tikzpicture}[x=1.5em,y=1.5em]
		\draw[step=1,gray, thin] (0,1) grid (4,5);
		\draw[color=black, thick](0,1)rectangle(4,5);
		\draw[thick] (.5,1)--(.5,2.5)--(2.5,2.5)--(2.5,4.5)--(4,4.5);
		\draw[thick] (1.5,1)--(1.5,3.5)--(4,3.5);
		\draw[thick] (2.5,1)--(2.5,1.5)--(4,1.5);
		\draw[thick] (3.5,1)--(3.5,2.5)--(4,2.5);
		\filldraw[color=black, fill=gray!40, thick](0,4)rectangle(1,5);
		\filldraw[color=black, fill=gray!40, thick](1,4)rectangle(2,5);
		\filldraw[color=black, fill=gray!40, thick](0,3)rectangle(1,4);
		\filldraw [black](1.5,3.5)circle(.1);
		\filldraw [black](3.5,2.5)circle(.1);
		\filldraw [color=black,fill=white,thick](2.5,2.5)circle(.1);
		\filldraw [black](2.5,4.5)circle(.1);
		\filldraw [black](.5,2.5)circle(.1);
		\filldraw [black](2.5,1.5)circle(.1);
		\end{tikzpicture}
	\end{center}
\caption{Pictured to the left is $A\in \asm(n)$.  On the right is the graph of $A$ pictured with its defining line segments.  The cells in $D(A)=\{(1,1),(1,2),(2,1)\}$ have been shaded gray for emphasis.}
\label{figure:rothe}
\end{figure}
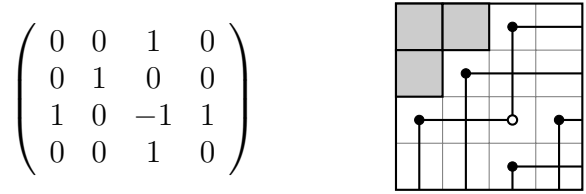

We define a map $\asmtobpd:{\sf ASM}(n)\rightarrow {\sf BPD}(n)$ by ``smoothing'' out bends in the defining segments, i.e.,\ by making the replacements pictured below.
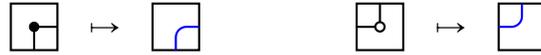
\[\begin{tikzpicture}[x=1.5em,y=1.5em]
\draw[color=black, thick](0,1)rectangle(1,2);
\draw[thick] (.5,1)--(.5,1.5)--(1,1.5);
\filldraw [black](.5,1.5)circle(.1);
\end{tikzpicture} 
\quad
\raisebox{.5em}{$\mapsto$}
\quad	
\begin{tikzpicture}[x=1.5em,y=1.5em]
\draw[color=black, thick](0,1)rectangle(1,2);
\draw[thick,rounded corners, color=blue] (.5,1)--(.5,1.5)--(1,1.5);
\end{tikzpicture} 
\hspace{5em} 
\begin{tikzpicture}[x=1.5em,y=1.5em]
\draw[color=black, thick](0,1)rectangle(1,2);
\draw[thick] (.5,2)--(.5,1.5)--(0,1.5);
\filldraw [color=black,fill=white,thick](.5,1.5)circle(.1);
\end{tikzpicture}
\quad
\raisebox{.5em}{$\mapsto$}
\quad
\begin{tikzpicture}[x=1.5em,y=1.5em]
\draw[color=black, thick](0,1)rectangle(1,2);
\draw[thick,rounded corners, color=blue] (.5,2)--(.5,1.5)--(0,1.5);
\end{tikzpicture}\]

\begin{lemma}
	\label{lemma:bijection}
	The map $\asmtobpd:{\sf ASM}(n)\rightarrow {\sf BPD}(n)$ is a bijection.
\end{lemma}
\begin{proof}
Bumpless pipe dreams are in transparent bijection with osculating lattice paths, which are in turn in bijection with ASMs (see \cite[Section~4]{Behrend}).  To change a BPD to an osculating path, simply replace each crossing tile with a bumping tile.  Using this correspondence, Lemma~\ref{lemma:bijection} follows.  
\end{proof}

In light of this bijection, we say $A\in {\sf ASM}(n)$ is {\bf reduced} if $\Phi(A)$ is a  reduced BPD.  Furthermore, notice $U(\asmtobpd(A))=N(A)$ and $D(\asmtobpd(A))=D(A)$.  Thus, our weights on BDPs are compatible with the weights on ASMs found in \cite{Lascoux:ice}, up to changing conventions.

\subsection{Corner sums}

The {\bf corner sum function} of $A\in {\sf ASM}(n)$ is 
\begin{equation}
r_A(i,j):=\sum_{a=1}^i\sum_{b=1}^j A_{a\,b}.
\end{equation}
By convention, we define $r_A(i,j)=0$ whenever $i=0$ or $j=0$.
Write \[{\sf R}(n):=\{r_A:A\in{\sf ASM}(n)\}.\]
\begin{lemma}[{\cite[Lemma~1]{robbins1986determinants}}]
	\label{lemma:cornersumchar}
	Corner sums of $n\times n$ ASMs are  characterized by the following properties.
	\begin{enumerate}
		\item $r_A(i,n)=r_A(n,i)=i$ for $i\in[n]$ and
		\item $r_A(i,j)-r_A(i-1,j)$ and $r_A(i,j)-r_A(i,j-1)$ are $0$ or $1$ for all $i,j\in [n]$.
	\end{enumerate}
\end{lemma}

The map $A\mapsto r_A$ places ${\sf ASM}(n)$ and ${\sf R}(n)$  in bijection.  By \cite{robbins1986determinants}, the explicit inverse map is obtained by setting 
\begin{equation}
\label{eqn:cornersuminverse}
A_{i\,j}=r_A(i,j)+r_A(i-1,j-1)-r_A(i-1,j)-r_A(i,j-1).
\end{equation}

Corner sum matrices have a natural poset structure defined by entry wise comparison. Say $r_A\leq r_B$ if $r_A(i,j)\leq r_B(i,j)$ for all $1\leq i,j\leq n$.  
This induces a poset structure on ASMs by declaring \[A\leq B \enspace \text{if and only if} \enspace r_A\geq r_B.\]

 The restriction of the poset $\asm(n)$ to $\SymGp_n$ produces the (strong) Bruhat order on the symmetric group.  Indeed, ${\sf ASM}(n)$ is the smallest lattice with this property \cite{lascoux1996treillis}.

We  will also need notation for partial row and column sums of $A$.  Write 
\begin{equation}
{\rm row}_{A,i}(j)=\sum_{k=1}^j A_{i\,k} \enspace \text{ and } \enspace {\rm col}_{A,j}(i)=\sum_{k=1}^i A_{k\,j}.
\end{equation}
Notice that
\begin{equation}
\label{eqn:cornersumbyrowandcol}
r_A(i,j)=\sum_{a=1}^i {\rm row}_{A,a}(j)=\sum_{b=1}^j{\rm col}_{A,b}(i).
\end{equation}

We call $\{(i,j)\in D(A):r_A(i,j)=0\}$ the \mydef{dominant} part of $D(A)$ (and $D(\asmtobpd(A))$).  Notice by Lemma~\ref{lemma:cornersumchar}, the dominant part of the diagram is always top-left justified and forms a partition shape.   

\subsection{Square ice}
We now recall the bijection between ASMs and square ice configurations.  An {\bf ice model} is an orientation of the edges of the square lattice (or a subset thereof) so that at each 4-valent vertex, two edges point inwards and two point outwards.  At each 4-valent vertex, there are  six  possible configurations,  pictured below. 
\begin{equation}
\label{eqn:sixice}
\begin{tikzpicture}[x=1.5em,y=1.5em]
\draw[thick,decoration={markings, mark=at position 0.6 with {\arrow[scale=1]{>}}}, postaction={decorate}] (-1,0)--(0,0);
\draw[thick,decoration={markings, mark=at position 0.6 with {\arrow[scale=1]{>}}}, postaction={decorate}]  (1,0)--(0,0);
\draw[thick,decoration={markings, mark=at position 0.6 with {\arrow[scale=1]{<}}}, postaction={decorate}]  (0,-1)--(0,0);
\draw[thick,decoration={markings, mark=at position 0.6 with {\arrow[scale=1]{<}}}, postaction={decorate}]  (0,1)--(0,0);
\filldraw [black](-1,0)circle(.1);
\filldraw [black](1,0)circle(.1);
\filldraw [black](0,-1)circle(.1);
\filldraw [black](0,1)circle(.1);
\filldraw [black](0,0)circle(.1);
\end{tikzpicture}
\hspace{2em}
\begin{tikzpicture}[x=1.5em,y=1.5em]
\draw[thick,decoration={markings, mark=at position 0.6 with {\arrow[scale=1]{<}}}, postaction={decorate}] (-1,0)--(0,0);
\draw[thick,decoration={markings, mark=at position 0.6 with {\arrow[scale=1]{<}}}, postaction={decorate}]  (1,0)--(0,0);
\draw[thick,decoration={markings, mark=at position 0.6 with {\arrow[scale=1]{>}}}, postaction={decorate}]  (0,-1)--(0,0);
\draw[thick,decoration={markings, mark=at position 0.6 with {\arrow[scale=1]{>}}}, postaction={decorate}]  (0,1)--(0,0);
\filldraw [black](-1,0)circle(.1);
\filldraw [black](1,0)circle(.1);
\filldraw [black](0,-1)circle(.1);
\filldraw [black](0,1)circle(.1);
\filldraw [black](0,0)circle(.1);
\end{tikzpicture}
\hspace{2em}
\begin{tikzpicture}[x=1.5em,y=1.5em]
\draw[thick,decoration={markings, mark=at position 0.6 with {\arrow[scale=1]{<}}}, postaction={decorate}] (-1,0)--(0,0);
\draw[thick,decoration={markings, mark=at position 0.6 with {\arrow[scale=1]{>}}}, postaction={decorate}]  (1,0)--(0,0);
\draw[thick,decoration={markings, mark=at position 0.6 with {\arrow[scale=1]{<}}}, postaction={decorate}]  (0,-1)--(0,0);
\draw[thick,decoration={markings, mark=at position 0.6 with {\arrow[scale=1]{>}}}, postaction={decorate}]  (0,1)--(0,0);
\filldraw [black](-1,0)circle(.1);
\filldraw [black](1,0)circle(.1);
\filldraw [black](0,-1)circle(.1);
\filldraw [black](0,1)circle(.1);
\filldraw [black](0,0)circle(.1);
\end{tikzpicture}
\hspace{2em}
\begin{tikzpicture}[x=1.5em,y=1.5em]
\draw[thick,decoration={markings, mark=at position 0.6 with {\arrow[scale=1]{>}}}, postaction={decorate}] (-1,0)--(0,0);
\draw[thick,decoration={markings, mark=at position 0.6 with {\arrow[scale=1]{<}}}, postaction={decorate}]  (1,0)--(0,0);
\draw[thick,decoration={markings, mark=at position 0.6 with {\arrow[scale=1]{>}}}, postaction={decorate}]  (0,-1)--(0,0);
\draw[thick,decoration={markings, mark=at position 0.6 with {\arrow[scale=1]{<}}}, postaction={decorate}]  (0,1)--(0,0);
\filldraw [black](-1,0)circle(.1);
\filldraw [black](1,0)circle(.1);
\filldraw [black](0,-1)circle(.1);
\filldraw [black](0,1)circle(.1);
\filldraw [black](0,0)circle(.1);
\end{tikzpicture}
\hspace{2em}
\begin{tikzpicture}[x=1.5em,y=1.5em]
\draw[thick,decoration={markings, mark=at position 0.6 with {\arrow[scale=1]{<}}}, postaction={decorate}] (-1,0)--(0,0);
\draw[thick,decoration={markings, mark=at position 0.6 with {\arrow[scale=1]{>}}}, postaction={decorate}]  (1,0)--(0,0);
\draw[thick,decoration={markings, mark=at position 0.6 with {\arrow[scale=1]{>}}}, postaction={decorate}]  (0,-1)--(0,0);
\draw[thick,decoration={markings, mark=at position 0.6 with {\arrow[scale=1]{<}}}, postaction={decorate}]  (0,1)--(0,0);
\filldraw [black](-1,0)circle(.1);
\filldraw [black](1,0)circle(.1);
\filldraw [black](0,-1)circle(.1);
\filldraw [black](0,1)circle(.1);
\filldraw [black](0,0)circle(.1);
\end{tikzpicture}
\hspace{2em}
\begin{tikzpicture}[x=1.5em,y=1.5em]
\draw[thick,decoration={markings, mark=at position 0.6 with {\arrow[scale=1]{>}}}, postaction={decorate}] (-1,0)--(0,0);
\draw[thick,decoration={markings, mark=at position 0.6 with {\arrow[scale=1]{<}}}, postaction={decorate}]  (1,0)--(0,0);
\draw[thick,decoration={markings, mark=at position 0.6 with {\arrow[scale=1]{<}}}, postaction={decorate}]  (0,-1)--(0,0);
\draw[thick,decoration={markings, mark=at position 0.6 with {\arrow[scale=1]{>}}}, postaction={decorate}]  (0,1)--(0,0);
\filldraw [black](-1,0)circle(.1);
\filldraw [black](1,0)circle(.1);
\filldraw [black](0,-1)circle(.1);
\filldraw [black](0,1)circle(.1);
\filldraw [black](0,0)circle(.1);
\end{tikzpicture}
\end{equation}

As our underlying graph, take an $(n+2)\times(n+2)$ square subset of the  lattice.  From this, remove the outer-most edges and the four corner vertices.  Call this a {\bf framed square grid}.  The 1-valent vertices are {\bf boundary vertices} and the edges which are adjacent to them are {\bf boundary edges}.  We refer to the other edges (and vertices) as being {\bf interior edges}  (and \mydef{vertices}).  
A {\bf square ice configuration} (of size $n$)  is an orientation of this graph so that
\begin{enumerate}
	\item horizontal boundary edges point inwards,
	\item vertical boundary edges point outwards, and
	\item at every interior vertex, two edges point in and two edges point out.
\end{enumerate}
Write ${\sf Ice}(n)$ for the set of square ice configurations of size $n$.  
To map from a square ice configuration to an ASM,  replace each interior vertex with an element from $\{-1,0,1\}$ as indicated below.  
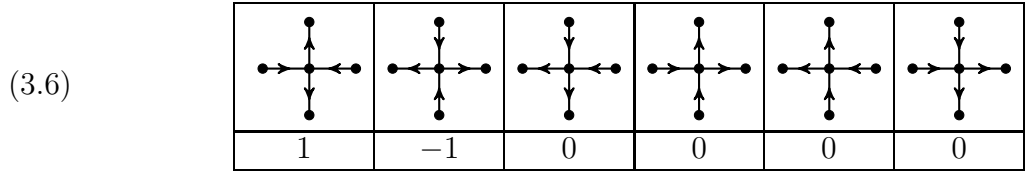
\begin{equation}
\label{eqn:IceASM}
\begin{tabular}{|c|c|c|c|c|c|} \hline
\rule[.75em]{0pt}{3em}

\begin{tikzpicture}[x=1.5em,y=1.5em]
\draw[thick,decoration={markings, mark=at position 0.6 with {\arrow[scale=1]{>}}}, postaction={decorate}] (-1,0)--(0,0);
\draw[thick,decoration={markings, mark=at position 0.6 with {\arrow[scale=1]{>}}}, postaction={decorate}]  (1,0)--(0,0);
\draw[thick,decoration={markings, mark=at position 0.6 with {\arrow[scale=1]{<}}}, postaction={decorate}]  (0,-1)--(0,0);
\draw[thick,decoration={markings, mark=at position 0.6 with {\arrow[scale=1]{<}}}, postaction={decorate}]  (0,1)--(0,0);
\filldraw [black](-1,0)circle(.1);
\filldraw [black](1,0)circle(.1);
\filldraw [black](0,-1)circle(.1);
\filldraw [black](0,1)circle(.1);
\filldraw [black](0,0)circle(.1);
\end{tikzpicture}
&
\begin{tikzpicture}[x=1.5em,y=1.5em]
\draw[thick,decoration={markings, mark=at position 0.6 with {\arrow[scale=1]{<}}}, postaction={decorate}] (-1,0)--(0,0);
\draw[thick,decoration={markings, mark=at position 0.6 with {\arrow[scale=1]{<}}}, postaction={decorate}]  (1,0)--(0,0);
\draw[thick,decoration={markings, mark=at position 0.6 with {\arrow[scale=1]{>}}}, postaction={decorate}]  (0,-1)--(0,0);
\draw[thick,decoration={markings, mark=at position 0.6 with {\arrow[scale=1]{>}}}, postaction={decorate}]  (0,1)--(0,0);
\filldraw [black](-1,0)circle(.1);
\filldraw [black](1,0)circle(.1);
\filldraw [black](0,-1)circle(.1);
\filldraw [black](0,1)circle(.1);
\filldraw [black](0,0)circle(.1);
\end{tikzpicture}
&
\begin{tikzpicture}[x=1.5em,y=1.5em]
\draw[thick,decoration={markings, mark=at position 0.6 with {\arrow[scale=1]{<}}}, postaction={decorate}] (-1,0)--(0,0);
\draw[thick,decoration={markings, mark=at position 0.6 with {\arrow[scale=1]{>}}}, postaction={decorate}]  (1,0)--(0,0);
\draw[thick,decoration={markings, mark=at position 0.6 with {\arrow[scale=1]{<}}}, postaction={decorate}]  (0,-1)--(0,0);
\draw[thick,decoration={markings, mark=at position 0.6 with {\arrow[scale=1]{>}}}, postaction={decorate}]  (0,1)--(0,0);
\filldraw [black](-1,0)circle(.1);
\filldraw [black](1,0)circle(.1);
\filldraw [black](0,-1)circle(.1);
\filldraw [black](0,1)circle(.1);
\filldraw [black](0,0)circle(.1);
\end{tikzpicture}
&
\begin{tikzpicture}[x=1.5em,y=1.5em]
\draw[thick,decoration={markings, mark=at position 0.6 with {\arrow[scale=1]{>}}}, postaction={decorate}] (-1,0)--(0,0);
\draw[thick,decoration={markings, mark=at position 0.6 with {\arrow[scale=1]{<}}}, postaction={decorate}]  (1,0)--(0,0);
\draw[thick,decoration={markings, mark=at position 0.6 with {\arrow[scale=1]{>}}}, postaction={decorate}]  (0,-1)--(0,0);
\draw[thick,decoration={markings, mark=at position 0.6 with {\arrow[scale=1]{<}}}, postaction={decorate}]  (0,1)--(0,0);
\filldraw [black](-1,0)circle(.1);
\filldraw [black](1,0)circle(.1);
\filldraw [black](0,-1)circle(.1);
\filldraw [black](0,1)circle(.1);
\filldraw [black](0,0)circle(.1);
\end{tikzpicture}
&
\begin{tikzpicture}[x=1.5em,y=1.5em]
\draw[thick,decoration={markings, mark=at position 0.6 with {\arrow[scale=1]{<}}}, postaction={decorate}] (-1,0)--(0,0);
\draw[thick,decoration={markings, mark=at position 0.6 with {\arrow[scale=1]{>}}}, postaction={decorate}]  (1,0)--(0,0);
\draw[thick,decoration={markings, mark=at position 0.6 with {\arrow[scale=1]{>}}}, postaction={decorate}]  (0,-1)--(0,0);
\draw[thick,decoration={markings, mark=at position 0.6 with {\arrow[scale=1]{<}}}, postaction={decorate}]  (0,1)--(0,0);
\filldraw [black](-1,0)circle(.1);
\filldraw [black](1,0)circle(.1);
\filldraw [black](0,-1)circle(.1);
\filldraw [black](0,1)circle(.1);
\filldraw [black](0,0)circle(.1);
\end{tikzpicture}
&
\begin{tikzpicture}[x=1.5em,y=1.5em]
\draw[thick,decoration={markings, mark=at position 0.6 with {\arrow[scale=1]{>}}}, postaction={decorate}] (-1,0)--(0,0);
\draw[thick,decoration={markings, mark=at position 0.6 with {\arrow[scale=1]{<}}}, postaction={decorate}]  (1,0)--(0,0);
\draw[thick,decoration={markings, mark=at position 0.6 with {\arrow[scale=1]{<}}}, postaction={decorate}]  (0,-1)--(0,0);
\draw[thick,decoration={markings, mark=at position 0.6 with {\arrow[scale=1]{>}}}, postaction={decorate}]  (0,1)--(0,0);
\filldraw [black](-1,0)circle(.1);
\filldraw [black](1,0)circle(.1);
\filldraw [black](0,-1)circle(.1);
\filldraw [black](0,1)circle(.1);
\filldraw [black](0,0)circle(.1);
\end{tikzpicture} \\\hline
$1$&$-1$&$0$&$0$&$0$&$0$\\ \hline
\end{tabular}
\end{equation}  See Figure~\ref{figure:iceasm} for an example.
This map defines a bijection from square ice configurations to ASMs (see \cite[Section~7]{elkies1992alternatingII}).  

\begin{figure}

	\begin{center}
		\raisebox{-0.45\height}{\begin{tikzpicture}[x=1.5em,y=1.5em]
			\draw[thick](0,0)--(3,0);
			\draw[thick](0,1)--(3,1);
			\draw[thick](0,2)--(3,2);
			\draw[thick](0,3)--(3,3);
			\draw[thick](0,0)--(0,3);
			\draw[thick](1,0)--(1,3);
			\draw[thick](2,0)--(2,3);
			\draw[thick](3,0)--(3,3);
			\draw[black,thick,decoration={markings, mark=at position 0.6 with {\arrow[scale=1]{>}}}, postaction={decorate}] (-1,0)--(0,0);
			\draw[black,thick,decoration={markings, mark=at position 0.6 with {\arrow[scale=1]{>}}}, postaction={decorate}] (-1,1)--(0,1);
			\draw[black,thick,decoration={markings, mark=at position 0.6 with {\arrow[scale=1]{>}}}, postaction={decorate}] (-1,2)--(0,2);
			\draw[black,thick,decoration={markings, mark=at position 0.6 with {\arrow[scale=1]{>}}}, postaction={decorate}] (-1,3)--(0,3);
			\draw[black,thick,decoration={markings, mark=at position 0.6 with {\arrow[scale=1]{>}}}, postaction={decorate}] (4,0)--(3,0);
			\draw[black,thick,decoration={markings, mark=at position 0.6 with {\arrow[scale=1]{>}}}, postaction={decorate}] (4,1)--(3,1);
			\draw[black,thick,decoration={markings, mark=at position 0.6 with {\arrow[scale=1]{>}}}, postaction={decorate}] (4,2)--(3,2);
			\draw[black,thick,decoration={markings, mark=at position 0.6 with {\arrow[scale=1]{>}}}, postaction={decorate}] (4,3)--(3,3);
			\draw[black,thick,decoration={markings, mark=at position 0.6 with {\arrow[scale=1]{<}}}, postaction={decorate}] (0,-1)--(0,0);
			\draw[black,thick,decoration={markings, mark=at position 0.6 with {\arrow[scale=1]{<}}}, postaction={decorate}] (1,-1)--(1,0);
			\draw[black,thick,decoration={markings, mark=at position 0.6 with {\arrow[scale=1]{<}}}, postaction={decorate}] (2,-1)--(2,0);
			\draw[black,thick,decoration={markings, mark=at position 0.6 with {\arrow[scale=1]{<}}}, postaction={decorate}] (3,-1)--(3,0);
			\draw[black,thick,decoration={markings, mark=at position 0.6 with {\arrow[scale=1]{<}}}, postaction={decorate}] (0,4)--(0,3);
			\draw[black,thick,decoration={markings, mark=at position 0.6 with {\arrow[scale=1]{<}}}, postaction={decorate}] (1,4)--(1,3);
			\draw[black,thick,decoration={markings, mark=at position 0.6 with {\arrow[scale=1]{<}}}, postaction={decorate}] (2,4)--(2,3);
			\draw[black,thick,decoration={markings, mark=at position 0.6 with {\arrow[scale=1]{<}}}, postaction={decorate}] (3,4)--(3,3);
			
			\draw[thick,decoration={markings, mark=at position 0.6 with {\arrow[scale=1]{<}}}, postaction={decorate}] (0,0)--(0,1);	
			\draw[thick,decoration={markings, mark=at position 0.6 with {\arrow[scale=1]{>}}}, postaction={decorate}] (0,1)--(0,2);		
			\draw[thick,decoration={markings, mark=at position 0.6 with {\arrow[scale=1]{>}}}, postaction={decorate}] (0,2)--(0,3);
			
			\draw[thick,decoration={markings, mark=at position 0.6 with {\arrow[scale=1]{<}}}, postaction={decorate}] (1,0)--(1,1);	
			\draw[thick,decoration={markings, mark=at position 0.6 with {\arrow[scale=1]{<}}}, postaction={decorate}] (1,1)--(1,2);		
			\draw[thick,decoration={markings, mark=at position 0.6 with {\arrow[scale=1]{>}}}, postaction={decorate}] (1,2)--(1,3);	
			
			\draw[thick,decoration={markings, mark=at position 0.6 with {\arrow[scale=1]{>}}}, postaction={decorate}] (2,0)--(2,1);	
			\draw[thick,decoration={markings, mark=at position 0.6 with {\arrow[scale=1]{<}}}, postaction={decorate}] (2,1)--(2,2);		
			\draw[thick,decoration={markings, mark=at position 0.6 with {\arrow[scale=1]{<}}}, postaction={decorate}] (2,2)--(2,3);	
			
			\draw[thick,decoration={markings, mark=at position 0.6 with {\arrow[scale=1]{<}}}, postaction={decorate}] (3,0)--(3,1);	
			\draw[thick,decoration={markings, mark=at position 0.6 with {\arrow[scale=1]{>}}}, postaction={decorate}] (3,1)--(3,2);		
			\draw[thick,decoration={markings, mark=at position 0.6 with {\arrow[scale=1]{>}}}, postaction={decorate}] (3,2)--(3,3);	
			
			\draw[thick,decoration={markings, mark=at position 0.6 with {\arrow[scale=1]{>}}}, postaction={decorate}] (0,0)--(1,0);	
			\draw[thick,decoration={markings, mark=at position 0.6 with {\arrow[scale=1]{>}}}, postaction={decorate}] (1,0)--(2,0);		
			\draw[thick,decoration={markings, mark=at position 0.6 with {\arrow[scale=1]{<}}}, postaction={decorate}] (2,0)--(3,0);	
			
			\draw[thick,decoration={markings, mark=at position 0.6 with {\arrow[scale=1]{<}}}, postaction={decorate}] (0,1)--(1,1);	
			\draw[thick,decoration={markings, mark=at position 0.6 with {\arrow[scale=1]{<}}}, postaction={decorate}] (1,1)--(2,1);		
			\draw[thick,decoration={markings, mark=at position 0.6 with {\arrow[scale=1]{>}}}, postaction={decorate}] (2,1)--(3,1);
			
			\draw[thick,decoration={markings, mark=at position 0.6 with {\arrow[scale=1]{>}}}, postaction={decorate}] (0,2)--(1,2);	
			\draw[thick,decoration={markings, mark=at position 0.6 with {\arrow[scale=1]{<}}}, postaction={decorate}] (1,2)--(2,2);		
			\draw[thick,decoration={markings, mark=at position 0.6 with {\arrow[scale=1]{<}}}, postaction={decorate}] (2,2)--(3,2);
			
			\draw[thick,decoration={markings, mark=at position 0.6 with {\arrow[scale=1]{>}}}, postaction={decorate}] (0,3)--(1,3);	
			\draw[thick,decoration={markings, mark=at position 0.6 with {\arrow[scale=1]{>}}}, postaction={decorate}] (1,3)--(2,3);		
			\draw[thick,decoration={markings, mark=at position 0.6 with {\arrow[scale=1]{<}}}, postaction={decorate}] (2,3)--(3,3);

			\filldraw [black](0,0)circle(.1);
			\filldraw [black](1,0)circle(.1);
			\filldraw [black](2,0)circle(.1);
			\filldraw [black](3,0)circle(.1);
			\filldraw [black](0,1)circle(.1);
			\filldraw [black](1,1)circle(.1);
			\filldraw [black](2,1)circle(.1);
			\filldraw [black](3,1)circle(.1);
			\filldraw [black](0,2)circle(.1);
			\filldraw [black](1,2)circle(.1);
			\filldraw [black](2,2)circle(.1);
			\filldraw [black](3,2)circle(.1);
			\filldraw [black](0,3)circle(.1);
			\filldraw [black](1,3)circle(.1);
			\filldraw [black](2,3)circle(.1);
			\filldraw [black](3,3)circle(.1);	
			\filldraw [black](0,-1)circle(.1);
			\filldraw [black](1,-1)circle(.1);
			\filldraw [black](2,-1)circle(.1);
			\filldraw [black](3,-1)circle(.1);
			\filldraw [black](0,4)circle(.1);
			\filldraw [black](1,4)circle(.1);
			\filldraw [black](2,4)circle(.1);
			\filldraw [black](3,4)circle(.1);
			\filldraw [black](-1,0)circle(.1);
			\filldraw [black](-1,1)circle(.1);
			\filldraw [black](-1,2)circle(.1);
			\filldraw [black](-1,3)circle(.1);
			\filldraw [black](4,0)circle(.1);
			\filldraw [black](4,1)circle(.1);
			\filldraw [black](4,2)circle(.1);
			\filldraw [black](4,3)circle(.1);
			\end{tikzpicture}}
		\hspace{5em} 
		$\left(\begin{array}{cccc}
		0&0&1&0\\
		0&1&0&0\\
		1&0&-1&1\\
		0&0&1&0
		\end{array}\right)$
	\end{center}
	\caption{Pictured on the left is a square ice configuration $\mathcal I$. 	The ASM on the right is the image of  $\mathcal I$ under the map defined by (\ref{eqn:IceASM}).}
\label{figure:iceasm}
\end{figure}
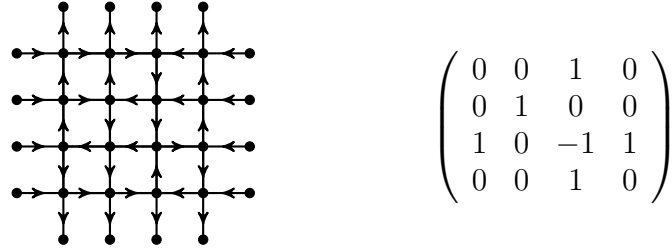

Just as there are six vertex states in  square ice configurations, bumpless pipe dreams consist of six types of tiles.  The bijection from square ice to ASMs, composed with the bijection from ASMs to BPDs, produces the following tile-by-tile correspondence.
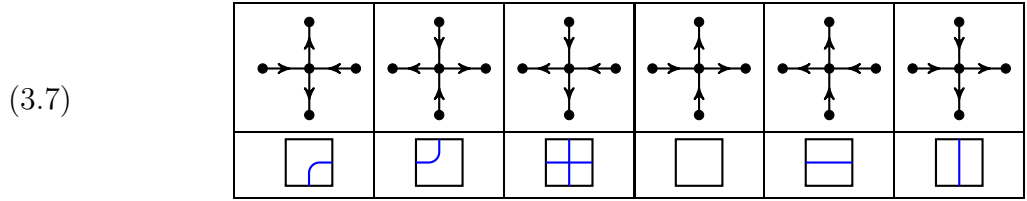
\begin{equation}
\label{eqn:IceBump}
\begin{tabular}{|c|c|c|c|c|c|} \hline
\rule[.75em]{0pt}{3em}

\begin{tikzpicture}[x=1.5em,y=1.5em]
\draw[thick,decoration={markings, mark=at position 0.6 with {\arrow[scale=1]{>}}}, postaction={decorate}] (-1,0)--(0,0);
\draw[thick,decoration={markings, mark=at position 0.6 with {\arrow[scale=1]{>}}}, postaction={decorate}]  (1,0)--(0,0);
\draw[thick,decoration={markings, mark=at position 0.6 with {\arrow[scale=1]{<}}}, postaction={decorate}]  (0,-1)--(0,0);
\draw[thick,decoration={markings, mark=at position 0.6 with {\arrow[scale=1]{<}}}, postaction={decorate}]  (0,1)--(0,0);
\filldraw [black](-1,0)circle(.1);
\filldraw [black](1,0)circle(.1);
\filldraw [black](0,-1)circle(.1);
\filldraw [black](0,1)circle(.1);
\filldraw [black](0,0)circle(.1);
\end{tikzpicture}
&
\begin{tikzpicture}[x=1.5em,y=1.5em]
\draw[thick,decoration={markings, mark=at position 0.6 with {\arrow[scale=1]{<}}}, postaction={decorate}] (-1,0)--(0,0);
\draw[thick,decoration={markings, mark=at position 0.6 with {\arrow[scale=1]{<}}}, postaction={decorate}]  (1,0)--(0,0);
\draw[thick,decoration={markings, mark=at position 0.6 with {\arrow[scale=1]{>}}}, postaction={decorate}]  (0,-1)--(0,0);
\draw[thick,decoration={markings, mark=at position 0.6 with {\arrow[scale=1]{>}}}, postaction={decorate}]  (0,1)--(0,0);
\filldraw [black](-1,0)circle(.1);
\filldraw [black](1,0)circle(.1);
\filldraw [black](0,-1)circle(.1);
\filldraw [black](0,1)circle(.1);
\filldraw [black](0,0)circle(.1);
\end{tikzpicture}
&
\begin{tikzpicture}[x=1.5em,y=1.5em]
\draw[thick,decoration={markings, mark=at position 0.6 with {\arrow[scale=1]{<}}}, postaction={decorate}] (-1,0)--(0,0);
\draw[thick,decoration={markings, mark=at position 0.6 with {\arrow[scale=1]{>}}}, postaction={decorate}]  (1,0)--(0,0);
\draw[thick,decoration={markings, mark=at position 0.6 with {\arrow[scale=1]{<}}}, postaction={decorate}]  (0,-1)--(0,0);
\draw[thick,decoration={markings, mark=at position 0.6 with {\arrow[scale=1]{>}}}, postaction={decorate}]  (0,1)--(0,0);
\filldraw [black](-1,0)circle(.1);
\filldraw [black](1,0)circle(.1);
\filldraw [black](0,-1)circle(.1);
\filldraw [black](0,1)circle(.1);
\filldraw [black](0,0)circle(.1);
\end{tikzpicture}
&
\begin{tikzpicture}[x=1.5em,y=1.5em]
\draw[thick,decoration={markings, mark=at position 0.6 with {\arrow[scale=1]{>}}}, postaction={decorate}] (-1,0)--(0,0);
\draw[thick,decoration={markings, mark=at position 0.6 with {\arrow[scale=1]{<}}}, postaction={decorate}]  (1,0)--(0,0);
\draw[thick,decoration={markings, mark=at position 0.6 with {\arrow[scale=1]{>}}}, postaction={decorate}]  (0,-1)--(0,0);
\draw[thick,decoration={markings, mark=at position 0.6 with {\arrow[scale=1]{<}}}, postaction={decorate}]  (0,1)--(0,0);
\filldraw [black](-1,0)circle(.1);
\filldraw [black](1,0)circle(.1);
\filldraw [black](0,-1)circle(.1);
\filldraw [black](0,1)circle(.1);
\filldraw [black](0,0)circle(.1);
\end{tikzpicture}
&
\begin{tikzpicture}[x=1.5em,y=1.5em]
\draw[thick,decoration={markings, mark=at position 0.6 with {\arrow[scale=1]{<}}}, postaction={decorate}] (-1,0)--(0,0);
\draw[thick,decoration={markings, mark=at position 0.6 with {\arrow[scale=1]{>}}}, postaction={decorate}]  (1,0)--(0,0);
\draw[thick,decoration={markings, mark=at position 0.6 with {\arrow[scale=1]{>}}}, postaction={decorate}]  (0,-1)--(0,0);
\draw[thick,decoration={markings, mark=at position 0.6 with {\arrow[scale=1]{<}}}, postaction={decorate}]  (0,1)--(0,0);
\filldraw [black](-1,0)circle(.1);
\filldraw [black](1,0)circle(.1);
\filldraw [black](0,-1)circle(.1);
\filldraw [black](0,1)circle(.1);
\filldraw [black](0,0)circle(.1);
\end{tikzpicture}
&
\begin{tikzpicture}[x=1.5em,y=1.5em]
\draw[thick,decoration={markings, mark=at position 0.6 with {\arrow[scale=1]{>}}}, postaction={decorate}] (-1,0)--(0,0);
\draw[thick,decoration={markings, mark=at position 0.6 with {\arrow[scale=1]{<}}}, postaction={decorate}]  (1,0)--(0,0);
\draw[thick,decoration={markings, mark=at position 0.6 with {\arrow[scale=1]{<}}}, postaction={decorate}]  (0,-1)--(0,0);
\draw[thick,decoration={markings, mark=at position 0.6 with {\arrow[scale=1]{>}}}, postaction={decorate}]  (0,1)--(0,0);
\filldraw [black](-1,0)circle(.1);
\filldraw [black](1,0)circle(.1);
\filldraw [black](0,-1)circle(.1);
\filldraw [black](0,1)circle(.1);
\filldraw [black](0,0)circle(.1);
\end{tikzpicture} \\\hline
\rule[.75em]{0pt}{1em}

\begin{tikzpicture}[x=1.5em,y=1.5em]
\draw[color=black, thick](0,1)rectangle(1,2);
\draw[thick,color=blue,rounded corners] (.5,1)--(.5,1.5)--(1,1.5);
\end{tikzpicture} &
\begin{tikzpicture}[x=1.5em,y=1.5em]
\draw[color=black, thick](0,1)rectangle(1,2);
\draw[thick,color=blue,rounded corners] (.5,2)--(.5,1.5)--(0,1.5);
\end{tikzpicture} &
\begin{tikzpicture}[x=1.5em,y=1.5em]
\draw[color=black, thick](0,1)rectangle(1,2);
\draw[thick,color=blue,rounded corners] (0,1.5)--(1,1.5);
\draw[thick,color=blue,rounded corners] (.5,1)--(.5,2);
\end{tikzpicture} &
\begin{tikzpicture}[x=1.5em,y=1.5em]
\draw[color=black, thick](0,1)rectangle(1,2);
\end{tikzpicture}&
\begin{tikzpicture}[x=1.5em,y=1.5em]
\draw[color=black, thick](0,1)rectangle(1,2);
\draw[thick,color=blue,rounded corners] (0,1.5)--(1,1.5);
\end{tikzpicture} &
\begin{tikzpicture}[x=1.5em,y=1.5em]
\draw[color=black, thick](0,1)rectangle(1,2);
\draw[thick,color=blue,rounded corners] (.5,1)--(.5,2);
\end{tikzpicture}  \\\hline
\end{tabular}
\end{equation}
Equivalently, the map draws pipes along edges which point to the left and the edges which point down.

\begin{lemma}
\label{lemma:icepipe}
	The map defined by (\ref{eqn:IceBump})  is a bijection from square ice configurations to bumpless pipe dreams.
\end{lemma}
\begin{remark}
	Defining segments of an ASM provide an easy description of the map ${\sf ASM}(n)\rightarrow {\sf Ice}(n)$.  Orient horizontal (vertical) interior edges pointing to the left (downwards) if and only if they coincide with a defining line segment of the Rothe diagram.  Due to (\ref{eqn:IceBump}) and the description of the map $\Phi$, this is the inverse to the map given in (\ref{eqn:IceASM}). 
\end{remark}

The next lemma provides the explicit connection between square ice configurations and corner sum functions.  
\begin{lemma}
\label{lemma:rankice}
		Take $i,j\in [n]$ and write $r_A(i,j)=k$.  If $A\mapsto \mathcal I$,  the restriction of $r_A$ to $\{i-1,i\}\times\{j-1,j\}$ corresponds to the interior vertex at $(i,j)$ of $\mathcal I$ as pictured below.
		\begin{center}
			\begin{tabular}{|c|c|c|c|c|c|} \hline
				\rule[.75em]{0pt}{3em}
				\begin{tikzpicture}[x=1.5em,y=1.5em]
				\draw[thick,decoration={markings, mark=at position 0.6 with {\arrow[scale=1]{>}}}, postaction={decorate}] (-1,0)--(0,0);
				\draw[thick,decoration={markings, mark=at position 0.6 with {\arrow[scale=1]{>}}}, postaction={decorate}]  (1,0)--(0,0);
				\draw[thick,decoration={markings, mark=at position 0.6 with {\arrow[scale=1]{<}}}, postaction={decorate}]  (0,-1)--(0,0);
				\draw[thick,decoration={markings, mark=at position 0.6 with {\arrow[scale=1]{<}}}, postaction={decorate}]  (0,1)--(0,0);
				\filldraw [black](-1,0)circle(.1);
				\filldraw [black](1,0)circle(.1);
				\filldraw [black](0,-1)circle(.1);
				\filldraw [black](0,1)circle(.1);
				\filldraw [black](0,0)circle(.1);
				\end{tikzpicture}
				&
				\begin{tikzpicture}[x=1.5em,y=1.5em]
				\draw[thick,decoration={markings, mark=at position 0.6 with {\arrow[scale=1]{<}}}, postaction={decorate}] (-1,0)--(0,0);
				\draw[thick,decoration={markings, mark=at position 0.6 with {\arrow[scale=1]{<}}}, postaction={decorate}]  (1,0)--(0,0);
				\draw[thick,decoration={markings, mark=at position 0.6 with {\arrow[scale=1]{>}}}, postaction={decorate}]  (0,-1)--(0,0);
				\draw[thick,decoration={markings, mark=at position 0.6 with {\arrow[scale=1]{>}}}, postaction={decorate}]  (0,1)--(0,0);
				\filldraw [black](-1,0)circle(.1);
				\filldraw [black](1,0)circle(.1);
				\filldraw [black](0,-1)circle(.1);
				\filldraw [black](0,1)circle(.1);
				\filldraw [black](0,0)circle(.1);
				\end{tikzpicture}
				&
				\begin{tikzpicture}[x=1.5em,y=1.5em]
				\draw[thick,decoration={markings, mark=at position 0.6 with {\arrow[scale=1]{<}}}, postaction={decorate}] (-1,0)--(0,0);
				\draw[thick,decoration={markings, mark=at position 0.6 with {\arrow[scale=1]{>}}}, postaction={decorate}]  (1,0)--(0,0);
				\draw[thick,decoration={markings, mark=at position 0.6 with {\arrow[scale=1]{<}}}, postaction={decorate}]  (0,-1)--(0,0);
				\draw[thick,decoration={markings, mark=at position 0.6 with {\arrow[scale=1]{>}}}, postaction={decorate}]  (0,1)--(0,0);
				\filldraw [black](-1,0)circle(.1);
				\filldraw [black](1,0)circle(.1);
				\filldraw [black](0,-1)circle(.1);
				\filldraw [black](0,1)circle(.1);
				\filldraw [black](0,0)circle(.1);
				\end{tikzpicture}
				&
				\begin{tikzpicture}[x=1.5em,y=1.5em]
				\draw[thick,decoration={markings, mark=at position 0.6 with {\arrow[scale=1]{>}}}, postaction={decorate}] (-1,0)--(0,0);
				\draw[thick,decoration={markings, mark=at position 0.6 with {\arrow[scale=1]{<}}}, postaction={decorate}]  (1,0)--(0,0);
				\draw[thick,decoration={markings, mark=at position 0.6 with {\arrow[scale=1]{>}}}, postaction={decorate}]  (0,-1)--(0,0);
				\draw[thick,decoration={markings, mark=at position 0.6 with {\arrow[scale=1]{<}}}, postaction={decorate}]  (0,1)--(0,0);
				\filldraw [black](-1,0)circle(.1);
				\filldraw [black](1,0)circle(.1);
				\filldraw [black](0,-1)circle(.1);
				\filldraw [black](0,1)circle(.1);
				\filldraw [black](0,0)circle(.1);
				\end{tikzpicture}
				&
				\begin{tikzpicture}[x=1.5em,y=1.5em]
				\draw[thick,decoration={markings, mark=at position 0.6 with {\arrow[scale=1]{<}}}, postaction={decorate}] (-1,0)--(0,0);
				\draw[thick,decoration={markings, mark=at position 0.6 with {\arrow[scale=1]{>}}}, postaction={decorate}]  (1,0)--(0,0);
				\draw[thick,decoration={markings, mark=at position 0.6 with {\arrow[scale=1]{>}}}, postaction={decorate}]  (0,-1)--(0,0);
				\draw[thick,decoration={markings, mark=at position 0.6 with {\arrow[scale=1]{<}}}, postaction={decorate}]  (0,1)--(0,0);
				\filldraw [black](-1,0)circle(.1);
				\filldraw [black](1,0)circle(.1);
				\filldraw [black](0,-1)circle(.1);
				\filldraw [black](0,1)circle(.1);
				\filldraw [black](0,0)circle(.1);
				\end{tikzpicture}
				&
				\begin{tikzpicture}[x=1.5em,y=1.5em]
				\draw[thick,decoration={markings, mark=at position 0.6 with {\arrow[scale=1]{>}}}, postaction={decorate}] (-1,0)--(0,0);
				\draw[thick,decoration={markings, mark=at position 0.6 with {\arrow[scale=1]{<}}}, postaction={decorate}]  (1,0)--(0,0);
				\draw[thick,decoration={markings, mark=at position 0.6 with {\arrow[scale=1]{<}}}, postaction={decorate}]  (0,-1)--(0,0);
				\draw[thick,decoration={markings, mark=at position 0.6 with {\arrow[scale=1]{>}}}, postaction={decorate}]  (0,1)--(0,0);
				\filldraw [black](-1,0)circle(.1);
				\filldraw [black](1,0)circle(.1);
				\filldraw [black](0,-1)circle(.1);
				\filldraw [black](0,1)circle(.1);
				\filldraw [black](0,0)circle(.1);
				\end{tikzpicture} \\\hline
				\rule[-1.5em]{0pt}{4em}
								
				$\scriptsize \left[\begin{array}{cc}k-1&k-1\\k-1&k\end{array}\right]$ & 
				$\scriptsize \left[\begin{array}{cc}k-1&k\\k&k\end{array}\right]$&	
				$\scriptsize \left[\begin{array}{cc}k-2&k-1\\k-1&k\end{array}\right]$ &
				$\scriptsize \left[\begin{array}{cc}k&k\\k&k\end{array}\right]$&
				$\scriptsize \left[\begin{array}{cc}k-1&k-1\\k&k\end{array}\right]$&
				$\scriptsize \left[\begin{array}{cc}k-1&k\\k-1&k\end{array}\right]$\\ \hline 
			\end{tabular}
		\end{center}
	\end{lemma}

	\begin{proof} 
		Suppose that $A\mapsto \mathcal I$.  
		We have fixed $r_A(i,j)=k$.  Furthermore, since \[r_A(i,j)-r_A(i-1,j)\in\{0,1\} \enspace \text{and} \enspace  r_A(i,j)-r_A(i-1,j)\in\{0,1\},\] there are only six possibilities for \[\left[\begin{array}{cc}r_A(i-1,j-1)&r_A(i-1,j)\\r_A(i,j-1)&r_A(i,j)\end{array}\right].\] These are precisely the types of matrices which appear in the statement of the lemma. 
		
		By (\ref{eqn:cornersuminverse}) and (\ref{eqn:IceASM}), we must have the correspondences pictured below.
		\[\begin{tikzpicture}[x=1.5em,y=1.5em]
		\draw[thick,decoration={markings, mark=at position 0.6 with {\arrow[scale=1]{>}}}, postaction={decorate}] (-1,0)--(0,0);
		\draw[thick,decoration={markings, mark=at position 0.6 with {\arrow[scale=1]{>}}}, postaction={decorate}]  (1,0)--(0,0);
		\draw[thick,decoration={markings, mark=at position 0.6 with {\arrow[scale=1]{<}}}, postaction={decorate}]  (0,-1)--(0,0);
		\draw[thick,decoration={markings, mark=at position 0.6 with {\arrow[scale=1]{<}}}, postaction={decorate}]  (0,1)--(0,0);
		\filldraw [black](-1,0)circle(.1);
		\filldraw [black](1,0)circle(.1);
		\filldraw [black](0,-1)circle(.1);
		\filldraw [black](0,1)circle(.1);
		\filldraw [black](0,0)circle(.1);
		\end{tikzpicture} \hspace{1em} \raisebox{1\height}{$\longleftrightarrow \enspace \left[\begin{array}{cc}k-1&k-1\\k-1&k\end{array}\right] $}
		\hspace {3em}
		\begin{tikzpicture}[x=1.5em,y=1.5em]
		\draw[thick,decoration={markings, mark=at position 0.6 with {\arrow[scale=1]{<}}}, postaction={decorate}] (-1,0)--(0,0);
		\draw[thick,decoration={markings, mark=at position 0.6 with {\arrow[scale=1]{<}}}, postaction={decorate}]  (1,0)--(0,0);
		\draw[thick,decoration={markings, mark=at position 0.6 with {\arrow[scale=1]{>}}}, postaction={decorate}]  (0,-1)--(0,0);
		\draw[thick,decoration={markings, mark=at position 0.6 with {\arrow[scale=1]{>}}}, postaction={decorate}]  (0,1)--(0,0);
		\filldraw [black](-1,0)circle(.1);
		\filldraw [black](1,0)circle(.1);
		\filldraw [black](0,-1)circle(.1);
		\filldraw [black](0,1)circle(.1);
		\filldraw [black](0,0)circle(.1);
		\end{tikzpicture} \hspace{1em}\raisebox{1\height}{$\longleftrightarrow \enspace \left[\begin{array}{cc}k-1&k\\k&k\end{array}\right] $}
		\]

		From (\ref{eqn:IceASM}), we see that adjacent horizontal (and vertical) edges change direction at an interior vertex $(i,j)$ if and only if $A_{i\,j}\in \{-1,1\}$.  Assume  $A_{i\,j}=0$. Then 
		\begin{enumerate}
			\item ${\rm row}_{A,i}(j)=1$ if and only if the adjacent horizontal edges to $(i,j)$ in $\mathcal I$ point left, and
			\item ${\rm col}_{A,j}(i)=1$ if and only if the adjacent vertical edges to $(i,j)$ in $\mathcal I$ point down.
		\end{enumerate}
		Since $r_A(i,j)=r_A(i-1,j)+{\rm row}_{A,i}(j)$, the horizontal edges point to the left if and only if $r_A(i-1,j)=k-1$.  Similarly, the vertical edges point down if and only if $r_A(i,j-1)=k-1$.  These assertions are enough to determine the last  four correspondences.
	\end{proof}

\subsection{Additional properties of bumpless pipe dreams}

We may now apply the above bijections to derive facts about bumpless pipe dreams. 
\begin{lemma}
	\label{lemma:uniquelydet}
	A bumpless pipe dream is uniquely determined by knowing any of the following:
	\begin{enumerate}
		\item the locations of its upward and downward elbow tiles,
		\item the locations of its crossing and blank tiles, or
		\item the locations of its horizontal and vertical tiles.
	\end{enumerate}
\end{lemma}
\begin{proof}
  \noindent (1) This follows immediately from the bijection between ASMs and bumpless pipe dreams.
  
  \noindent (2) Fix $A\in{\sf ASM}(n)$ and let $\mathcal P=\Phi(A)$.  By Lemma~\ref{lemma:rankice}, 
  \begin{equation}
  \label{eqn:diagrankjump}
  r_A(a,b)-r_A(a-1,b-1)=
  \begin{cases}
  0& \text{if}\enspace \mathcal P_{a\,b} \enspace \text{is blank}\\
  2&\text{if}\enspace \mathcal P_{a\,b} \enspace \text{is crossing, and}\\
  1& \text{otherwise}.
  \end{cases}  
  \end{equation}
  As such, $r_A$ can be uniquely recovered starting from the initial conditions $r_A(i,0)=r_A(0,i)=0$ using (\ref{eqn:diagrankjump}).   Since $r_A$ is determined by this information, so is $\mathcal P=\Phi(A)$.
  
  \noindent (3) By an analogous argument to the previous part, $r_A$ can be recovered by working along antidiagonals, using the positions of the horizontal and vertical edges in $\Phi(A)$.
   Explicitly, \begin{equation}
  \label{eqn:antidiagrankjump}
  r_A(a-1,b)-r_A(a,b-1)=
  \begin{cases}
  -1& \text{if}\enspace \mathcal P_{a\,b} \enspace \text{is a horizontal pipe,}\\
  1&\text{if}\enspace \mathcal P_{a\,b} \enspace \text{is a vertical pipe, and}\\
  0& \text{otherwise}. 
  \end{cases}  
  \end{equation}
Thus, the result follows.
\end{proof}

It is well known (say from the perspective of square ice) that there are as many crossing tiles as blank tiles in each BPD.  Likewise, there are the same number of horizontal and vertical tiles (see \cite[Section~2.1]{behrend.weighted}).  We will make use of the following refinement.
\begin{lemma}
  Let $\mathcal P\in{\sf BPD}(n)$.  
  \begin{enumerate}
  \item Within each diagonal of $\mathcal P$, there are as many crossing tiles as blank tiles.
  \item Within each antidiagonal of $\mathcal P$, there are as many horizontal tiles as vertical tiles.
  \end{enumerate}
\end{lemma}
\begin{proof}

  \noindent (1) Fix $i\geq 0$.  Then $r_A(0,i)=0$ and $r_A(n-i,n)=n-i$.  Then by (\ref{eqn:diagrankjump}), \[n-i=r_A(n-i,n)-r_A(0,i)=\sum_{k=i+1}^{n}r_A(k-i,k)-r_A(k-i-1,k-1).\]  Since there are $n-i$ summands, taking values in $\{0,1,2\}$, there are as many $0$'s as $2$'s.  Therefore, $\mathcal P$ has as many blank tiles as crossing tiles in this diagonal.  A similar argument works for the diagonal containing $(i,0)$ and $(n,n-i)$.
  
  \noindent (2) We have \[0=r_A(0,i)-r_A(i,0)=\sum_{k=0}^{i-1}r_A(i-1-k,1+k)-r_A(i-k,0+k).\]
  By (\ref{eqn:antidiagrankjump}), the summands take values in $\{-1,0,1\}$.  Thus, there must be the same number of $-1$'s as $1$'s in the sum.
  Likewise,
  \[0=r_A(i,n)-r_A(n,i)=\sum_{k=0}^{n-i-1}r_A(n-1-k,i+1+k)-r_A(n-k,i+k).\]
  Again, there are as many $1$'s as $-1$'s in the sum.
Thus, in each antidiagonal, there are as many horizontal tiles as vertical tiles.
\end{proof}

\begin{lemma}
\label{lemma:rankcrossing}
Fix $A\in \asm(n)$ and let $\mathcal P=\asmtobpd(A)$. The label of the crossing $(i,j)\in \mathcal P$ is $r_A(i,j)-1$.
\end{lemma}
\begin{proof}
Since $(i,j)$ is a crossing tile,\[r_A(i,j)=r_A(i-1,j-1)+2.\]
Thus, the lemma follows from Lemma~\ref{lemma:rankice} by noting $r_A(i-1,j-1)$ counts the number of pipes which past strictly northwest of $(i,j)$.
\end{proof}

\section{Keys via bumpless pipe dreams}

\label{section:key}

\subsection{Inflation on alternating sign matrices}
\label{section:inflation}

Lascoux \cite{Lascoux:ice} defined a procedure to iteratively remove the negative entries from an ASM.  
Fix $A\in \asm(n)$ and pick $(a,b)\in [n]\times[n]$. The pivots of $(a,b)$ are those positions $(i,j)\neq (a,b)$  so that
\begin{enumerate}
	\item $(i,j)$ sits weakly northwest of $(a,b)$,
	\item $A_{i\,j}=1$, and
	\item the only nonzero entries of $A$ in the region $[i,r]\times[j,s]$ are $A_{i\,j}$ and possibly $A_{a \, b}$.
\end{enumerate} 
Now suppose $A_{a\,b}=-1$.  Write $(i_1,j_1),\,\ldots,\,(i_k,j_k)$ for  the pivots of $(a,b)$.  Furthermore, assume \[i_1<i_2<\cdots<i_k \enspace \text{and} \enspace j_1>j_2>\cdots >j_k.\]
We say $A_{a\,b}$ is a \mydef{removable} $-1$ if there are no negative entries (excluding $A_{a\,b}$) in the region
\begin{equation}
\label{eqn:inflationregion}
\bigcup_{\ell=1}^{k-1}[i_{\ell},a]\times [j_{\ell+1},b].
\end{equation}
In particular, if $A_{a\,b}$ is removable,  the only nonzero entries in this region happen at $(a,b)$ and its pivots.  A northeast most negative entry within $A$ is always removable.  

\mydef{Inflation} sets all of these entries to $0$ and then places $1$'s in the positions \[(i_1,j_2),\,\ldots,\, (i_{k-1},j_k).\] See Figure~\ref{figure:inflation} for an example. Applying inflation to a removable $-1$ produces a well-defined ASM with one less negative entry than the original ASM.    If $A\mapsto A'$ by inflation, also call the map $\asmtobpd(A)\mapsto \asmtobpd(A')$ inflation.
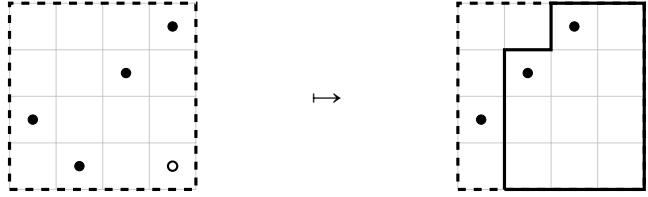
\begin{figure}
	\begin{center}
		\begin{tikzpicture}[x=1.5em,y=1.5em]
		\draw[step=1,gray!40, thin] (0,1) grid (4,5);
		\draw[color=black, very thick,dashed](0,1)rectangle(4,5);
		\filldraw [black](1.5,1.5)circle(.1);
		\filldraw [black](3.5,4.5)circle(.1);
		\filldraw [black](2.5,3.5)circle(.1);
		\filldraw [black](.5,2.5)circle(.1);
		\filldraw [color=black,fill=white,thick](3.5,1.5)circle(.1);
		\end{tikzpicture}
		\hspace{3em}
		\raisebox{2.75em}{$\mapsto$}
		\hspace{3em}
		\begin{tikzpicture}[x=1.5em,y=1.5em]
		\draw[step=1,gray!40, thin] (0,1) grid (4,5);
		\draw[color=black, very thick,dashed](0,1)rectangle(4,5);
		\draw[color=black,very thick](1,1)--(1,4)--(2,4)--(2,5)--(3,5)--(4,5)--(4,1)--(1,1);
		\filldraw [black](1.5,3.5)circle(.1);
		\filldraw [black](2.5,4.5)circle(.1);
		\filldraw [black](.5,2.5)circle(.1);
		\end{tikzpicture}
	\end{center}	
\caption{Pictured above is (a portion of) the graph an ASM before and after applying inflation.  The region from (\ref{eqn:inflationregion}) has been outlined for emphasis. }
\label{figure:inflation}
\end{figure}
  The \mydef{key} of $A$, denoted $\mkey{A}$, is the permutation  obtained by working along rows in $A$ from left to right (starting with the top row) and removing each removable $-1$ using inflation. See Figure~\ref{figure:inflate} for an example.
\begin{figure}
	\begin{center}
		\begin{tikzpicture}[x=1.5em,y=1.5em]
		\draw[step=1,gray!40, thin] (0,0) grid (7,7);
		\draw[color=black, thick](0,0)rectangle(7,7);
		\filldraw[color=black,fill=white,thick](5.5,2.5)circle(.1);
		\filldraw[color=black,fill=white,thick](2.5,1.5)circle(.1);
		\filldraw [color=black,thick](5.5,6.5)circle(.1);
		\filldraw [color=black,thick](6.5,2.5)circle(.1);
		\filldraw [color=black,thick](2.5,.5)circle(.1);
		\filldraw [color=black,thick](5.5,1.5)circle(.1);
		\filldraw [color=black,thick](.5,1.5)circle(.1);
		\filldraw [color=black,thick](1.5,2.5)circle(.1);
		\filldraw [color=black,thick](2.5,5.5)circle(.1);
		\filldraw [color=black,thick](4.5,4.5)circle(.1);
		\filldraw [color=black,thick](3.5,3.5)circle(.1);
		\end{tikzpicture}
		\hspace{1em}
	\begin{tikzpicture}[x=1.5em,y=1.5em]
		\draw[step=1,gray!40, thin] (0,0) grid (7,7);
		\draw[color=black, thick](0,0)rectangle(7,7);
		\filldraw[color=black,fill=white,thick](2.5,1.5)circle(.1);
		\filldraw [color=black,thick](4.5,6.5)circle(.1);
		\filldraw [color=black,thick](6.5,2.5)circle(.1);
		\filldraw [color=black,thick](2.5,.5)circle(.1);
		\filldraw [color=black,thick](5.5,1.5)circle(.1);
		\filldraw [color=black,thick](1.5,3.5)circle(.1);
		\filldraw [color=black,thick](.5,1.5)circle(.1);
		\filldraw [color=black,thick](2.5,5.5)circle(.1);
		\filldraw [color=black,thick](3.5,4.5)circle(.1);
		\end{tikzpicture}
		\hspace{1em}
		\begin{tikzpicture}[x=1.5em,y=1.5em]
		\draw[step=1,gray!40, thin] (0,0) grid (7,7);
		\draw[color=black, thick](0,0)rectangle(7,7);
		\filldraw [color=black,thick](6.5,2.5)circle(.1);
		\filldraw [color=black,thick](2.5,.5)circle(.1);
		\filldraw [color=black,thick](5.5,1.5)circle(.1);
		\filldraw [color=black,thick](.5,3.5)circle(.1);
		\filldraw [color=black,thick](1.5,5.5)circle(.1);
		\filldraw [color=black,thick](4.5,6.5)circle(.1);
		\filldraw [color=black,thick](3.5,4.5)circle(.1);
		\end{tikzpicture}
	\end{center}
	\caption{	Above, we show the reduction of the graph of an ASM to the graph of its key, $5241763$. The ASM on the left corresponds to the BPD pictured in Figure~\ref{figure:demprod}.}
	\label{figure:inflate}  
\end{figure}

Lascoux's choice of the name key was meant to invoke an analogy with keys of semistandard tableaux.  If we identify an ASM with its \emph{monotone triangle}, the two definitions are compatible \cite[Corollary~9]{Aval2007}. Furthermore, an ASM is always (weakly) larger than its key in the poset of ASMs, i.e.,\ $\mkey{A}\leq A$ \cite[Proposition~3]{Aval2007}.  

\begin{remark}
	Note that $\mkey{A}$ need not be maximal among permutations below $A$.  For instance
	\[A=\left(\begin{array}{cccc}
	0&1&0&0\\
	1&-1&1&0\\
	0&1&-1&1\\
	0&0&1&0\\
	\end{array}\right)
	>
	\left(\begin{array}{cccc}
	0&1&0&0\\
	1&0&0&0\\
	0&0&0&1\\
	0&0&1&0\\
	\end{array}\right)
	>
	\left(\begin{array}{cccc}
	1&0&0&0\\
	0&1&0&0\\
	0&0&0&1\\
	0&0&1&0\\
	\end{array}\right)=\mkey{A}.
	\]
	In particular, this illustrates that if $A>B$, it is not necessarily true  that $\mkey{A}>\mkey{B}$. 
\end{remark}

\subsection{Keys of ASMs via Demazure products}

In this section, we connect Lascoux's key  to Demazure products of words defined by bumpless pipe dreams.

\begin{lemma}
\label{lemma:compatinflation}
Suppose $A'$ is obtained from $A$ by removing a removable negative 1.  
\begin{enumerate}
\item  $\demprod{\asmtobpd(A)}=\demprod{\asmtobpd(A')}$.
\item If the $-1$ in $A$ has $k$ pivots, then $ |\mathbf a_{\asmtobpd(A')}|+k-2=|\mathbf a_{\asmtobpd(A)}|.$
\end{enumerate}

\end{lemma}
\begin{proof}
\noindent (1) If the removable $-1$ has just two pivots,  inflation amounts to iteratively applying the moves in (\ref{eqn:planarmoves1}).  (See the example below.)
\[\begin{tikzpicture}[x=1em,y=1em]
\draw[thick] (2,-1)--(2,6);
\draw[thick] (1,-1)--(1,6);
\draw[thick] (-1,2)--(4,2);
\draw[thick] (-1,4)--(4,4);
\draw[thick,rounded corners,color=blue] (0,-1)--(0,0)--(3,0)--(3,5)--(4,5);
\end{tikzpicture}
\hspace{1em}\raisebox{3em}{$\rightarrow$}\hspace{1em}
\begin{tikzpicture}[x=1em,y=1em]
\draw[thick] (2,-1)--(2,6);
\draw[thick] (1,-1)--(1,6);
\draw[thick] (-1,2)--(4,2);
\draw[thick] (-1,4)--(4,4);
\draw[thick,rounded corners,color=blue] (0,-1)--(0,0)--(1.5,0)--(1.5,3)--(2.5,3)--(3.5,3)--(3.5,5)--(4,5);
\end{tikzpicture}
\hspace{1em}\raisebox{3em}{$\rightarrow$}\hspace{1em}
\begin{tikzpicture}[x=1em,y=1em]
\draw[thick] (2,-1)--(2,6);
\draw[thick] (1,-1)--(1,6);
\draw[thick] (-1,2)--(4,2);
\draw[thick] (-1,4)--(4,4);
\draw[thick,rounded corners,color=blue] (0,-1)--(0,0)--(1.5,0)--(1.5,5)--(4,5);
\end{tikzpicture}
\hspace{1em}\raisebox{3em}{$\rightarrow$}\hspace{1em}
\begin{tikzpicture}[x=1em,y=1em]
\draw[thick] (2,-1)--(2,6);
\draw[thick] (1,-1)--(1,6);
\draw[thick] (-1,2)--(4,2);
\draw[thick] (-1,4)--(4,4);
\draw[thick,rounded corners,color=blue] (0,-1)--(0,3)--(1.5,3)--(1.5,5)--(4,5);
\end{tikzpicture}
\hspace{1em}\raisebox{3em}{$\rightarrow$}\hspace{1em}
\begin{tikzpicture}[x=1em,y=1em]
\draw[thick] (2,-1)--(2,6);
\draw[thick] (1,-1)--(1,6);
\draw[thick] (-1,2)--(4,2);
\draw[thick] (-1,4)--(4,4);
\draw[thick,rounded corners,color=blue] (0,-1)--(0,5)--(4,5);
\end{tikzpicture}
\]
Otherwise, we can find a configuration as pictured below (with some number of irrelevant horizontal / vertical black pipes).

\[
\begin{tikzpicture}[x=1em,y=1em]
\draw[thick] (2,-1)--(2,9);
\draw[thick] (1,-1)--(1,9);
\draw[thick] (-3,2)--(4,2);
\draw[thick] (-3,4)--(4,4);
\draw[thick] (-3,7)--(4,7);
\draw[thick,rounded corners,color=Plum] (-3,0)--(3,0)--(3,8)--(4,8);
\draw[thick,rounded corners,color=blue] (-1,-1)--(-1,6)--(4,6);
\end{tikzpicture}
\hspace{1em}\raisebox{3em}{$\rightarrow$}\hspace{1em}
\begin{tikzpicture}[x=1em,y=1em]
\draw[thick] (2,-1)--(2,9);
\draw[thick] (1,-1)--(1,9);
\draw[thick] (-3,2)--(4,2);
\draw[thick] (-3,4)--(4,4);
\draw[thick] (-3,7)--(4,7);
\draw[thick,rounded corners,color=blue] (-1,-1)--(-1,6)--(4,6);
\draw[thick,rounded corners,color=Plum] (-3,0)--(-2,0)--(-2,5)--(0,5)--(0,8)--(4,8);
\end{tikzpicture}
\hspace{1em}\raisebox{3em}{$\rightarrow$}\hspace{1em}
\begin{tikzpicture}[x=1em,y=1em]
\draw[thick] (2,-1)--(2,9);
\draw[thick] (1,-1)--(1,9);
\draw[thick] (-3,2)--(4,2);
\draw[thick] (-3,4)--(4,4);
\draw[thick] (-3,7)--(4,7);
\draw[thick,rounded corners,color=blue] (-1,-1)--(-1,6)--(0,6)--(0,8)--(4,8);
\draw[thick,rounded corners,color=Plum] (-3,0)--(-2,0)--(-2,5)--(0,5)--(0,6)--(4,6);
\end{tikzpicture}
\hspace{1em}\raisebox{3em}{$\rightarrow$}\hspace{1em}
\begin{tikzpicture}[x=1em,y=1em]
\draw[thick] (2,-1)--(2,9);
\draw[thick] (1,-1)--(1,9);
\draw[thick] (-3,2)--(4,2);
\draw[thick] (-3,4)--(4,4);
\draw[thick] (-3,7)--(4,7);
\draw[thick,rounded corners,color=blue] (-1,-1)--(-1,8)--(4,8);
\draw[thick,rounded corners,color=Plum] (-3,0)--(3,0)--(3,6)--(4,6);
\end{tikzpicture}
\]
By Lemma~\ref{lemma:planarmoves}, the above moves all preserve the Demazure product.
At the level of ASMs, we are reduced to a negative entry with one less pivot.  Continue in this way until the negative entry has just two pivots.  At this point, we are reduced to the previous argument.

\noindent (2) From the proof of part (1), if $k=2$, the number of crossings is preserved.  For each additional pivot beyond the second, a crossing is removed (due to the application of one of the moves pictured in (\ref{eqn:planarmoves2})).
\end{proof}

\begin{theorem}
\label{theorem:demkey}
Fix $A\in \asm(n)$.  Then
 $\demprod{\asmtobpd(A)}=\mkey{A}$.  
\end{theorem}
\begin{proof}
By Lemma~\ref{lemma:permword}, $\demprod{\asmtobpd(w)}=w=\kappa(w)$. 

If $A\in \asm(n)$ is not a permutation matrix, then
there exists a sequence 
\[A= A^{(0)}\mapsto A^{(1)} \mapsto \cdots \mapsto A^{(|N(A)|)}=\mkey{A} \]
where $A^{(i)}\mapsto A^{(i+1)}$ is inflation of the leftmost $-1$ in the topmost row of $A^{(i)}$ with negative entries.  Furthermore, by definition, $\mkey{A}=\mkey{A^{(i)}}$ for all $i$.

By repeated application of Lemma~\ref{lemma:compatinflation}, $\demprod{\asmtobpd(A)}=\demprod{\asmtobpd(A^{(i)})}$ for all $i$.  Therefore, \[\demprod{\asmtobpd(A)}=\demprod{\asmtobpd(A^{(|N(A)|)})}=\mkey{A}. \qedhere\]
\end{proof}

Indeed, as long as a negative entry of $A$ is removable, it does not matter what order we apply inflation. 
\begin{corollary}
\label{cor:inflationorder}
If $A'$ is obtained from $A$ by applying inflation to a removable $-1$ then $\mkey{A}=\mkey{A'}$.
\end{corollary}
\begin{proof}
 By Lemma~\ref{lemma:compatinflation}, $\demprod{\asmtobpd(A)}=\demprod{\asmtobpd(A')}$.  Applying Theorem~\ref{theorem:demkey}, 
\[\mkey{A}=\demprod{\asmtobpd(A)}=\demprod{\asmtobpd(A')}=\mkey{A'}. \qedhere\]
\end{proof}
Corollary~\ref{cor:inflationorder} implies  \cite[Corollary~4]{Lascoux:ice}, which was stated without proof.

\subsection{A procedure for generating BPDs}
\label{section:deflate}

If $A\mapsto A'$ is an inflation, we call the replacement $A'\mapsto A$ \mydef{deflation}.  Notice deflation corresponds to selecting $(a,b)\in D(A)$ and some subset of the  pivots of $(a,b)$. 
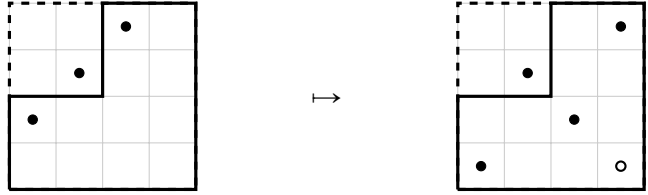
\begin{figure}
	\begin{center}
		\hspace{3em}
		\begin{tikzpicture}[x=1.5em,y=1.5em]
		\draw[step=1,gray!40, thin] (0,1) grid (4,5);
		\draw[color=black, very thick,dashed](0,1)rectangle(4,5);
		\draw[color=black,very thick](0,1)--(0,3)--(2,3)--(2,5)--(3,5)--(4,5)--(4,1)--(0,1);
		\filldraw [black](1.5,3.5)circle(.1);
		\filldraw [black](2.5,4.5)circle(.1);
		\filldraw [black](.5,2.5)circle(.1);
		\end{tikzpicture}
		\hspace{3em}
		\raisebox{2.75em}{$\mapsto$}
		\hspace{3em}
		\begin{tikzpicture}[x=1.5em,y=1.5em]
		\draw[step=1,gray!40, thin] (0,1) grid (4,5);
		\draw[color=black,very thick,dashed](0,1)rectangle(4,5);
		\draw[color=black,very thick](0,1)--(0,3)--(2,3)--(2,5)--(3,5)--(4,5)--(4,1)--(0,1);
		\filldraw [black](1.5,3.5)circle(.1);
		\filldraw [black](2.5,2.5)circle(.1);
		\filldraw [black](3.5,4.5)circle(.1);
		\filldraw [black](.5,1.5)circle(.1);
		\filldraw [color=black,fill=white,thick](3.5,1.5)circle(.1);
		\end{tikzpicture}
	\end{center}
\caption{Pictured above is deflation on an ASM, using the first and third pivots, (assuming the bottom right corner is in the diagram).}
\end{figure}
If $A\mapsto A'$ by deflation, then call the map $\asmtobpd(A)\mapsto \asmtobpd(A')$ deflation as well.
If  deflation uses exactly one pivot, we call it a \mydef{simple} deflation.  Interpreted in terms of bumpless pipe dreams, simple deflations are the \emph{droops} of \cite{Lam.Lee.Shimozono}.  See the example below.
\[
\begin{tikzpicture}[x=1.5em,y=1.5em]
\draw[step=1,gray!40, thin] (0,0) grid (5,4);
\draw[black, very thick,dashed] (0,0) rectangle (5,4);
\draw[color=black, fill=gray!40, thick](4,0)rectangle(5,1);
	\draw[thick,rounded corners, color=blue] (.5,0)--(.5,3.5)--(5,3.5);
\end{tikzpicture}
\hspace{3em}
		\raisebox{2.75em}{$\mapsto$}
		\hspace{3em}
\begin{tikzpicture}[x=1.5em,y=1.5em]
\draw[step=1,gray!40, thin] (0,0) grid (5,4);
\draw[black, very thick,dashed] (0,0) rectangle (5,4);
\draw[color=black, fill=gray!40, thick](0,3)rectangle(1,4);
	\draw[thick,rounded corners, color=blue] (.5,0)--(.5,.5)--(4.5,.5)--(4.5,3.5)--(5,3.5);
\end{tikzpicture}
\]
It is permissible to have more (unpictured) pipes in the rectangular region, but none of these may include elbow tiles.
  By \cite[Proposition~5.3]{Lam.Lee.Shimozono}, any  $\mathcal P\in\rpipes{w}$ can be reached from $\asmtobpd(w)$ by a sequence of droops.

We now introduce \mydef{K-theoretic droops}, which are moves of the form pictured below.
\[
\begin{tikzpicture}[x=1.5em,y=1.5em]
\draw[step=1,gray!40, thin] (0,0) grid (5,4);
\draw[black, very thick,dashed] (0,0) rectangle (5,4);
	\draw[thick,rounded corners, color=blue] (.5,0)--(.5,3.5)--(5,3.5);
	\draw[thick,rounded corners, color=blue] (2.5,0)--(2.5,.5)--(4.5,.5)--(4.5,4);
\end{tikzpicture}
\hspace{2em}
		\raisebox{2.75em}{$\mapsto$}
		\hspace{2em}
\begin{tikzpicture}[x=1.5em,y=1.5em]
\draw[step=1,gray!40, thin] (0,0) grid (5,4);
\draw[black, very thick,dashed] (0,0) rectangle (5,4);
\draw[color=black, fill=gray!40, thick](0,3)rectangle(1,4);
	\draw[thick,rounded corners, color=blue] (.5,0)--(.5,.5)--(4.5,.5)--(4.5,3.5)--(4.5,4);
	\draw[thick,rounded corners, color=blue] (2.5,0)--(2.5,3.5)--(5,3.5);
\end{tikzpicture}
\]
\[
\begin{tikzpicture}[x=1.5em,y=1.5em]
\draw[step=1,gray!40, thin] (0,0) grid (5,4);
\draw[black, very thick,dashed] (0,0) rectangle (5,4);
\draw[thick,rounded corners, color=blue] (.5,0)--(.5,3.5)--(5,3.5);	
\draw[thick,rounded corners, color=blue] (0,.5)--(4.5,.5)--(4.5,1.5)--(5,1.5);
\end{tikzpicture}
\hspace{2em}
		\raisebox{2.75em}{$\mapsto$}
		\hspace{2em}
\begin{tikzpicture}[x=1.5em,y=1.5em]
\draw[step=1,gray!40, thin] (0,0) grid (5,4);
\draw[black, very thick,dashed] (0,0) rectangle (5,4);
\draw[color=black, fill=gray!40, thick](0,3)rectangle(1,4);
\draw[thick,rounded corners, color=blue] (.5,0)--(.5,1.5)--(5,1.5);	
\draw[thick,rounded corners, color=blue] (0,.5)--(4.5,.5)--(4.5,3.5)--(5,3.5);
\end{tikzpicture}\]
Again, there may be other unpictured pipes, but these must not have elbow tiles within the pictured rectangle.

\begin{proposition}
\label{prop:droops}
The set $\pipes{w}$ is closed under applying droops and K-theoretic droops. Furthermore, any $\mathcal P\in \pipes{w}$ can be reached from $\asmtobpd(w)$ by applying a sequence of droops and K-theoretic droops.
\end{proposition}
\begin{proof}
That $\pipes{w}$ is closed under droops and K-theoretic droops follows from observing each of these moves can be realized as a composition of the moves from Lemma~\ref{lemma:planarmoves}.

We know by Theorem~\ref{theorem:demkey} and the definition of the key,  starting from $\mathcal P\in \pipes{w}$, there exists a sequence of inflations taking $\mathcal P$ to $\asmtobpd(w)$.  Suppose $\mathcal P\mapsto \mathcal P'$ is an inflation.  Suppose further that there was a sequence of droops and K-theoretic droops from $\asmtobpd(w)$ to $\mathcal P'$.

Notice that each deflation move on a BPD can be realized as a composition of a droop, combined with some number of K-theoretic droops.  Simply droop the first pivot (from the top) involved in deflation, then apply K-theoretic droops to the rest in order. 
\end{proof}

We conclude with a lemma.
Given $\mathcal P\in {\sf BPD}(n)$,  write \[\Mute(\mathcal P)=\bigcup_{(i,j)\in D(\mathcal P)}[1,i]\times[1,j]\] for the {\bf mutable region} of $\mathcal P$. We write $\lambda^{(w)}:=\Mute(\asmtobpd(w))$.

\begin{lemma}
	\label{lemma:muteable}
	Suppose $\mathcal P'$ is obtained from $\mathcal P$ by deflation.  Then:
	\begin{enumerate}
		\item $\Mute(\mathcal P')\subseteq \Mute(\mathcal P)$.
		\item $\mathcal P_{i\,j}=\mathcal P'_{i\,j}$ for all $(i,j) \not \in \Mute(\mathcal P)$.
		\item Given $\mathcal P\in {\sf Pipes}(w)$, we have $\mathcal P_{i\,j}=\asmtobpd(w)_{i\,j}$ for all $(i,j)\not \in \yd{\lambda^{(w)}}$.
	\end{enumerate}
\end{lemma}
\begin{proof}
\noindent(1) By definition, it is immediate  that deflation of pivots does not create any diagram boxes outside of $\Mute(\mathcal P)$.

\noindent(2) This is again clear, since deflation takes place entirely within $\Mute(\mathcal P)$.

\noindent(3) Since $\mathcal P\in \pipes{w}$, we can reach $\mathcal P$ from $\asmtobpd(w)$ by a sequence of deflation moves.  Thus, the claim follows by repeated application of (2).
\end{proof}

\subsection{Bumpless pipe dreams supported on a partition}

Fix a partition $\lambda$.  We write $\bpd{\lambda}$ for the set of tilings of $\yd{\lambda}$ with the tiles pictured in (\ref{eqn:sixtiles}) so that each pipe:
\begin{enumerate}
\item starts in the bottommost cell of a column of $\yd{\lambda}$ and 
\item ends in the rightmost cell of a row of $\yd{\lambda}$.
\end{enumerate}
Note that there is no requirement on the number of pipes which appear in the tiling.  We define sets $D(\mathcal P)$ and $U(\mathcal P)$ in the same way as before.

If $\mathcal P\in \bpd{n}$ and $\yd{\lambda}\subseteq [n]\times[n]$,  restricting $\mathcal P$ to $\yd{\lambda}$ produces a well-defined element of $\bpd{\lambda}$.  We write $\restrictBPD_\lambda$ for this map.
Starting with $\mathcal P\in \bpd{\lambda}$, we would like to complete it to a square BPD in a way which preserves $D(\mathcal P)$ and $U(\mathcal P)$.

\begin{lemma}
Fix a partition $\lambda\subseteq [m]\times [n]$.  Take $\mathcal P\in \bpd{\lambda}$ and suppose there are $k$ pipes in $\mathcal P$.  There exists a unique $\widetilde{\mathcal P}\in \bpd{m+n}$ so that $\restrictBPD_\lambda(\widetilde{\mathcal P})=\mathcal P$ and $D(\widetilde{\mathcal P}),U(\widetilde{\mathcal P})\subseteq \yd{\lambda}$. 
\end{lemma}

\begin{proof}
We start by constructing $\widetilde{\mathcal P}$.  Embed $\mathcal P$ in the $(m+n)\times (m+n)$ grid so that it is top-left justified.  Extend all of the horizontal and vertical pipes at the boundary edge of $\yd{\lambda}$ so that they hit the right and bottom edges of the grid.

Let $i$ be the index of the first row (from the top) which has a blank tile outside of $\yd{\lambda}$.  Let $j$ be the index of the left-most tile in  row $i$ so that $(i,j)\not \in\yd{\lambda}$. In this cell, place a downward elbow tile.  The only tiles to the right of position $(i,j)$ are blank or vertical.  Thus, we may extend the downward elbow to the right edge of the grid.  Likewise, the only tiles below position $(i,j)$ are blank or horizontal.  So we extend the pipe to reach the bottom edge of the grid.

Continue to the next row which has a blank tile outside of $\yd{\lambda}$ and repeat the procedure given above.  This construction results in a tiling which has no blank tiles or upward elbows outside of $\yd{\lambda}$.  Furthermore, there are $m+$ pipes, each connecting the bottom row of the grid to the right column.  Finally, a downward elbow was placed outside of $\yd{\lambda}$ if and only if there was no pipe which exited $\lambda$  in that row.  Thus, a pipe exits every row.  From this, we see the construction produced a well-defined element of $\bpd{m+n}$.

Uniqueness follows from the bijection between $\asm(m+n)$ and $\bpd{m+n}$ combined with the fact that ASMs are determined by the restriction of $r_A$ to the southeast most corners of each connected component of $N(A)\cup D(A)$ (see \cite[Proposition~3.11]{weigandt2017prism}).
\end{proof}

Often, this construction produces a BPD which larger than necessary.  If the bottom right corner is a downward elbow, we are free to remove the last row and column to produce a well-defined BPD of one size smaller.  With this in mind, we define $\completebpd(\mathcal P)$ to be the smallest possible square BPD which restricts to $\mathcal P$.  See Figure~\ref{figure:partitionbpd} for an example.

\begin{figure}
\[\begin{tikzpicture}[x=1.5em,y=1.5em]
\draw[color=white, thick](0,1)rectangle(4,5);
\draw[step=1,gray, thin] (0,3) grid (3,5);
\draw[step=1,gray, thin] (0,2) grid (2,5);
\draw[color=black,very thick] (0,2)--(0,5)--(3,5)--(3,3)--(2,3)--(2,2)--(0,2)--(0,3);
\draw[thick,color=blue,rounded corners] (1.5,2)--(1.5,2.5)--(1.5,3.5)--(2.5,3.5)--(2.5,4.5)--(3,4.5);
\draw[thick,color=blue,rounded corners](.5,2)--(.5,2.5)--(2,2.5);
\end{tikzpicture}
\hspace{3em}
\begin{tikzpicture}[x=1.5em,y=1.5em]
\draw[step=1,gray, thin] (0,1) grid (4,5);
\draw[color=black, thick](0,1)rectangle(4,5);
\draw[thick,color=blue,rounded corners] (1.5,1)--(1.5,2.5)--(1.5,3.5)--(2.5,3.5)--(2.5,4.5)--(4,4.5);
\draw[thick,color=blue,rounded corners](.5,1)--(.5,2.5)--(4,2.5);
\draw[thick,color=blue,rounded corners](2.5,1)--(2.5,1.5)--(4,1.5);
\draw[thick,color=blue,rounded corners](3.5,1)--(3.5,3.5)--(4,3.5);
\end{tikzpicture}
\]
\caption{ Let $\lambda=(3,3,2)$.  Pictured  on the left an element  $\mathcal P\in\bpd{\lambda}$ and on the right is $\completebpd(\mathcal P)$.}
\label{figure:partitionbpd}
\end{figure}
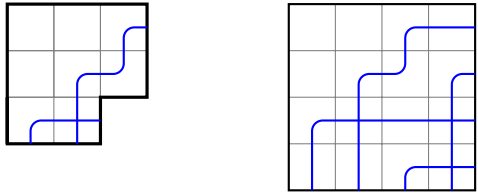

\begin{lemma}
\label{lemma:compatcompletion}
Take $\mathcal P\in \bpd{\lambda}$.  If  $\completebpd(P)\in \pipes{w}$, then \[\pipes{w}\subseteq \{\completebpd(\mathcal Q):\mathcal Q\in \bpd{\lambda}\}.\]
\end{lemma}
\begin{proof}
Take $\mathcal P\in \bpd{\lambda}$.  Let $w=\demprod{\completebpd(\mathcal P)}$.  Notice, by construction, the tiles outside of $\yd{\lambda}$ in $\completebpd(\mathcal P)$ only depend on which pipes exit the southeast boundary of $\yd{\lambda}$ in $\mathcal P$.  In particular, if $\mathcal P'\in \bpd{\lambda}$ is obtained from $\mathcal P$ by applying inflation (or deflation), their completions agree outside of $\yd{\lambda}$.

By construction, $U(\completebpd(\mathcal P))\subseteq \yd{\lambda}$.
This means any inflation on $\completebpd(\mathcal P)$ happens entirely within $\yd{\lambda}$.  As such, we apply a sequence of inflations to $\completebpd(\mathcal P)$ to reach $\asmtobpd(w)$.  Furthermore,  each replacement along the way is confined to $\yd{\lambda}$.  Thus, we see $\yd{\lambda^{(w)}}\subseteq \yd{\lambda}$ and the result follows.
\end{proof}

\section{Transition on bumpless pipe dreams}
\label{section:transitionBPD}
This section follows the techniques outlined in \cite{Lascoux:ice}, translated to the language of bumpless pipe dreams.  We provide additional details to make the arguments self-contained.  Throughout this section, we adopt the notation from Section~\ref{section:introtransition}.

\subsection{Preliminaries}
Given $\mathcal P\in \pipes{w}$, define $\mc(\mathcal P)=\mc(w)$.  Notice $\mc(\asmtobpd({\rm id}))$ is not defined.
  Define a map $\asmtransition:\bpd{n}\rightarrow \bpd{n}$ as follows.
Let $\asmtransition(\id)=\id$. Otherwise, write  $(a,b):=\mc(\mathcal P)$. We know since $\pipes{w}$ is connected by sequences of deflations, $(a,b)$ is a blank tile or an upward elbow.  If it is a blank tile, make the replacements pictured below involving the blue pipes in the rectangle $[a,w^{-1}(b)]\times [b,w(a)]$: 
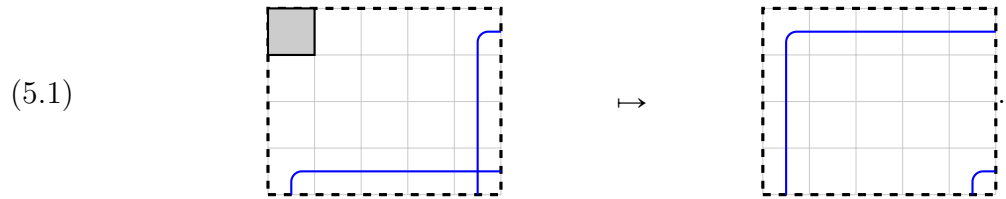
\begin{equation}
\label{eq:transitionpicture1}
\raisebox{-3em}{
\begin{tikzpicture}[x=1.5em,y=1.5em]
\draw[step=1,gray!40, thin] (0,0) grid (5,4);
\draw[black, very thick,dashed] (0,0) rectangle (5,4);
\draw[color=black, fill=gray!40, thick](0,3)rectangle(1,4);
\draw[thick,rounded corners,color=blue] (.5,0)--(.5,.5)--(5,.5);
\draw[thick,rounded corners,color=blue] (4.5,0)--(4.5,3.5)--(5,3.5);
\end{tikzpicture}
\hspace{3em}
		\raisebox{2.75em}{$\mapsto$}
		\hspace{3em}
\begin{tikzpicture}[x=1.5em,y=1.5em]
\draw[step=1,gray!40, thin] (0,0) grid (5,4);
\draw[black, very thick,dashed] (0,0) rectangle (5,4);
	\draw[thick,rounded corners,color=blue] (.5,0)--(.5,3.5)--(5,3.5);
	\draw[thick,rounded corners,color=blue] (4.5,0)--(4.5,.5)--(5,.5);
\end{tikzpicture}}.
\end{equation}
If $\mc(\mathcal P)$ is an upward elbow, make the following replacements:
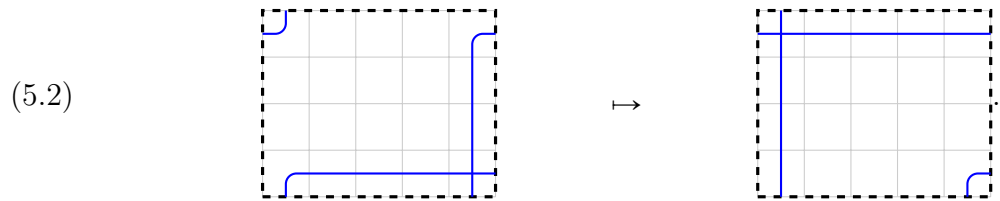
\begin{equation}
\label{eq:transitionpicture2}
\raisebox{-3em}{\begin{tikzpicture}[x=1.5em,y=1.5em]
\draw[step=1,gray!40, thin] (0,0) grid (5,4);
\draw[black, very thick,dashed] (0,0) rectangle (5,4);
	\draw[thick,rounded corners,color=blue] (0,3.5)--(.5,3.5)--(.5,4);
	\draw[thick,rounded corners,color=blue] (.5,0)--(.5,.5)--(5,.5);
\draw[thick,rounded corners,color=blue] (4.5,0)--(4.5,3.5)--(5,3.5);
\end{tikzpicture}
\hspace{3em}
		\raisebox{2.75em}{$\mapsto$}
		\hspace{3em}
\begin{tikzpicture}[x=1.5em,y=1.5em]
\draw[step=1,gray!40, thin] (0,0) grid (5,4);
\draw[black, very thick,dashed] (0,0) rectangle (5,4);
	\draw[thick,rounded corners,color=blue] (.5,0)--(.5,4);
	\draw[thick,rounded corners,color=blue] (0,3.5)--(5,3.5);	
\draw[thick,rounded corners,color=blue] (4.5,0)--(4.5,.5)--(5,.5);
\end{tikzpicture}}.
\end{equation}
All other pipes in the region (not pictured in the diagrams) remain the same.

\begin{lemma}
\label{lemma:transitiondeflation}
Take $w\in \SymGp_n$.  Fix $I\subseteq \phi(w)$. Suppose $\mathcal P$ is obtained from $\asmtobpd(w)$ by deflating the pivots of $\mc(w)$ which sit in the rows of $I$.  Then $\asmtransition(\mathcal P)=\asmtobpd(w_I)$.  Furthermore, $\ell(w_I)=\ell(w)+|I|-1$.
\end{lemma}

\begin{proof}
Write $\mc(w)=(a,b)$.

If $I=\emptyset$ then no deflation occurred.  As such, we are in the situation of $(\ref{eq:transitionpicture1})$.  At the level of permutation matrices, we are simply swapping rows to remove the inversion $\mc(w)$ and so $\asmtobpd(w_\emptyset)=\asmtransition(\mathcal P)$.

Now suppose $I\neq\emptyset$. 
 Write $I=\{i_1,i_2,\ldots,i_k\}$ with $i_1<i_2<\cdots<i_k$.  Since $\mathcal P$ was obtained from $\asmtobpd(w)$ by deflation of the pivots in $I$, we know $\mathcal P$ has downward elbows in positions 
\begin{equation}
\label{eqn:deflatepivots}
(i_1,b), \, (i_2,w(i_1)),\,\ldots,\,(i_{k},w(i_{k-1})),\,(a,w(i_k)).
\end{equation}
Furthermore, $\mathcal P$ has a unique upward elbow tile which sits in $\mc(w)$.  This is removed when applying  $\asmtransition$.  So $\asmtransition(\mathcal P)$ is a Rothe BPD.

From the description in (\ref{eqn:deflatepivots}), we confirm $\asmtransition(\mathcal P)=\asmtobpd(w_\emptyset\cdot c_I^{(a)})=\asmtobpd(w_I)$.
That $\ell(w_I)=\ell(w)+|I|-1$ is a consequence of Lemma~\ref{lemma:compatinflation}.
\end{proof}

Below, we translate \cite[Proposition~7]{Lascoux:ice} into the language of bumpless pipe dreams.  We include a proof for completeness.
\begin{proposition}\label{prop:asmtransition}
Fix a permutation $w\neq \id$. The restriction of $\asmtransition$ to $\pipes{w}$ defines a bijection
\[\widetilde{\asmtransition}:\pipes{w}\rightarrow \bigcup_{I\subseteq \phi(w)}\pipes{w_I}.\] 
\end{proposition}
\begin{proof}

Let $(a,b):=\mc(w)$.
Given $I\subseteq\phi(w)$, let $\mathcal P_I$ be the BPD obtained from $\asmtobpd(w)$ by applying deflation to the pivots of $\mc(w)$ in $\asmtobpd(w)$ which sit in the rows of $I$.

\noindent {\bf Claim:}  $\widetilde{\asmtransition}$ is injective.

Fix $\mathcal P\in \pipes{w}$.
The only cell of $\mathcal P$ in $\yd{\lambda^{(w)}}$ affected by the map $\asmtransition$ is $\mc(w)$. Thus, knowing the result of applying $\widetilde{\asmtransition}$ to $\mathcal P$ uniquely determines $\mathcal P$.

\noindent {\bf Claim}: $\widetilde{\asmtransition}$ is well-defined.

Fix $\mathcal P\in \pipes{w}$.  We know there exists a sequence of deflations \[\mathcal P=\mathcal P^{(0)}\mapsto \cdots \mapsto \mathcal P^{(k)}\mapsto \asmtobpd(w).\]  We may assume, without loss of generality, that $\mathcal P^{(k)}=\mathcal P_I$ for some $I\subseteq\phi(w)$.  (In the case $I=\emptyset$,  $\mathcal P^{(k)}=\asmtobpd(w)$.)
Furthermore, the sequence
\[\asmtransition(\mathcal P^{(0)})\mapsto \cdots \mapsto \asmtransition(\mathcal P^{(k)})\]
consists of the \emph{same} deflation moves, since the transition replacements happen in a disjoint region of the grid from the deflations.
By Lemma~\ref{lemma:transitiondeflation}, $\asmtransition(\mathcal P^{(k)})=\asmtransition(\mathcal P_I)=\asmtobpd(w_I)$.  Thus, $\asmtransition(\mathcal P)\in \pipes{w_I}$.

\noindent {\bf Claim}: $\widetilde{\asmtransition}$ is surjective.

Take $\mathcal P\in \pipes{w_I}$.  Then we know we can find a sequence of deflations so that \[\mathcal P=\mathcal P^{(0)}\mapsto \cdots \mapsto \mathcal P^{(k)}\mapsto \asmtobpd(w_I).\]

By Lemma~\ref{lemma:transitiondeflation}, we know $\widetilde{\asmtransition}(\mathcal P_I)=\asmtobpd(w_I)$. Furthermore, by Lemma~\ref{lemma:muteable}, $\mathcal P^{(i)}$ agrees with $\asmtobpd(w_I)$ on all tiles outside of $\yd{\lambda^{(w_I)}}$.  Then, we may make the same replacements in the rectangle $[a,w^{-1}(b)]\times[b,w(a)]$ to get from $\mathcal P^{(i)}$ to some $\mathcal P'^{(i)}$ as one does to go from $\asmtobpd(w_I)$ to $\mathcal P_I$.  This produces a valid BPD.  

We now claim $\mathcal P'\in \pipes{w}$.  To see this, notice 
 \[\mathcal P'=\mathcal P'^{(0)}\mapsto \cdots \mapsto \mathcal P'^{(k)}\mapsto \mathcal P_I\] is a valid sequence of deflations.  Since $\mathcal P_I\in \pipes{w}$, so is $\mathcal P'$.  
Furthermore, it is immediate that $\widetilde{\asmtransition}(\mathcal P')=\mathcal P$ and so we are done.
\end{proof}

\begin{lemma}
\label{lemma:bpdtransitionweights}
Fix  $w\in \SymGp_n$ with $\mc(w)=(a,b)$ and take $\mathcal P\in \pipes{w}$.  Then,
\[\wt{\mathcal P}=
\begin{cases} 
\beta(x_a\oplus y_b)\wt{\asmtransition(\mathcal P)}&\text{if}\enspace \asmtransition(\mathcal P)\in \bpd{w_\emptyset}, \enspace \text{and}\\
\left(1+\beta(x_a\oplus y_b)\right)\wt{\asmtransition(\mathcal P)}&\text{otherwise}.
\end{cases}\]
\end{lemma}

\begin{proof}
By the argument in the proof of Proposition~\ref{prop:asmtransition}, we see that $\asmtransition(\mathcal P)\in \pipes{w_\emptyset}$ if and only if $\mc(\mathcal P)$ is a blank tile.  So the result follows from (\ref{eq:transitionpicture1}) and (\ref{eq:transitionpicture2}) and the definition of the weights of BPDs.
\end{proof}

\subsection{Proof of Theorem~\ref{thm:main}}
\label{section:pfmain}
We now prove our main theorem.
\begin{proof}
The formula is trivial to verify for  the case $\mc(w)=(1,1)$. Fix $w\in \SymGp_n$ with $\mc(w)=(a,b)$.  Assume the bumpless pipe dream formula holds for all permutations $v$ so that $\mc(v)<\mc(w)$ in lexicographical order.

We have for all $I\subseteq \phi(w)$ that $\mc(w_I)<\mc(w)$.  
By Theorem~\ref{thm:transition},
\[\Groth^{(\beta)}_w(\mathbf x;\mathbf y)=(x_a\oplus y_b)\Groth^{(\beta)}_{w_\emptyset}(\mathbf x;\mathbf y)+(1+\beta( x_a\oplus y_b))\sum_{I\subseteq \phi(w):I\neq\emptyset}\beta^{|I|-1}\Groth^{(\beta)}_{w_I}(\mathbf x;\mathbf y).\]
Applying the inductive hypothesis,
\begin{align*}
\Groth^{(\beta)}_w(\mathbf x;\mathbf y)
&=(x_a\oplus y_b)\left(\beta^{-\ell(w_\emptyset)}\sum_{\mathcal P\in \pipes{w_\emptyset}}\wt{\mathcal P}\right)\\
&\quad+(1+\beta( x_a\oplus y_b))\sum_{I\subseteq \phi(w):I\neq\emptyset}\beta^{|I|-1}\left(\beta^{-\ell(w_I)}\sum_{\mathcal P\in \pipes{w_I}}\wt{\mathcal P}\right)\\
&=\beta^{-\ell(w)} \sum_{\mathcal P\in \pipes{w_\emptyset}}\beta(x_a\oplus y_b)\wt{\mathcal P}  \\
&\quad+\beta^{-\ell(w)}\sum_{I\subseteq \phi(w):I\neq\emptyset}\sum_{\mathcal P\in \pipes{w_I}}(1+\beta( x_a\oplus y_b))\wt{\mathcal P}\\
&=\beta^{-\ell(w)}\sum_{\mathcal P\in \pipes{w}}\wt{\mathcal P}.
\end{align*}
The last equality is a  consequence of Proposition~\ref{prop:asmtransition} and Lemma~\ref{lemma:bpdtransitionweights}.
\end{proof}

\subsection{Specialization  to double Schubert polynomials}

We conclude by giving a new proof of the Lam-Lee-Shimozono formula for double Schubert polynomials.
\begin{proof}[Proof of Theorem~\ref{theorem:bumplessSchubert}]
Fix $\mathcal P\in \pipes{w}$.  Then 
\begin{align*}
\beta^{-\ell(w)}\,\wt{\mathcal P}&=\beta^{-\ell(w)}\left( \prod_{(i,j)\in D(\mathcal P)}\beta(x_i\oplus y_j)\right) \left(\prod_{(i,j)\in U(\mathcal P)} 1+\beta(x_i\oplus y_j) \right)\\
&=\beta^{|D(\mathcal P)|-\ell(w)}\left( \prod_{(i,j)\in D(\mathcal P)}(x_i\oplus y_j)\right) \left(\prod_{(i,j)\in U(\mathcal P)} 1+\beta(x_i\oplus y_j) \right).
\end{align*}
We know by Lemma~\ref{lemma:uniquelydet} that $|D(\mathcal P)|=|\mathbf a_{\mathcal P}|\geq\ell(w)$.  Therefore,
\[\left(\beta^{-\ell(w)}\,\wt{\mathcal P}\right)|_{\beta=0}=
\begin{cases}
 \displaystyle \prod_{(i,j)\in D(\mathcal P)}(x_i+y_j)&\text{if} \enspace |D(\mathcal P)|=\ell(w), \enspace \text{and}\\
0&\text{otherwise}.
\end{cases}
\]
Since $|D(\mathcal P)|=|\mathbf a_{\mathcal P}|$, we know $|D(\mathcal P)|=\ell(w)$ if and only if $\mathbf a_{\mathcal P}$ is a reduced word, i.e.,\ $\mathcal P$ is a reduced bumpless pipe dream.  Thus,
\[\Schub_w(\mathbf x;\mathbf y)=\Groth^{(0)}_w(\mathbf x;-\mathbf y)=\sum_{\mathcal P\in \rpipes{w}}\prod_{(i,j)\in D(\mathcal P)}(x_i-y_j). \qedhere\]
\end{proof}

\section{Comparisons between ordinary pipe dreams and bumpless pipe dreams}
\label{section:comparisons}

We now discuss connections between BPDs and  pipe dreams which appeared earlier in the literature (see \cite{Bergeron.Billey}, \cite{Fomin.Kirillov}, \cite{Knutson.Miller}). We call these objects ordinary pipe dreams to distinguish them from BPDs. 

\subsection{Ordinary pipe dreams}

Write  $\delta^{(n)}=(n-1,\, n-2,\, \ldots,\, 1)$ for the \mydef{staircase partition}.
An \mydef{ordinary pipe dream} of size $n$ is a tiling of $\yd{\delta^{(n)}}$ with the tiles pictured below.
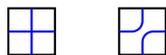
\[	\begin{tikzpicture}[x=1.5em,y=1.5em]
	\draw[color=black, thick](0,1)rectangle(1,2);
	\draw[thick,rounded corners,color=blue] (0,1.5)--(1,1.5);
	\draw[thick,rounded corners,color=blue] (.5,1)--(.5,2);
	\end{tikzpicture}
\hspace{2em} 
	\begin{tikzpicture}[x=1.5em,y=1.5em]
	\draw[color=black, thick](0,1)rectangle(1,2);
	\draw[thick,rounded corners, color=blue] (.5,2)--(.5,1.5)--(0,1.5);
	\draw[thick,rounded corners, color=blue] (1,1.5)--(.5,1.5)--(.5,1);
	\end{tikzpicture} 
\]
We form a planar history by placing upward elbow tiles in each of the cells \[\{(n,1),\,(n-1,2),\,\ldots,\,(1,n)\}.\]  Any such tiling produces a network of $n$ pipes.  Each pipe  starts at the left edge of the grid and ends at the top.  We can define a reading word and Demazure product in an entirely analogous way as we did for bumpless pipe dreams.  This time read crossings along rows from right to left, starting at the top and working down.  An ordinary pipe dream is reduced if pairwise pipes cross at most one time. We refer the reader to \cite{Knutson.Miller.subword} for further details.

We identify each ordinary pipe dream with a subset $\mathcal  S\subseteq \yd{\delta^{(n)}}$  by recording the positions of the crossing tiles.  Write $\demprod{\mathcal S}$ for the Demazure product of the reading word of $\mathcal S$.  Define \[\wt{\mathcal S}=\prod_{(i,j)\in \mathcal S}\beta(x_i\oplus y_j).\]
With this weight, we recall the following theorem.
\begin{theorem}[{\cite{Fomin.Kirillov,Knutson.Miller.subword}}]
\label{thm:ordinarypipe}
The $\beta$-double Grothendieck polynomial is the sum
\[\mathfrak G^{(\beta)}_w(\mathbf x;\mathbf y)=\beta^{-\ell(w)}\sum_{\mathcal S:\demprod{\mathcal S}=w}\wt{\mathcal S}.\]
\end{theorem}
Again, specializing $\beta=0$ and substituting $\mathbf y\mapsto -\mathbf y$ amounts to summing over reduced ordinary pipe dreams with the usual weights.  This gives a formula for double Schubert polynomials.

\subsection{Bijections from ordinary pipe dreams to bumpless pipe dreams}

In light of Theorem~\ref{thm:main} and Theorem~\ref{thm:ordinarypipe}, we have two different sets of objects whose weight generating functions produce the same polynomial. In particular:
\begin{proposition}
Marked BPDs for $w$ are equinumerous with ordinary pipe dreams for $w$.
\end{proposition}
\begin{proof}
Consider  $\mathfrak G_w^{(1)}(\mathbf 1;\mathbf 0)$.  Under this specialization, each marked BPD and ordinary BPD has a weight of $1$.  Applying  Theorem~\ref{thm:main} and Theorem~\ref{thm:ordinarypipe}, we see immediately that $\mathfrak G_w^{(1)}(\mathbf 1;\mathbf 0)=|\mpipes{w}|=|\{S: \demprod{\mathcal S}=w\}|.$
\end{proof}
   Indeed, one could hope for a weight preserving bijection which connects marked BPDs to ordinary pipe dreams. 
  In the dominant case there is such a bijection, though for trivial reasons.  When $w$ is dominant, 
\begin{equation}
\mathfrak G^{(\beta)}_w(\mathbf x;\mathbf y)=\prod_{(i,j)\in D(w)}(x_i\oplus y_j).
\end{equation}
 In particular, there is a single ordinary pipe dream and a single marked BPD, both of which have the same weight.
Outside of the dominant case, there are clear obstructions to the existence of a general weight preserving bijection, illustrated by the following example.

\begin{example}
\label{ex:132}
 The elements of $\mpipes{132}$ are pictured below.
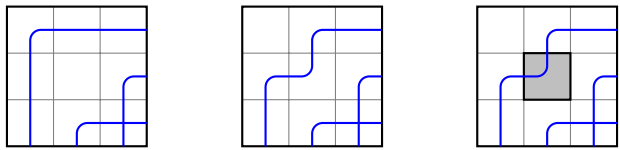
\[\begin{tikzpicture}[x=1.5em,y=1.5em]
	\draw[step=1,gray, thin] (0,0) grid (3,3);
	\draw[color=black, thick](0,0)rectangle(3,3);
	\draw[thick,rounded corners,color=blue] (.5,0)--(.5,2.5)--(3,2.5);
	\draw[thick,rounded corners,color=blue] (1.5,0)--(1.5,.5)--(3,.5);
	\draw[thick,rounded corners,color=blue] (2.5,0)--(2.5,1.5)--(3,1.5);
	\end{tikzpicture}
\hspace{3em}
	\begin{tikzpicture}[x=1.5em,y=1.5em]
	\draw[step=1,gray, thin] (0,0) grid (3,3);
	\draw[color=black, thick](0,0)rectangle(3,3);
	\draw[thick,rounded corners,color=blue] (.5,0)--(.5,1.5)--(1.5,1.5)--(1.5,2.5)--(3,2.5);
\draw[thick,rounded corners,color=blue] (1.5,0)--(1.5,.5)--(3,.5);
\draw[thick,rounded corners,color=blue] (2.5,0)--(2.5,1.5)--(3,1.5);
	\end{tikzpicture}
\hspace{3em}
	\begin{tikzpicture}[x=1.5em,y=1.5em]
	\draw[step=1,gray, thin] (0,0) grid (3,3);
\draw[color=black,fill=lightgray, thick](1,1)rectangle(2,2);
	\draw[color=black, thick](0,0)rectangle(3,3);
	\draw[thick,rounded corners,color=blue] (.5,0)--(.5,1.5)--(1.5,1.5)--(1.5,2.5)--(3,2.5);
\draw[thick,rounded corners,color=blue] (1.5,0)--(1.5,.5)--(3,.5);
\draw[thick,rounded corners,color=blue] (2.5,0)--(2.5,1.5)--(3,1.5);

	\end{tikzpicture}\]
We compare these to the ordinary pipe dreams for 132.
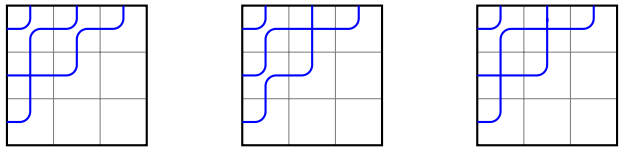
\[\begin{tikzpicture}[x=1.5em,y=1.5em]
	\draw[step=1,gray, thin] (0,0) grid (3,3);
	\draw[color=black, thick](0,0)rectangle(3,3);
	\draw[thick,rounded corners,color=blue] (0,2.5)--(.5,2.5)--(.5,3);
\draw[thick,rounded corners,color=blue] (0,1.5)--(1.5,1.5)--(1.5,2.5)--(2.5,2.5)--(2.5,3);
\draw[thick,rounded corners,color=blue] (0,.5)--(.5,.5)--(.5,2.5)--(1.5,2.5)--(1.5,3);
	\end{tikzpicture}
\hspace{3em}
	\begin{tikzpicture}[x=1.5em,y=1.5em]
	\draw[step=1,gray, thin] (0,0) grid (3,3);
	\draw[color=black, thick](0,0)rectangle(3,3);
\draw[thick,rounded corners,color=blue] (0,2.5)--(.5,2.5)--(.5,3);
\draw[thick,rounded corners,color=blue] (0,1.5)--(.5,1.5)--(.5,2.5)--(2.5,2.5)--(2.5,3);
\draw[thick,rounded corners,color=blue] (0,.5)--(.5,.5)--(.5,1.5)--(1.5,1.5)--(1.5,3);
	\end{tikzpicture}
\hspace{3em}
	\begin{tikzpicture}[x=1.5em,y=1.5em]
	\draw[step=1,gray, thin] (0,0) grid (3,3);
	\draw[color=black, thick](0,0)rectangle(3,3);
\draw[thick,rounded corners,color=blue] (0,2.5)--(.5,2.5)--(.5,3);
\draw[thick,rounded corners,color=blue] (0,1.5)--(1.5,1.5)--(1.5,2.5)--(1.5,2.5)--(1.5,3);
\draw[thick,rounded corners,color=blue] (0,.5)--(.5,.5)--(.5,2.5)--(2.5,2.5)--(2.5,3);
	\end{tikzpicture}\]
We have ordered each set of pipe dreams so that the $\mathbf x$ weights are the same from left to right.  However, there is no bijection which preserves both $\mathbf x$ and $\mathbf y$ weights simultaneously.
\end{example}

Adjusting expectations, once could hope for a direct bijection which respects the single $\mathbf x$ weights (with the $\mathbf y$ variables specialized to 0).  Even this seems difficult. Lenart, Robinson, and Sottile  discussed the problem in \cite{Lenart.Robinson.Sottile} saying:
\begin{quote}
There is one construction of double Grothendieck polynomials which we do
not know how to relate to the classical construction in terms of RC-graphs. This is due to Lascoux and involves certain alternating sign matrices. Our attempts to do so suggest that any bijective relation between these formulas will be quite subtle.
\end{quote}
Knutson \cite{Knutson:cotransition} attributes this difficulty to the fact that, while ordinary pipe dreams are naturally compatible with the \emph{co-transition} formula:
\begin{quote}
...one might expect it to be very difficult to connect the pipe dream formula to the transition formula, requiring ``Little bumping algorithms'' and the like, and essentially impossible if one wants to include the $y$ variables. 
\end{quote}
Even if it is too much to hope for a bijective understanding of the double formulas, a direct bijection from ordinary pipe dreams to marked bumpless pipe dreams which preserves the $\mathbf x$ weights would be highly interesting.

We note that there are bijections in the literature between (ordinary) reduced pipe dreams for vexillary permutations and flagged tableaux. A proof for permutations of the form $1^k\times w_0$ is found in work of Serrano and Stump (see \cite[Section~3.1]{Serrano.Stump}).  More generally, Lenart stated a bijection (without proof) which holds for all vexillary permutations (see \cite[Remark~4.12]{Lenart}). These, combined with Theorem~\ref{thm:ssyttovexbpdbij}, produce a bijection from reduced pipe dreams for $v$ to $\pipes{v}$, when $v$ is vexillary.  Furthermore, this bijection preserves the $\mathbf x$ weights.

\section{Vexillary Grothendieck polynomials}
\label{s:vex}

In this section, we restrict our attention to vexillary permutations.  There are many interconnected combinatorial formulas in the literature for expressing vexillary Schubert and Grothendieck polynomials.  Notably, the flagged set-valued tableaux of Knutson, Miller, and Yong \cite{Knutson.Miller.Yong} provide a link between the work of Wachs \cite{Wachs} on flagged tableaux and Buch \cite{Buch} on set-valued tableaux.  Knutson, Miller, and Yong also gave geometric meaning to these tableaux by describing a weight preserving bijection with \emph{diagonal pipe dreams}.  Reduced diagonal pipe dreams  naturally label irreducible components of  diagonal Gr\"obner degenerations of vexillary \emph{matrix Schubert varieties}.

Diagonal pipe dreams are closely related to the excited Young diagrams of \cite{Ikeda.Naruse}.  In the reduced case, they are complements of the families of non-intersecting lattice paths of  Kreiman \cite{Kreiman}.  We will see that these lattice paths are in transparent bijection with vexillary bumpless pipe dreams.  We make this connection explicit in Section~\ref{section:paths}.  We use Kreiman's lattice paths as a bridge between set-valued tableaux and vexillary bumpless pipe dreams.  Kreiman's choice of flagging is different from the one we use here, but equivalent.  We refer the reader to \cite[Section~3.3]{Morales.Pak.Panova.I}, in particular, Remark~3.10.  See also   \cite[Section~3]{Morales.Pak.Panova.II}.

\subsection{Vexillary Grothendieck polynomials}

If $v$ is vexillary, write $\mu^{(v)}$ for the partition obtained by sorting the entries of $\code_v$ to form a weakly decreasing sequence.  Equivalently, $\mu^{(v)}$ is the (unique) partition which has as many cells in each diagonal of $\yd{\mu^{(v)}}$ as there are elements of $D(v)$ in the same diagonal (see Lemma~\ref{lemma:diagonalswork}).

 Recall, $\lambda^{(v)}$ is the smallest partition so that $D(v)\subseteq \yd{\lambda^{(v)}}$.  
We define the flag $\mathbf f^{(v)}$ by setting $f_i^{(v)}$ to be the maximum row index $j$ so that $(j,j-i+\mu^{(v)}_i)\in \yd{\lambda^{(v)}}.$  In other words, this is row of the southeast most cell in $\yd{\lambda^{(v)}}$ which sits in the same diagonal as $(i,\mu^{(v)}_i)$.

Abbreviate $\FSSYT(v):=\FSSYT(\mu^{(v)},\mathbf f^{(v)})$ and $\FSet(v):=\FSet(\mu^{(v)},\mathbf f^{(v)})$. 
We weight a set-valued tableau as follows:
\[\wt{\mathbf T}=\beta^{|\mathbf T|-|\lambda|}\prod_{(i,j)\in \yd{\lambda}}\prod_{k\in \mathbf T(i,j)} x_k\oplus y_{k+j-i}.\]
In Theorem~\ref{theorem:flaggedtovexmarkedbpdbij}, we will show if $v$ is vexillary, there is a weight preserving bijection \[\flaggedtovexmarkedbpd:\FSet(v)\rightarrow \mpipes{v}.\]  We use this to recover a theorem of Knutson-Miller-Yong from the marked BPD formula for Grothendieck polynomials.
\begin{theorem}[\cite{Knutson.Miller.Yong}]
\label{thm:KMY}
If $v\in \SymGp_n$ is vexillary, then 
\[\mathfrak G^{(\beta)}_v(\mathbf x;\mathbf y)=\sum_{\mathbf T\in \FSet(v)} \wt{\mathbf T}.\]
\end{theorem}
\begin{proof}
Suppose $v$ is vexillary.
By Corollary~\ref{cor:main}, \[\Groth^{(\beta)}_w(\mathbf x;\mathbf y)=\sum_{(\mathcal P,\mathcal S)\in \mpipes{w}}\beta^{|D(\mathcal P)|+|\mathcal S|-\ell(w)}\left(\prod_{(i,j)\in D(\mathcal P)\cup \mathcal S}(x_i\oplus y_j)\right).\]
By Theorem~\ref{theorem:flaggedtovexmarkedbpdbij}, there is a weight preserving bijection \[\flaggedtovexmarkedbpd:\FSet(v)\rightarrow \mpipes{v}.\]
From this, the result follows immediately.
\end{proof}

\subsection{Vexillary bumpless pipe dreams}

 \label{section:vex}

Write $C(\mathcal P)$ for the set of the coordinates of the crossing tiles in $\mathcal P$.

\begin{lemma}
\label{lemma:vexnoncrossing}
Fix $\mathcal P\in \pipes{w}$.  Then $\mathcal P$ has a crossing tile in $\yd{\lambda^{(w)}}$ if and only if $w$ is not vexillary.  Furthermore, if $w$ is vexillary then $\rpipes{w}=\pipes{w}$.
\end{lemma}
\begin{proof}
We first prove the claim for Rothe bumpless pipe dreams.  Suppose $\asmtobpd(w)$ has a crossing tile at $(a,b)\in\yd{\lambda^{(w)}}$.  Then this crossing appears weakly northwest of some $(c,d)\in D(w)$. Furthermore, since $\asmtobpd(w)$ is a Rothe pipe dream, $(a,b)$ must be strictly northwest of $(c,d)$.  To have a crossing at $(a,b)$, necessarily $w^{-1}_b<a$ and $w_a<b$.  Furthermore, since $(c,d)\in D(w)$, $c<w^{-1}_d$ and $d<w_c$.  Thus, $ w^{-1}_b<a<c<w^{-1}_d$ and $w_a<b<d<w_c$ which implies $w$ contains $2143$.

	Conversely, suppose $w$ contains $2143$.  Then there exist $1\leq i_1<i_2<i_3<i_4\leq n$ so that $w_{i_2}<w_{i_1}<w_{i_4}<w_{i_3}$. In particular, this implies $(i_3,w_{i_4})\in D(w)$ and $(i_2,w_{i_1})\in C(\asmtobpd(w))$.  Since $i_2<i_3$ and $w_{i_1}<w_{i_4}$, it follows that $(i_2,w_{i_1})\in \yd{\lambda^{(w)}}$.

We now break the remainder of the argument into cases.

\noindent Case 1: Suppose $w$ is vexillary. 
	
	 If $\mathcal P\in \pipes{w}$ is reduced, then $|C(\mathcal P)|=|C(\asmtobpd(w))|$.  By Lemma~\ref{lemma:muteable}, $\mathcal P$ agrees with $\asmtobpd(w)$ outside of $\yd{\lambda^{(w)}}$.  Therefore, all crossings of $\mathcal P$ occur outside of the mutable region.
In particular, there are no K-theoretic droops available to apply to $\mathcal P$.  Thus by Proposition~\ref{prop:droops},
 we conclude $\rpipes{w}=\pipes{w}$.

\noindent  Case 2: Suppose $w$ is not vexillary.  
	
	$\mathcal P$ agrees with $\asmtobpd(w)$ outside of $\yd{\lambda^{(w)}}$.  In particular, since $\asmtobpd(w)$ has less than $\ell(w)$ crossings outside of  $\yd{\lambda^{(w)}}$, so does $\mathcal P$.  Since $\mathcal P$ has at least as many crossing tiles as $\ell(w)$, it follows that $\mathcal P$ must have some crossing in $\yd{\lambda^{(w)}}$.
\end{proof}

\begin{lemma}
\label{lemma:diagonalswork}
Each diagonal of $\yd{\mu^{(v)}}$ has as many cells as the corresponding diagonal in $D(v)$.
\end{lemma}
\begin{proof}
If $v$ is dominant, $\code_v$ is already decreasing and the result is immediate.

  Whenever $v$ is vexillary, but not dominant, there exists $v'$, also vexillary, so that $\yd{\lambda^{(v')}}\subsetneq \yd{\lambda^{(v)}}$.  We can obtain such a $v'$ by applying a droop move to a southeast most box of $\asmtobpd(v)$ and then ``toggling transition,'' i.e.,\ performing the replacement pictured in (\ref{eq:transitionpicture2}).

Since $v$ is vexillary, this  droop does not involve crossing tiles.  Thus it also preserves the count of blank tiles within each diagonal.
Thus, $D(v')$ has the same number of cells in each of its diagonals as $D(v)$. 
Furthermore, $\code_v$ and $\code_{v'}$ only differ by permuting two entries (the indices of the top and bottom rows of the droop).   Thus, $\mu^{(v)}=\mu^{(v')}$.

 Continuing in this way, we may construct a chain of vexillary permutations to reach a dominant permutation.
\end{proof}

For the discussion that follows, we will need a more local notion of a droop.  Call the replacements pictured below {\bf local moves}.  
\begin{equation}
\label{eqn:vexmoves}
\raisebox{-.75em}{\begin{tikzpicture}[x=1em,y=1em]
	\draw[step=1,gray, thin] (0,0) grid (2,2);
	\draw[color=black, thick](0,0)rectangle(2,2);
	\draw[thick,rounded corners,color=blue] (.5,0)--(.5,1.5)--(2,1.5);
	\end{tikzpicture}
	\hspace{.1em}
	\raisebox{.75em}{$\mapsto$}
	\hspace{.1em}
	\begin{tikzpicture}[x=1em,y=1em]
	\draw[step=1,gray, thin] (0,0) grid (2,2);
	\draw[color=black, thick](0,0)rectangle(2,2);
	\draw[thick,rounded corners,color=blue] (.5,0)--(.5,.5)--(1.5,.5)--(1.5,1.5)--(2,1.5);
	\end{tikzpicture}
	\hspace{1em}
	\begin{tikzpicture}[x=1em,y=1em]
	\draw[step=1,gray, thin] (0,0) grid (2,2);
	\draw[color=black, thick](0,0)rectangle(2,2);
	\draw[thick,rounded corners,color=blue] (0,.5)--(.5,.5)--(.5,1.5)--(2,1.5);
	\end{tikzpicture}
	\hspace{.1em}
	\raisebox{.75em}{$\mapsto$}
	\hspace{.1em}
	\begin{tikzpicture}[x=1em,y=1em]
	\draw[step=1,gray, thin] (0,0) grid (2,2);
	\draw[color=black, thick](0,0)rectangle(2,2);
	\draw[thick,rounded corners,color=blue] (0,.5)--(1.5,.5)--(1.5,1.5)--(2,1.5);
	\end{tikzpicture}
	\hspace{1em}
	\begin{tikzpicture}[x=1em,y=1em]
	\draw[step=1,gray, thin] (0,0) grid (2,2);
	\draw[color=black, thick](0,0)rectangle(2,2);
	\draw[thick,rounded corners,color=blue] (.5,0)--(.5,1.5)--(1.5,1.5)--(1.5,2);
	\end{tikzpicture}
	\hspace{.1em}
	\raisebox{.75em}{$\mapsto$}
	\hspace{.1em}
	\begin{tikzpicture}[x=1em,y=1em]
	\draw[step=1,gray, thin] (0,0) grid (2,2);
	\draw[color=black, thick](0,0)rectangle(2,2);
	\draw[thick,rounded corners,color=blue] (.5,0)--(.5,.5)--(1.5,.5)--(1.5,2);
	\end{tikzpicture}
	\hspace{1em}
	\begin{tikzpicture}[x=1em,y=1em]
	\draw[step=1,gray, thin] (0,0) grid (2,2);
	\draw[color=black, thick](0,0)rectangle(2,2);
	\draw[thick,rounded corners,color=blue] (0,.5)--(.5,.5)--(.5,1.5)--(1.5,1.5)--(1.5,2);
	\end{tikzpicture}
	\hspace{.1em}
	\raisebox{.75em}{$\mapsto$}
	\hspace{.1em}
	\begin{tikzpicture}[x=1em,y=1em]
	\draw[step=1,gray, thin] (0,0) grid (2,2);
	\draw[color=black, thick](0,0)rectangle(2,2);
	\draw[thick,rounded corners,color=blue] (0,.5)--(1.5,.5)--(1.5,2);
	\end{tikzpicture}}
\end{equation}
Applying a local move to a BPD produces a new (well-defined) BPD.  Furthermore local moves do not modify crossing tiles.  As such, if $\mathcal P\mapsto \mathcal P'$ by one of the moves in (\ref{eqn:vexmoves}), it follows that $\demprod{\mathcal P}=\demprod{\mathcal P'}$.
We call the reversed replacements {\bf inverse local moves}.
\begin{equation}
\label{eqn:vexmovesinv}
\raisebox{-.75em}{
	\begin{tikzpicture}[x=1em,y=1em]
	\draw[step=1,gray, thin] (0,0) grid (2,2);
	\draw[color=black, thick](0,0)rectangle(2,2);
	\draw[thick,rounded corners,color=blue] (.5,0)--(.5,.5)--(1.5,.5)--(1.5,1.5)--(2,1.5);
	\end{tikzpicture}
	\hspace{.1em}
	\raisebox{.75em}{$\mapsto$}
	\hspace{.1em}
	\begin{tikzpicture}[x=1em,y=1em]
	\draw[step=1,gray, thin] (0,0) grid (2,2);
	\draw[color=black, thick](0,0)rectangle(2,2);
	\draw[thick,rounded corners,color=blue] (.5,0)--(.5,1.5)--(2,1.5);
	\end{tikzpicture}
	\hspace{1em}
	\begin{tikzpicture}[x=1em,y=1em]
	\draw[step=1,gray, thin] (0,0) grid (2,2);
	\draw[color=black, thick](0,0)rectangle(2,2);
	\draw[thick,rounded corners,color=blue] (0,.5)--(1.5,.5)--(1.5,1.5)--(2,1.5);
	\end{tikzpicture}
	\hspace{.1em}
	\raisebox{.75em}{$\mapsto$}
	\hspace{.1em}
	\begin{tikzpicture}[x=1em,y=1em]
	\draw[step=1,gray, thin] (0,0) grid (2,2);
	\draw[color=black, thick](0,0)rectangle(2,2);
	\draw[thick,rounded corners,color=blue] (0,.5)--(.5,.5)--(.5,1.5)--(2,1.5);
	\end{tikzpicture}
	\hspace{1em}
	\begin{tikzpicture}[x=1em,y=1em]
	\draw[step=1,gray, thin] (0,0) grid (2,2);
	\draw[color=black, thick](0,0)rectangle(2,2);
	\draw[thick,rounded corners,color=blue] (.5,0)--(.5,.5)--(1.5,.5)--(1.5,2);
	\end{tikzpicture}
	\hspace{.1em}
	\raisebox{.75em}{$\mapsto$}	
	\hspace{.1em}
	\begin{tikzpicture}[x=1em,y=1em]
	\draw[step=1,gray, thin] (0,0) grid (2,2);
	\draw[color=black, thick](0,0)rectangle(2,2);
	\draw[thick,rounded corners,color=blue] (.5,0)--(.5,1.5)--(1.5,1.5)--(1.5,2);
	\end{tikzpicture}
	\hspace{1em}
	\begin{tikzpicture}[x=1em,y=1em]
	\draw[step=1,gray, thin] (0,0) grid (2,2);
	\draw[color=black, thick](0,0)rectangle(2,2);
	\draw[thick,rounded corners,color=blue] (0,.5)--(1.5,.5)--(1.5,2);
	\end{tikzpicture}
	\hspace{.1em}
	\raisebox{.75em}{$\mapsto$}
	\hspace{.1em}
	\begin{tikzpicture}[x=1em,y=1em]
	\draw[step=1,gray, thin] (0,0) grid (2,2);
	\draw[color=black, thick](0,0)rectangle(2,2);
	\draw[thick,rounded corners,color=blue] (0,.5)--(.5,.5)--(.5,1.5)--(1.5,1.5)--(1.5,2);
	\end{tikzpicture}}
\end{equation}

The following lemma says that for vexillary permutations, ${\sf Pipes}(v)$ is connected by local moves and inverse local moves.  See Figure~\ref{figure:vexlocalmoves} for an example.
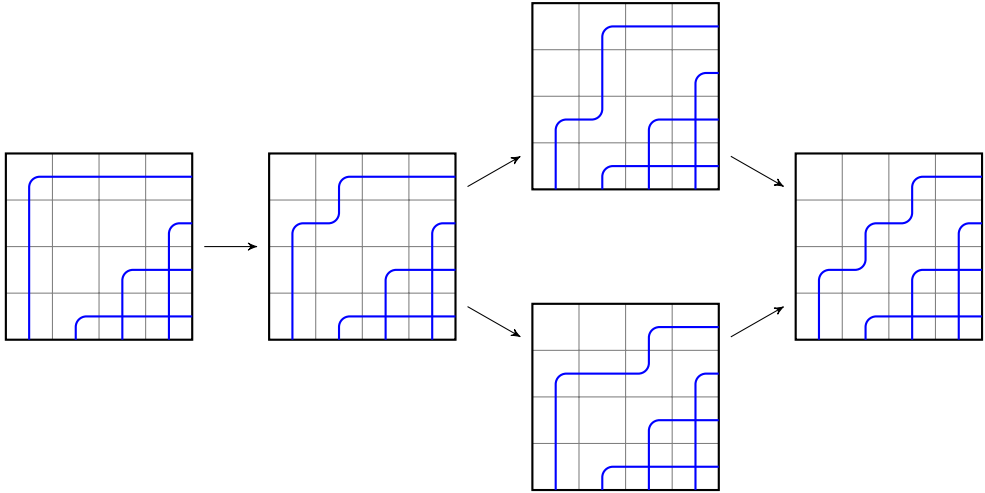
\begin{figure}

\[\begin{tikzpicture}[node distance=1em]

\node (A) at (0, 0) {
\begin{tikzpicture}[x=1.5em,y=1.5em]
\draw[step=1,gray, thin] (0,1) grid (4,5);
\draw[color=black, thick](0,1)rectangle(4,5);

\draw[thick,rounded corners,color=blue] (.5,1)--(.5,4.5)--(4,4.5);
\draw[thick,rounded corners,color=blue] (3.5,1)--(3.5,3.5)--(4,3.5);
\draw[thick,rounded corners,color=blue] (1.5,1)--(1.5,1.5)--(4,1.5);
\draw[thick,rounded corners,color=blue] (2.5,1)--(2.5,2.5)--(4,2.5);

\end{tikzpicture}
};
\node (B) at (3.5, 0) {
\begin{tikzpicture}[x=1.5em,y=1.5em]
\draw[step=1,gray, thin] (0,1) grid (4,5);
\draw[color=black, thick](0,1)rectangle(4,5);

\draw[thick,rounded corners,color=blue] (.5,1)--(.5,3.5)--(1.5,3.5)--(1.5,4.5)--(4,4.5);
\draw[thick,rounded corners,color=blue] (3.5,1)--(3.5,3.5)--(4,3.5);
\draw[thick,rounded corners,color=blue] (1.5,1)--(1.5,1.5)--(4,1.5);
\draw[thick,rounded corners,color=blue] (2.5,1)--(2.5,2.5)--(4,2.5);

\end{tikzpicture}

};
\node (C) at (7, 2) {
\begin{tikzpicture}[x=1.5em,y=1.5em]
\draw[step=1,gray, thin] (0,1) grid (4,5);
\draw[color=black, thick](0,1)rectangle(4,5);

\draw[thick,rounded corners,color=blue] (.5,1)--(.5,2.5)--(1.5,2.5)--(1.5,4.5)--(4,4.5);
\draw[thick,rounded corners,color=blue] (3.5,1)--(3.5,3.5)--(4,3.5);
\draw[thick,rounded corners,color=blue] (1.5,1)--(1.5,1.5)--(4,1.5);
\draw[thick,rounded corners,color=blue] (2.5,1)--(2.5,2.5)--(4,2.5);

\end{tikzpicture}
};
\node (D) at (7, -2) {
\begin{tikzpicture}[x=1.5em,y=1.5em]
\draw[step=1,gray, thin] (0,1) grid (4,5);
\draw[color=black, thick](0,1)rectangle(4,5);

\draw[thick,rounded corners,color=blue] (.5,1)--(.5,3.5)--(2.5,3.5)--(2.5,4.5)--(4,4.5);
\draw[thick,rounded corners,color=blue] (3.5,1)--(3.5,3.5)--(4,3.5);
\draw[thick,rounded corners,color=blue] (1.5,1)--(1.5,1.5)--(4,1.5);
\draw[thick,rounded corners,color=blue] (2.5,1)--(2.5,2.5)--(4,2.5);

\end{tikzpicture}
};
\node (E) at (10.5,0)
{
\begin{tikzpicture}[x=1.5em,y=1.5em]
\draw[step=1,gray, thin] (0,1) grid (4,5);
\draw[color=black, thick](0,1)rectangle(4,5);
\draw[thick,rounded corners,color=blue] (.5,1)--(.5,2.5)--(1.5,2.5)--(1.5,3.5)--(2.5,3.5)--(2.5,4.5)--(4,4.5);
\draw[thick,rounded corners,color=blue] (3.5,1)--(3.5,3.5)--(4,3.5);
\draw[thick,rounded corners,color=blue] (1.5,1)--(1.5,1.5)--(4,1.5);
\draw[thick,rounded corners,color=blue] (2.5,1)--(2.5,2.5)--(4,2.5);
\end{tikzpicture}
};
\draw[->]  (A) edge (B);
\draw[->]  (B) edge (C);
\draw[->]  (B) edge (D);
\draw[->]  (C) edge (E);
\draw[->]  (D) edge (E);

\end{tikzpicture}\]
\caption{The elements of $\pipes{1432}$ are pictured above.
The BPD on the left is $\asmtobpd(1432)$.  The rightmost BPD is $\mathcal T^{(1432)}$.  The arrows between BPDs indicate the application of a local move.}
\label{figure:vexlocalmoves}
\end{figure}

\begin{lemma}
	\label{lemma:vexblanktiles}
	Fix $v\in \SymGp_n$ so that $v$ is vexillary.  
	\begin{enumerate}
		\item Any $\mathcal P\in {\sf Pipes}(v)$ is uniquely determined by the positions of its blank tiles.
		\item  ${\sf Pipes}(v)$ is closed under applications of the moves in (\ref{eqn:vexmoves}).  Furthermore, if $\mathcal P\in {\sf Pipes}(v)$, there is a sequence 
		\[\asmtobpd(v)=\mathcal P_0\mapsto \mathcal P_1\mapsto \cdots \mapsto \mathcal P_k=\mathcal P\] where each $\mathcal P_i\mapsto \mathcal P_{i+1}$ is a local move.
		\item There exists a unique $\mathcal T^{(v)}\in {\sf Pipes}(v)$ so that $D(\mathcal T^{(v)})=\yd{\mu^{(v)}}$. Call $\mathcal T^{(v)}$ the {\bf top} bumpless pipe dream for $v$. 
		\item Any $\mathcal P\in {\sf Pipes}(v)$ can be reached from $\mathcal T^{(v)}$ by a sequence of inverse local moves.
	\end{enumerate}	
\end{lemma}
\begin{proof}
	\noindent (1) By Lemma~\ref{lemma:uniquelydet}, each $\mathcal P\in \bpd{n}$ is uniquely determined by knowing the positions of its crossing tiles and its blank tiles.  By Lemma~\ref{lemma:vexnoncrossing}, for all $\mathcal P\in {\sf Pipes}(v)$, the crossing tiles occur outside of $\yd{\lambda^{(v)}}$.  In particular, the crossing tiles are in the same location for each $\mathcal P\in {\sf Pipes}(v)$.  This implies $\mathcal P\in {\sf Pipes}(v)$ is uniquely determined by the locations of its blank tiles.
	
	\noindent (2)
	 By Lemma~\ref{lemma:vexnoncrossing}, since $v$ is vexillary, there are no crossing tiles in $\yd{\lambda^{(v)}}$.  Furthermore, ${\sf Pipes}(v)={\sf RPipes}(v)$.  As such, $\mathcal P$ can be reached from $\asmtobpd(v)$ by a sequence of droops.  In particular, these droops do not involve crossing tiles.
	  Therefore, each one is merely a composition of the moves in (\ref{eqn:vexmoves}).

	\noindent (3)  
	If $\mathcal P$ is a Hecke BPD, then by Lemma~\ref{lemma:diagonalswork}, we are done.  Otherwise, there exists  some $(i,j)\in D(\mathcal P)$ so that $(i,j)$ is not in the dominant part of $D(\mathcal P)$. 
By choosing a northwest most tile in its connected component, we may assume WLOG that $(i,j)$ has no blank tile immediately to its left or immediately above.
  Since $(i,j)$ is not in the dominant part of $D(\mathcal P)$, the tiles in $(i-1,j)$ and $(i,j-1)$ must be occupied by some pipe. 
 In particular, $(i-1,j-1)$ cannot be blank. 
 Therefore, we are in one of the situations pictured in (\ref{eqn:vexmoves}) and can apply a local move. 
 We may continue this process until we have reached a Hecke BPD.

By the second part, ${\sf Pipes}(v)$ is connected by local moves.  In particular, this implies that each diagonal contains the same number of blank tiles for each $\mathcal P\in {\sf Pipes}(v)$.  As such, there is at most one Hecke BPD in ${\sf Pipes}(v)$ whose diagonal lengths are uniquely determined by the number of cells in each diagonal of $D(v)$. By Lemma~\ref{lemma:diagonalswork}, this partition coincides with $\mu^{(v)}$.

	\noindent (4) 	This follows immediately from the construction of $\mathcal T^{(v)}$ in the previous part. We may simply reverse the path from $\mathcal P$ to $\mathcal T^{(v)}$ by using inverse local moves.
\end{proof}

\subsection{Families of non-intersecting lattice paths}
\label{section:paths}

We now turn our attention to the families of lattice paths from \cite{Kreiman}.
For this section, place a point in the center of each cell of the $n\times n$ grid.  From this, create a square lattice.  We will consider northeast lattice paths on this grid.

A northeast lattice path $\mathcal L$ is \mydef{supported on $\lambda$} if it:
\begin{enumerate}
\item starts at the lowest cell of a column in $\yd{\lambda}$,
\item moves north or east with each step, and
\item ends at the rightmost cell of a row in $\yd{\lambda}$.
\end{enumerate}

Suppose $\mathcal F=(\mathcal L_1,\,\ldots,\, \mathcal L_k)$ is a tuple of northeast lattice paths supported on $\lambda$.  We say $\mathcal F$ is \mydef{non-intersecting} if no two pairs of paths in $\mathcal F$ share a common vertex.  
It is immediate that there is an injection from tuples of non-intersecting northeast lattice paths supported on $\lambda$ into $\bpd{\lambda}$.  The image is precisely  \[\ncbpd{\lambda}:=\{\mathcal P\in \bpd{\lambda}: \mathcal P \text{ has no crossing tiles}\}.\] 
As such, we conflate $\mathcal F$ with its image in $\ncbpd{\lambda}$ from this point on.
Furthermore, Kreiman's ladder moves (named for the analogous moves of \cite{Bergeron.Billey} on ordinary pipe dreams) are immediately seen to be the moves in (\ref{eqn:vexmovesinv}).

\begin{lemma}
\label{lemma:kreimansupport}
For any $\mu\subseteq \lambda$, there exists a unique $\mathcal P\in \ncbpd{\lambda}$ so that $D(\mathcal P)=\yd{\mu}$.  Write $\mathcal P^{\rm top}_{\lambda}(\mu)$ for this BPD.
\end{lemma} 
\begin{proof}
This is \cite[Lemma~5.3]{Kreiman}.
\end{proof}

Write $\paths_\lambda(\mu)$ for the set of $\mathcal P\in \ncbpd{\lambda}$ which can be reached from  $\mathcal P^{\rm top}_{\lambda}(\mu)$ by a sequence of inverse local moves.

\begin{proposition}
\label{prop:vexcomp}
\begin{enumerate}
\item If $\mathcal P\in \ncbpd{\lambda}$, then $\completebpd(\mathcal P)$ is a vexillary BPD.
\item The completion map defines a bijection from $\paths_\lambda(\mu)$ to $\pipes{v}$ for some vexillary permutation $v$.
\end{enumerate}
\end{proposition}
\begin{proof}
\noindent (1) Fix $\mathcal P\in \ncbpd{\lambda}$.  By assumption, $\mathcal P$ has no crossing tiles.  Therefore, $\completebpd(\mathcal P)$ does not have crossing tiles in $\Mute(\demprod{\completebpd(\mathcal P)})$.  Therefore, by Lemma~\ref{lemma:vexnoncrossing}, $\demprod{\completebpd(\mathcal P)}$ is vexillary.

\noindent (2)  Write $v=\demprod{\mathcal P^{\rm top}_{\lambda}(\mu)}$.  It is immediate that for all $\mathcal P\in \paths_\lambda(\mu)$, the completion $\completebpd(\mathcal P)\in \pipes{v}$.  Furthermore, by construction,  the  completion map is injective.  
Surjectivity follows from Lemma~\ref{lemma:compatcompletion}.
\end{proof}

\subsection{Tableaux to pipes}
\label{section:tabpipes}
Recall the map $\FSSYTtovexbpd:\FSSYT(v)\rightarrow \pipes{v}$ which takes $T\in \FSSYT(v)$ to (the unique) $\mathcal P\in \pipes{v}$ so that 
\[D(\mathcal P)=\{(T(i,j),T(i,j)-i+j):(i,j)\in \yd{\mu^{(v)}}\}.\]

\begin{theorem}
\label{thm:ssyttovexbpdbij}  If $v\in \SymGp_n$ is vexillary, 
the map $$\FSSYTtovexbpd:\FSSYT(v)\rightarrow \pipes{v}$$ is a bijection.
\end{theorem}
\begin{proof}
This is a consequence of \cite[Proposition~3.6]{Morales.Pak.Panova.I}  and Proposition~\ref{prop:vexcomp}.  
\end{proof}
In Figure~\ref{figure:flaggedtabs}, we list the elements of $\FSSYT(1432)$.  They are positioned in the same relative order as their counterparts in Figure~\ref{figure:vexlocalmoves}.  Notice that applying a local move to $\mathcal P\in \pipes{1432}$ corresponds to decreasing the value of a cell within  $\FSSYTtovexbpd(\mathcal P)$ by one.

\begin{figure} 
\[\begin{tikzpicture}[node distance=1em]

\node (A) at (0, 0) {
$\begin{ytableau}
2&2\\3
\end{ytableau}$
};
\node (B) at (3.5, 0) {
$\begin{ytableau}
1&2\\3
\end{ytableau}$
};
\node (C) at (7, 2) {
$\begin{ytableau}
1&2\\2
\end{ytableau}$
};
\node (D) at (7, -2) {
$\begin{ytableau}
1&1\\3
\end{ytableau}$
};
\node (E) at (10.5,0)
{
$\begin{ytableau}
1&1\\2
\end{ytableau}$
};
\draw[->]  (A) edge (B);
\draw[->]  (B) edge (C);
\draw[->]  (B) edge (D);
\draw[->]  (C) edge (E);
\draw[->]  (D) edge (E);

\end{tikzpicture}\]
\caption{We list the flagged tableaux for $v=1432$.  In this case, $\mu^{(v)}=(2,1)$, $\lambda^{(v)}=(3,3,2)$, and $\mathbf f^{(v)}=(2,3)$.}
\label{figure:flaggedtabs}
\end{figure}
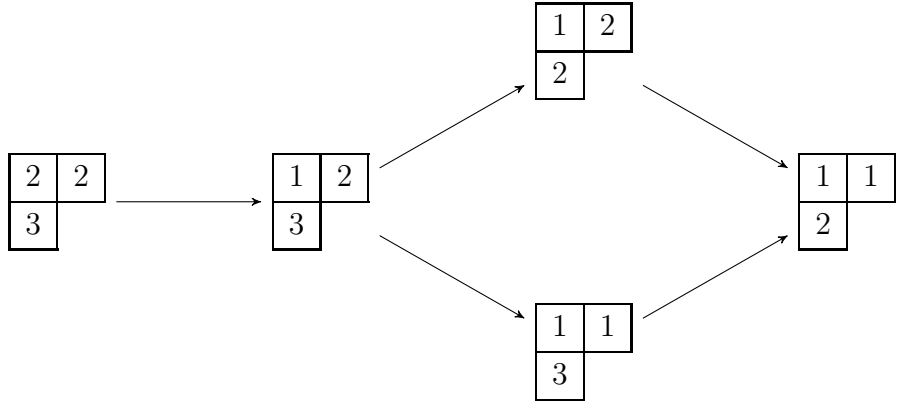

Now recall \[\mathcal S_{\mathbf T}=\{(k,k-i+j):k \in \mathbf T(i,j) \text{ and } k\neq \min(\mathbf T(i,j)) \}.\]  Define $\flaggedtovexmarkedbpd:\FSet(v)\rightarrow \mpipes{v}$ by \[\flaggedtovexmarkedbpd(\mathbf T)=(\FSSYTtovexbpd(\flatten(\mathbf T)),\mathcal S_{\mathbf T}).\]

We seek to show $\flaggedtovexmarkedbpd$ is a bijection.
We start by introducing an auxiliary set of tableaux.
Call $\mathbf T\in \FSet(\lambda,\mathbf f)$ \mydef{saturated} if for all $(i,j)\in \yd{\lambda}$, there is no $k>\min(\mathbf T(i,j))$ for which adding $k$ to $\mathbf T(i,j)$ produces an element of $\FSet(\lambda,\mathbf f)$.  Write $\SatFSet(\lambda,\mathbf f)$ for the set of saturated tableaux in $\FSet(\lambda,\mathbf f)$.  We abbreviate $\SatFSet(v):=\SatFSet(\mu^{(v)},\mathbf f^{(v)})$.

\begin{lemma}
\label{lemma:satflattenbij}
The map  $\flatten:\SatFSet(\lambda,\mathbf f)\rightarrow \FSSYT(\lambda,\mathbf f)$  is a bijection.
\end{lemma}
\begin{proof}
Take $\mathbf T,\mathbf T'\in \SatFSet(\lambda,\mathbf f)$ so that $\flatten(\mathbf T)=\flatten(\mathbf T')$.  
Suppose for contradiction, there exists $k\in \mathbf T(i,j)\setminus\mathbf T'(i,j)$. Then  $k\leq \min(\mathbf T(i,j+1))=\min(\mathbf T'(i,j+1))$, $k<\min(\mathbf T(i+1,j))=\min(\mathbf T'(i+1,j))$, and $k\leq \mathbf f_i$.  As such, we can add $k$ to $\mathbf T'(i,j)$ to produce a valid element of $\FSet(\lambda,\mathbf f)$  But then $\mathbf T'$ is not saturated, which is a contradiction.  An identical argument shows there is no $k\in \mathbf T'(i,j)\setminus\mathbf T(i,j)$.  Thus, $\mathbf T=\mathbf T'$.

We now show the map is surjective.
Start with $T\in \FSSYT(\lambda,\mathbf f)$.  Let $\mathbf T$ be the set-valued tableau so that $\mathbf T(i,j)=\{T(i,j)\}$ for all $(i,j)\in \yd{\lambda}$. If $\mathbf T$ is saturated, the result is automatic. If not, by definition, there is some $(i,j)\in \yd{\lambda}$ and $k\in \mathbb P$ so that $k>\min(\mathbf T(i,j))$ and adding $k$ to $\mathbf T(i,j)$ produces an element of $\FSet(\lambda,\mathbf f)$.  Continue this process, until the resulting tableau is saturated.  This tableau maps to $T$ by construction.
\end{proof}

As an immediate consequence, $\SatFSet(v)$ is in bijection with $\pipes{v}$.
We now show the excess entries in a saturated tableau for $v$ tell us the positions of the upward elbows in the corresponding BPD.
\begin{lemma}
\label{lemma:elbows}
 If $\mathbf T\in \SatFSet(v)$ then $\mathcal S_{\mathbf T}=U(\FSSYTtovexbpd(\flatten(\mathbf T)))$.
\end{lemma}
\begin{proof}
First notice that  \[\FSSYTtovexbpd\circ\flatten:\SatFSet(v)\rightarrow \pipes{v}\]  is a bijection by Theorem~\ref{thm:ssyttovexbpdbij} and Lemma~\ref{lemma:satflattenbij}.

We proceed by induction on local moves from $\asmtobpd(v)$.
In the base case, let $T\in \FSSYT(v)$ so that $\FSSYTtovexbpd(T)=\asmtobpd(v)$.  None of the entries of $T$ can be increased without violating the flagging or semistandardness conditions defining $\FSSYT(v)$.  In particular, this means $T$ maps to the saturated tableau $\mathbf T$ defined by $\mathbf T(i,j)=\{T(i,j)\}$ for all $(i,j)\in \yd{\mu^{(v)}}$.  Thus, $\mathcal S_{\mathbf T}=\emptyset=U(\asmtobpd(v))$.

Now fix $T\in \FSSYT(v)$.  Let $\mathcal P=\FSSYTtovexbpd(T)$.  If $\mathcal P=\asmtobpd(v)$ we are done, so assume not.  Then  by Lemma~\ref{lemma:vexblanktiles}, there is $\mathcal P'\in \pipes{v}$ so that $\mathcal P'\rightarrow \mathcal P$ is a local move.  Let $T'\in \FSSYT(v)$ so that  $\FSSYTtovexbpd(T')=\mathcal P'$ and let $\mathbf T'\in \SatFSet(v)$ so that $\flatten(\mathbf T')=T$.  

Suppose $\mathcal S_{\mathbf T'}=U(\mathcal P')$.  We seek to show $\mathcal S_{\mathbf T}=U(\mathcal P)$.  We proceed by analyzing the local moves in (\ref{eqn:vexmoves}).
The effect of each local move $T'\mapsto T$ is to take some label $k+1=T'(i,j)$ and replace it with $k$.  As such, $\mathbf T(i,j)=\mathbf T'(i,j)\cup\{k\}$.  
Furthermore, $\mathbf T(i,j-1)=\mathbf T'(i,j-1)-\{k+1\}$ and  $\mathbf T(i-1,j)=\mathbf T'(i,j-1)-\{k\}$. All other entries of $\mathbf T$ and $\mathbf T'$ agree.

These replacements are compatible with the local moves in (\ref{eqn:vexmoves}).  As such, $\mathcal S_{\mathbf T}=U(\mathcal P)$.
\end{proof}

We conclude by proving the main theorem of this section.

\begin{proof}[Proof of Theorem~\ref{theorem:flaggedtovexmarkedbpdbij}]
That the map $\flaggedtovexmarkedbpd$ is weight preserving is immediate from the definitions.  Furthermore, $\FSet(v)$ and $\pipes{v}$ are in bijection (by Lemma~\ref{lemma:satflattenbij} and Theorem~\ref{thm:ssyttovexbpdbij}.)

Elements of $\FSet(v)$ are uniquely determined by first fixing $\mathbf T\in \SatFSet(v)$ and then freely selecting a subset of the non-minimal elements of each cell.  Likewise, elements of $\mpipes{v}$ are determined by the choice of $\mathcal P\in \pipes{v}$ and a subset $\mathcal S\subseteq U(\mathcal P)$.  

By Lemma~\ref{lemma:elbows}, $\mathcal S_{\mathbf T}=U(\FSSYTtovexbpd(\flatten(\mathbf T)))$.
Thus, the choice of non-minimal elements in each cell uniquely determines a subset of $U(\mathcal P)$ (and vice versa).
\end{proof}

\section{Hecke bumpless pipe dreams}
\label{section:hecketabs}
The goal of this section is to answer a problem posed in \cite{Lam.Lee.Shimozono} relating Edelman-Greene BPDs to increasing tableaux whose row reading words are reduced words. We give a bijection between Hecke BPDs and decreasing tableaux. In the case of reduced BPDs, this restricts to a solution of \cite[Problem~5.19]{Lam.Lee.Shimozono} (up to a convention shift).  

\subsection{Hecke bumpless pipe dreams}

Recall, a Hecke BPD is a BPD whose diagram is top left justified.  In this case, the diagram  necessarily forms a partition called the shape of the Hecke BPD.  Recall $\hbpd{n}$ is the set of  Hecke BPDs of size $n$.  Let $\hbpd{\lambda,n}$ be the subset of Hecke BPDs of shape $\lambda$. 

\begin{lemma}
\label{lemma:pipesinvolution}
Reflecting an ASM across its antidiagonal defines an involution on BPDs.  Furthermore, this map has the effect of changing blank tiles to crossing tiles (and vice versa) and then reflecting these sets across the antidiagonal.
\end{lemma}
\begin{proof}
That the map is an involution is immediate.
On square ice configurations, this map has the effect of reversing arrows all arrows and then transposing across the antidiagonal.    Translated to BPDs,  this map changes crossings to blank tiles
 (and vice versa) and then flips these sets across the antidiagonal.
\end{proof}

Given $\lambda\subseteq \delta^{(n)}$, let $v(\lambda,n)$ be the (unique) vexillary permutation so that \[C(\asmtobpd(v(\lambda,n)))=\{(n-j+1,n-i+1):(i,j)\in \yd{\lambda}\}.\]  Notice $v(\lambda,n)$ is guaranteed to exist by Lemma~\ref{lemma:pipesinvolution}, since there exists a dominant permutation $u$ so that $D(u)=\yd{\lambda}$.  Applying the above involution to $u$ produces the Rothe BPD for a vexillary permutation with the desired set of crossing tiles.  We will make use of the following restriction of the involution on BPDs.
\begin{lemma}
\label{lemma:pipesinvolutioncor}
There is a bijection between $\hbpd{\lambda,n}$ and $\pipes{v(\lambda,n)}$.
\end{lemma}
\begin{proof}
  When applying the map from  Lemma~\ref{lemma:pipesinvolution}, it is immediate that all Hecke tableaux of shape $\lambda$ map to BPDs for $v(\lambda,n)$ and vice versa.  Since this map is an involution on $\bpd{n}$, the result follows.
\end{proof}
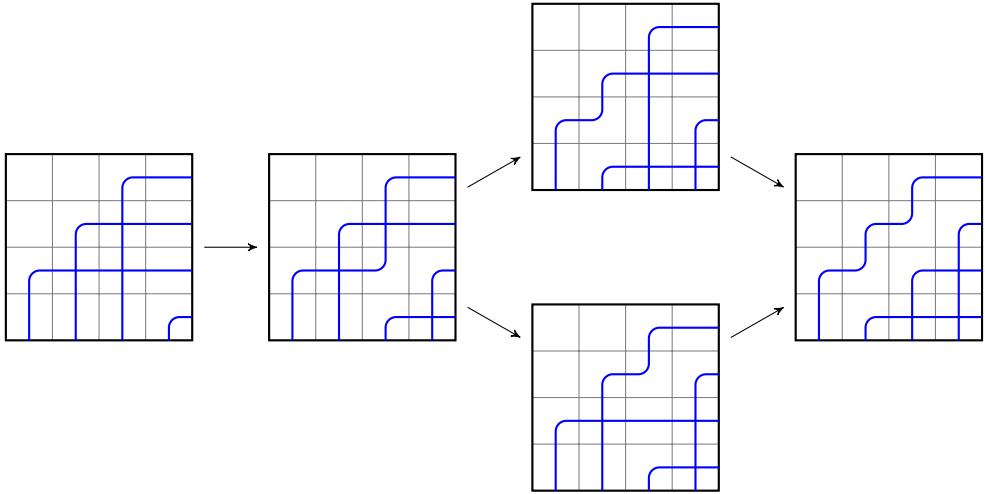
\begin{figure}[h] 
\[\begin{tikzpicture}[node distance=1em]

\node (A) at (0, 0) {
\begin{tikzpicture}[x=1.5em,y=1.5em]
\draw[step=1,gray, thin] (0,1) grid (4,5);
\draw[color=black, thick](0,1)rectangle(4,5);

\draw[thick,rounded corners,color=blue] (3.5,1)--(3.5,1.5)--(4,1.5);
\draw[thick,rounded corners,color=blue] (2.5,1)--(2.5,4.5)--(4,4.5);
\draw[thick,rounded corners,color=blue] (.5,1)--(.5,2.5)--(4,2.5);
\draw[thick,rounded corners,color=blue] (1.5,1)--(1.5,3.5)--(4,3.5);

\end{tikzpicture}
};
\node (B) at (3.5, 0) {

\begin{tikzpicture}[x=1.5em,y=1.5em]
\draw[step=1,gray, thin] (0,1) grid (4,5);
\draw[color=black, thick](0,1)rectangle(4,5);
\draw[thick,rounded corners, color=blue] (.5,1)--(.5,2.5)--(2.5,2.5)--(2.5,4.5)--(4,4.5);
\draw[thick,rounded corners, color=blue] (1.5,1)--(1.5,3.5)--(4,3.5);
\draw[thick,rounded corners, color=blue] (2.5,1)--(2.5,1.5)--(4,1.5);
\draw[thick,rounded corners, color=blue] (3.5,1)--(3.5,2.5)--(4,2.5);
\end{tikzpicture}
};
\node (C) at (7, 2) {
\begin{tikzpicture}[x=1.5em,y=1.5em]
\draw[step=1,gray, thin] (0,1) grid (4,5);
\draw[color=black, thick](0,1)rectangle(4,5);

\draw[thick,rounded corners,color=blue] (.5,1)--(.5,2.5)--(1.5,2.5)--(1.5,3.5)--(4,3.5);
\draw[thick,rounded corners,color=blue] (3.5,1)--(3.5,2.5)--(4,2.5);
\draw[thick,rounded corners,color=blue] (1.5,1)--(1.5,1.5)--(4,1.5);
\draw[thick,rounded corners,color=blue] (2.5,1)--(2.5,4.5)--(4,4.5);

\end{tikzpicture}
};
\node (D) at (7, -2) {
\begin{tikzpicture}[x=1.5em,y=1.5em]
\draw[step=1,gray, thin] (0,1) grid (4,5);
\draw[color=black, thick](0,1)rectangle(4,5);
\draw[thick,color=blue,rounded corners] (1.5,1)--(1.5,2.5)--(1.5,3.5)--(2.5,3.5)--(2.5,4.5)--(4,4.5);
\draw[thick,color=blue,rounded corners](.5,1)--(.5,2.5)--(4,2.5);
\draw[thick,color=blue,rounded corners](2.5,1)--(2.5,1.5)--(4,1.5);
\draw[thick,color=blue,rounded corners](3.5,1)--(3.5,3.5)--(4,3.5);
\end{tikzpicture}
};
\node (E) at (10.5,0)
{
\begin{tikzpicture}[x=1.5em,y=1.5em]
\draw[step=1,gray, thin] (0,1) grid (4,5);
\draw[color=black, thick](0,1)rectangle(4,5);
\draw[thick,rounded corners,color=blue] (.5,1)--(.5,2.5)--(1.5,2.5)--(1.5,3.5)--(2.5,3.5)--(2.5,4.5)--(4,4.5);
\draw[thick,rounded corners,color=blue] (3.5,1)--(3.5,3.5)--(4,3.5);
\draw[thick,rounded corners,color=blue] (1.5,1)--(1.5,1.5)--(4,1.5);
\draw[thick,rounded corners,color=blue] (2.5,1)--(2.5,2.5)--(4,2.5);
\end{tikzpicture}
};
\draw[->]  (A) edge (B);
\draw[->]  (B) edge (C);
\draw[->]  (B) edge (D);
\draw[->]  (C) edge (E);
\draw[->]  (D) edge (E);

\end{tikzpicture}\]
\caption{Pictured above are the elements of $\hbpd{(2,1),4}$.}
\label{figure:heckepipes}
\end{figure}

See Figure~\ref{figure:heckepipes} for the Hecke BPDs which correspond to the elements of $\pipes{1432}$ (pictured in Figure~\ref{figure:flaggedtabs}).

\subsection{From flagged tableaux to decreasing tableaux}

Pair blank tiles and crossing tiles within diagonals by associating the first blank tile with the last crossing tile and so on.
Now given $\mathcal P\in \pipes{v(\lambda,n)}$,  define a tableau of shape $\lambda$ by recording the difference of the row indices of each crossing tile with its corresponding blank tile.  Write \[\vexBPDtodectab:\pipes{v(\lambda,n)}\rightarrow \DT(\lambda,n-1)\] for this map.
In Figure~\ref{figure:decreasing}, we list the decreasing tableaux for $v((2,1),4)=1432$.  They are in the same relative positions as their counterparts in Figure~\ref{figure:flaggedtabs}.  Applying a local move to $\mathcal P\in \pipes{1432}$ corresponds to increasing a label in $\vexBPDtodectab(\mathcal P)$ by one.

\begin{figure}
\[\begin{tikzpicture}[node distance=1em]

\node (A) at (0, 0) {
$\begin{ytableau}
2&1\\1
\end{ytableau}$
};
\node (B) at (3.5, 0) {
$\begin{ytableau}
3&1\\1
\end{ytableau}$
};
\node (C) at (7, 2) {
$\begin{ytableau}
3&1\\2
\end{ytableau}$
};
\node (D) at (7, -2) {
$\begin{ytableau}
3&2\\1
\end{ytableau}$
};
\node (E) at (10.5,0)
{
$\begin{ytableau}
3&2\\2
\end{ytableau}$
};
\draw[->]  (A) edge (B);
\draw[->]  (B) edge (C);
\draw[->]  (B) edge (D);
\draw[->]  (C) edge (E);
\draw[->]  (D) edge (E);

\end{tikzpicture}\]
\caption{
Above are the decreasing tableaux for $v((2,1),4)=1432$.}
\label{figure:decreasing}
\end{figure}
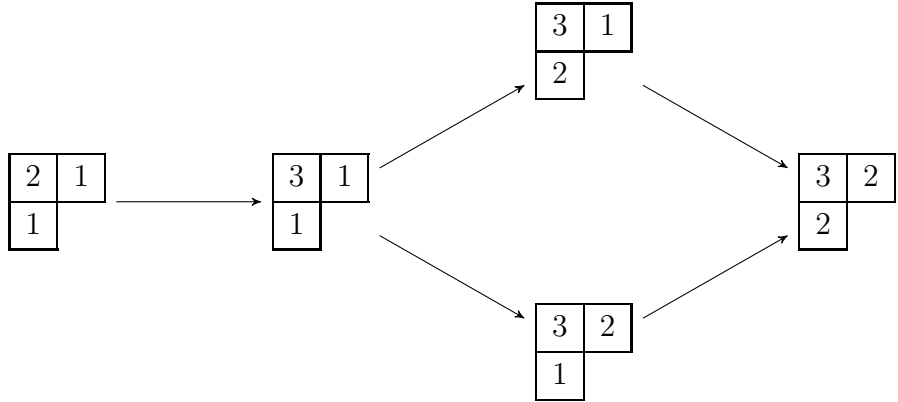

\begin{lemma}
\label{lemma:vexbpdtodectabbij}
The  map $\vexBPDtodectab:\pipes{v(\lambda,n)}\rightarrow \DT(\lambda,n-1)$ is a bijection.
\end{lemma} 
\begin{proof}

We start by showing $\vexBPDtodectab$ is well-defined. 
Take $\mathcal P\in \pipes{v(\lambda,n)}$ and write $D=\vexBPDtodectab(\mathcal P)$.  Fix $T\in \FSSYT(v(\lambda,n))$ so that  $\mathcal P=\FSSYTtovexbpd(T)$.

Notice if $\mathcal P$ is the top BPD for $v(\lambda,n)$ we have $T(i,j)+D(i,j)=n-j+1$.  In general, 
any local move on $\mathcal P$ has the effect of decreasing a label in $T(i,j)$ and increasing the corresponding label in $D(i,j)$ (both by $1$). Therefore, $T(i,j)+D(i,j)=n-j+1$  for all $(i,j)\in \yd{\lambda}$ for arbitrary $\mathcal P\in \pipes{v(\lambda,n)}$.
In particular, 
\begin{equation}
\label{eqn:decreastoss}
D(i,j)=n-j-T(i,j)+1.
\end{equation}
 Since $T$ is semistandard, it is quick to verify that $D$ is decreasing.  Furthermore, by definition, $D(i,j)\in[n-1]$ (any diagonal in $[n]\times[n]$ has at most $n$ entries).  Thus, $D\in \DT(\lambda,n-1)$.

It is clear from (\ref{eqn:decreastoss}), that the map is injective.  All that remains is to show surjectivity.
If $\mathcal P$ is the top BPD for $v(\lambda,n)$, then $D$ is the maximal element in $\DT(\lambda,n-1)$, i.e.,\ $D(i,j)=n-i-j+1$ for all $(i,j)\in \yd{\lambda}$.  So it is enough to verify if an entry of $D$ can decrease by $1$ to produce a valid $D'\in \DT(\lambda,n-1)$, then there is a corresponding $\mathcal P'\in \pipes{v(\lambda,n)}$ so that $\mathcal P'\mapsto D'$.

Suppose $D(i,j)$ can be decreased.  This means $T(i,j)$ can be increased without violating semistandardness.  Write $(a,b)$ for  the position of the blank tile in $\mathcal P$ which corresponds to $(i,j)\in \yd{\lambda}$.  Since $T(i,j)$ can be increased, $(a+1,b), (a,b+1),$ and $(a+1,b+1)$ are all not blank tiles in $\mathcal P$ (this follows from semistandardness).  If $(a+1,b+1)$ is a crossing tile, this would imply $D(i,j)=1$.  However this would mean $D(i,j)$ could not be decreased in the first place.  So $(a+1,b+1)$ is also not a crossing tile.  Thus, we can apply an inverse local move to $\mathcal P$ and we obtain the desired $\mathcal P'\in \pipes{v(\lambda,n)}$.
\end{proof}

\subsection{Proof of Theorem~\ref{thm:hecketabs}}

We now recall the map \[\HeckeBPDtodectab:\hbpd{n}\rightarrow \DT(n-1).\] 
 Fix $\mathcal P\in\hbpd{n}$ to $\DT(n-1)$.  In each diagonal, pair the first blank tile  with the last crossing tile, the next with the second to last, and so on.  If a cell in $D(\mathcal P)$  sits in row $i$ and its corresponding crossing tile in row $i'$, we fill the cell with the label $i'-i$ to obtain $\HeckeBPDtodectab(\mathcal P)$.

Notice that if $\mathcal P\in \hbpd{\lambda,n}$, then $\HeckeBPDtodectab(\mathcal P)$ is the result of composing the  map from $\hbpd{n}$ to $\pipes{v(\lambda,n)}$ and the map $\vexBPDtodectab:\pipes{v(\lambda,n)}\rightarrow \DT(\lambda,n-1)$. As such, $\HeckeBPDtodectab$ is well-defined. 

\begin{lemma}
\label{lemma:heckecommutationclass}
The map \[\HeckeBPDtodectab:\hbpd{n}\rightarrow \DT(n-1)\] preserves commutation classes of reading words.
\end{lemma}
\begin{proof}
We will induct on $n$.  In the base case $n=1$, there is nothing to show.
Now fix $n>1$ and suppose the statement holds for $n-1$.  

Take $\mathcal P\in\hbpd{n}$ and let $T=\HeckeBPDtodectab(\mathcal P)$.  Let $A$ be the ASM which corresponds to $\mathcal P$.
Let $T'$ be the tableau obtained by removing the first column of $T$.  Let $\mathcal P\in \hbpd{n-1}$ so that $\HeckeBPDtodectab(\mathcal P')=T'$.

 Write $\mathbf a_\mathcal P$ for the reading word of $\mathcal P$. 
By the inductive hypothesis, the reading words of $\mathcal P'$ and $T'$ are in the same commutation class.  So it is enough to show we may apply commutation relations to move the reflections in $\mathbf a_\mathcal P$ which correspond to the first column of $\mathbf T$ to the start of the word (keeping their relative order).  Label the positions of these crossings $(i_1,j_1),\,(i_2,j_2),\,\ldots,\,(i_k,j_k)$.

Fix $(i,j)\in \{(i_1,j_1),\,(i_2,j_2),\,\ldots,\,(i_k,j_k)\}$.  By Lemma~\ref{lemma:rankcrossing}, this corresponds to the reflection $s_{r_A(i,j)-1}$.  We know $(i,j)$ appears first in its respective column.  Furthermore,  any crossing which appears earlier in the reading order, but does not belong to the set $\{(i_1,j_1),(i_2,j_2),\ldots,(i_k,j_k)\}$, must be strictly northwest of $(i,j)$.  If $(i',j')$ is strictly northwest of $(i,j)$, then $r_A(i',j')\leq r_A(i,j)-2$ (this is Lemma~\ref{lemma:rankice} combined with the fact that rank functions of ASMs are weakly increasing).  Thus, the reflection for $(i,j)$ is free to commute past the reflection for $(i',j')$ (by Lemma~\ref{lemma:rankcrossing}).
\end{proof}

\begin{proof}[Proof of Theorem~\ref{thm:hecketabs}]
By Lemma~\ref{lemma:heckecommutationclass}, the map preserves commutation classes. Therefore, it also preserves Hecke products.  That the map is a shape preserving bijection follows from Lemma~\ref{lemma:pipesinvolutioncor} and Lemma~\ref{lemma:vexbpdtodectabbij}.
\end{proof}

From Theorem~\ref{thm:hecketabs}, we obtain the immediate corollary:
\begin{corollary}
The restriction of $\HeckeBPDtodectab$ to Edelman-Greene tableaux in $\bpd{n}$ defines a bijection to $\RWT(n-1)$.  This bijection is shape preserving and takes Edelman-Greene tableaux for $w$ to reduced word tableaux for $w$.
\end{corollary}
This provides a new solution to \cite[Problem 5.19]{Lam.Lee.Shimozono}.

\appendix

\section{Transition equations for $\beta$-double Grothendieck polynomials}
\label{appendix:transitionproof}

In this appendix, we establish transition equations for $\beta$-double Grothendieck polynomials. 
First, we recall some notation.  If
   $k\in \mathbb P$, let $1^k$ be the tuple $(1,\ldots,1,0,0,\ldots)$ which starts with $k$ ones.
If $I=\{i_1,\ldots, i_k\}$, with $1\leq i_1<i_2<\cdots<i_k<a$,  then $c^{(a)}_{I}=(a\, i_k\, i_{k-1}\, \ldots\, i_1)$.  Furthermore, if $w\in \SymGp_n$ and $I\subseteq [n]$,  write $w\cdot I:=\{w_i:i\in I\}$.

\begin{lemma}
\label{lemma:cycles}
Fix $I\subseteq [a-1]$.
\begin{enumerate}
\item If $i\not \in I$ or $i+1 \not \in I$ then $s_ic^{(a)}_{I}s_i=c^{(a)}_{s_i\cdot I}$.
\item If $i,i+1\in I$ then $s_ic^{(a)}_{I}=c^{(a)}_{I-\{i+1\}}$ and $c^{(a)}_{I}s_i=c^{(a)}_{I-\{i\}}$.
\end{enumerate}
\end{lemma}
The proof of Lemma~\ref{lemma:cycles} is elementary, so we omit it.

Recall, $\phi(w)=\{i:(i,j) \enspace \text{is a pivot of} \enspace \mc(w) \enspace \text{in} \enspace w\}.$  We typically write $\mc(w)=(a,b)$ and $b'=w^{-1}(b)$.  (In particular, $a=\des(w)$.)
Given $I\subseteq \phi(w)$, define $w_I=w t_{a \,b'} c^{(a)}_{I}.$

\begin{lemma}
\label{lemma:GrothAppendBoxes}
Fix $w\in \SymGp_{n-1}$ and let $a=\des(w)$.  Suppose $a\leq k<n$ and let $w'$ be (the unique) permutation in $\SymGp_n$ so that $\code_{w'}=1^k+\code_w$.  Then
\[\Groth^{(\beta)}_{w'}(\mathbf x;\mathbf y)=\left(\prod_{i=1}^k x_i\oplus y_1\right)\Groth^{(\beta)}_w(x_1,\ldots,x_{n-1};y_2,\ldots,y_{n}). \]
\end{lemma}
This lemma follows immediately from the ordinary pipe dream formula\footnote{Lemma~\ref{lemma:GrothAppendBoxes} can also be proved directly from first principles, i.e.,\ the following arguments are not logically dependent on pipe dreams.  Regardless, we omit the proof.}.

\begin{lemma}
\label{lemma:technicaldivdifA}
Let $w\in \SymGp_n$.  Suppose $w_i=1$ for some $i\in[\des(w)-2]$.  Let $v=ws_i$.
If $i\not \in \phi(v)$ or $i+1\not \in \phi(v)$  then
\begin{enumerate}
\item $\phi(w)=s_i\cdot \phi(v)$ and 
\item for all $I\subseteq \phi(v)$, we have $\pi_i(\Groth^{(\beta)}_{v_I})=\Groth^{(\beta)}_{w_{s_i\cdot I}}$.
\end{enumerate} 
\end{lemma}
\begin{proof}

Throughout, write $(a,b)=\mc(w)$ and let $b'=w^{-1}(b)$.

\noindent (1) $D(v)$ and $D(w)$ only differ in rows $i$ and $i+1$.  Since $i+1<\des(w)$, $\mc(v)=\mc(w)$.  

\noindent Case 1:
If $i+1 \not \in \phi(v)$, then there is some $j\in \phi(v)$ with $j\neq i+1$ so that $(j,v_j)\in [i+1,a]\times [1,b]$.  As such, $(j,v_j)=(j,w_j)\in [i,a]\times [1,b]$ and so $i\not \in \phi(w)$.  
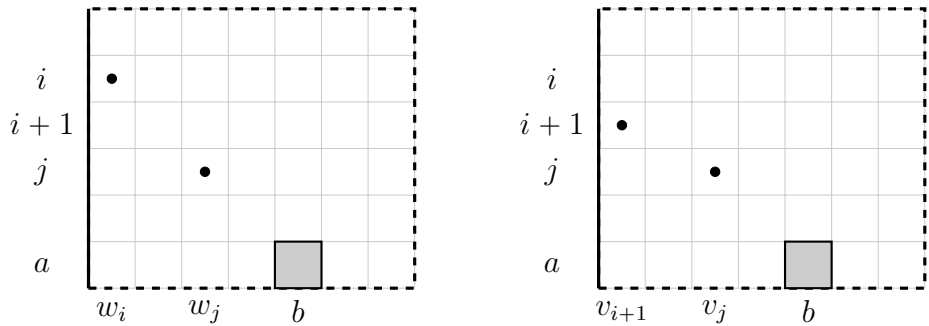
\[
\begin{tikzpicture}[x=1.5em,y=1.5em]
\draw[step=1,gray!40, thin] (0,0) grid (7,6);
\draw[black, very thick,dashed] (0,0) rectangle (7,6);
		\draw[color=black,very thick](0,0)--(0,6);
		\filldraw [black](.5,4.5)circle(.1);
\filldraw [black](2.5,2.5)circle(.1);
\draw[color=black, fill=gray!40, thick](4,0)rectangle(5,1);
\draw (-1,4.5) node  [align=left] {$i$};
\draw (-1,3.5) node  [align=left] {$i+1$};
\draw (-1,2.5) node  [align=left] {$j$};
\draw (-1,0.5) node  [align=left] {$a$};
\draw (.5,-.5) node [align=left]{$w_i$};
\draw (2.5,-.5) node [align=left]{$w_j$};
\draw (4.5,-.5) node [align=left]{$b$};
\end{tikzpicture}
\hspace{3em}
\begin{tikzpicture}[x=1.5em,y=1.5em]
\draw[step=1,gray!40, thin] (0,0) grid (7,6);
\draw[black, very thick,dashed] (0,0) rectangle (7,6);
		\draw[color=black,very thick](0,0)--(0,6);
\filldraw [black](2.5,2.5)circle(.1);
		\filldraw [black](.5,3.5)circle(.1);
\draw[color=black, fill=gray!40, thick](4,0)rectangle(5,1);
\draw (-1,4.5) node  [align=left] {$i$};
\draw (-1,3.5) node  [align=left] {$i+1$};
\draw (-1,2.5) node  [align=left] {$j$};
\draw (-1,0.5) node  [align=left] {$a$};
\draw (.5,-.5) node [align=left]{$v_{i+1}$};
\draw (2.5,-.5) node [align=left]{$v_j$};
\draw (4.5,-.5) node [align=left]{$b$};
\end{tikzpicture}
\]
Furthermore, $i\in \phi(v)$ if and only if $i+1\in \phi(w)$.  No other pivots of $(a,b)$ are created or deleted by swapping rows $i$ and $i+1$.  Therefore, $\phi(w)=s_i\cdot \phi(v)$.

\noindent Case 2:  Now suppose $i+1\in \phi(v)$.  By hypothesis, since $i+1\in \phi(v)$, we must have  $i\not \in \phi(v)$.  This implies $v_i>b$.
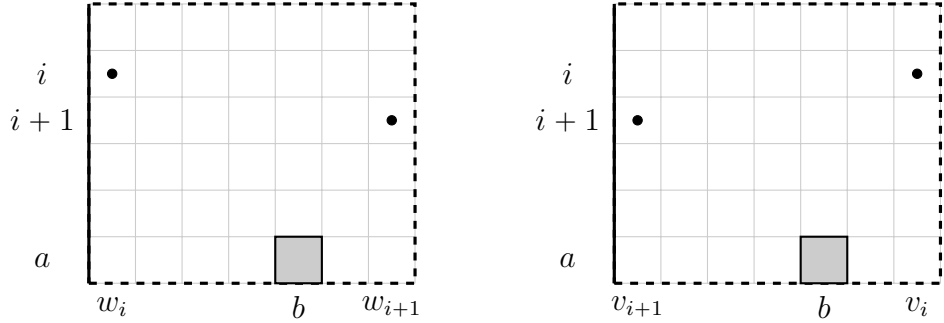
\[
\begin{tikzpicture}[x=1.5em,y=1.5em]
\draw[step=1,gray!40, thin] (0,0) grid (7,6);
\draw[black, very thick,dashed] (0,0) rectangle (7,6);
		\draw[color=black,very thick](0,0)--(0,6);
		\filldraw [black](.5,4.5)circle(.1);
\filldraw [black](6.5,3.5)circle(.1);
\draw[color=black, fill=gray!40, thick](4,0)rectangle(5,1);
\draw (-1,4.5) node  [align=left] {$i$};
\draw (-1,3.5) node  [align=left] {$i+1$};
\draw (-1,0.5) node  [align=left] {$a$};
\draw (.5,-.5) node [align=left]{$w_i$};
\draw (6.5,-.5) node [align=left]{$w_{i+1}$};
\draw (4.5,-.5) node [align=left]{$b$};
\end{tikzpicture}
\hspace{3em}
\begin{tikzpicture}[x=1.5em,y=1.5em]
\draw[step=1,gray!40, thin] (0,0) grid (7,6);
\draw[black, very thick,dashed] (0,0) rectangle (7,6);
		\draw[color=black,very thick](0,0)--(0,6);
\filldraw [black](6.5,4.5)circle(.1);
		\filldraw [black](.5,3.5)circle(.1);
\draw[color=black, fill=gray!40, thick](4,0)rectangle(5,1);
\draw (-1,4.5) node  [align=left] {$i$};
\draw (-1,3.5) node  [align=left] {$i+1$};
\draw (-1,0.5) node  [align=left] {$a$};
\draw (.5,-.5) node [align=left]{$v_{i+1}$};
\draw (6.5,-.5) node [align=left]{$v_i$};
\draw (4.5,-.5) node [align=left]{$b$};
\end{tikzpicture}
\]
Again, no other pivots are impacted by exchanging rows $i$ and $i+1$ and so \[\phi(w)=(\phi(v)\cup\{i\})-\{i+1\}=s_i\cdot \phi(v).\]

\noindent (2) If $i+1 \not \in I$, then for all $I\subseteq \phi(v)$, $v_I(i+1)=1$ and so $v_I$ has a descent at $i$.

 On the other hand, suppose $i+1\in \phi(v)$ but $i\not \in \phi(v)$.  Then we  have $v(i)>b$.  Furthermore, for any $I\subseteq \phi(v)$ we have $v_I(i+1)\leq b<v_I(i)$.  Hence, $v_I$ has a descent at $i$.

In both cases, by (\ref{eq:betagrothdef}), $\pi_i(\Groth^{(\beta)}_{v_I})=\Groth^{(\beta)}_{v_Is_i}$.  Applying Lemma~\ref{lemma:cycles}, we see \[v_Is_i=w s_i t_{a\,b'}  c^{(a)}_{I} s_i=w  t_{a\,b'} s_i c^{(a)}_{I} s_i= w t_{a\,b'} c^{(a)}_{s_i\cdot I}=w_{s_i\cdot I}. \qedhere\]
Thus, $\pi_i(\Groth^{(\beta)}_{v_I})=\Groth^{(\beta)}_{w_{s_i\cdot I}}$.
\end{proof}

\begin{lemma}
\label{lemma:technicaldivdifB}
Let $w\in \SymGp_n$.  Suppose $w_i=1$ for some $i\in[\des(w)-2]$.  Let $v=ws_i$.
If $i,i+1\in \phi(v)$ then $\phi(w)=\phi(v)-\{i\}$.
 Furthermore, given $I\subseteq \phi(v)$,
\[\pi_i(\Groth^{(\beta)}_{v_I})=
\begin{cases}
\Groth^{(\beta)}_{v_{I-\{i\}}}&\text{if}\enspace i,i+1\in I,\\
-\beta\Groth^{(\beta)}_{v_{I}}&\text{if}\enspace i\not \in I \enspace \text{and} \enspace i+1\in I, \enspace \text{and}\\
\Groth^{(\beta)}_{w_{s_i\cdot I}}&\text{if} \enspace i+1\not \in I.
\end{cases}\]
\end{lemma}
\begin{proof}

Write $(a,b)=\mc(w)$ and let  $b'=w^{-1}(b)$.

  Since $i,i+1\in \phi(v)$, if we swap rows $i$ and $i+1$ to obtain $w$, we see that  $(i+1,w(i+1))\in [i,a]\times[1,b]$ and so $i\not \in \phi(w)$.
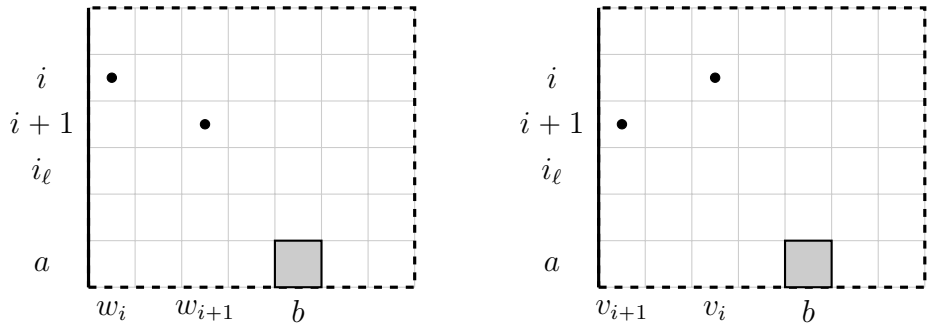
\[
\begin{tikzpicture}[x=1.5em,y=1.5em]
\draw[step=1,gray!40, thin] (0,0) grid (7,6);
\draw[black, very thick,dashed] (0,0) rectangle (7,6);
		\draw[color=black,very thick](0,0)--(0,6);
		\filldraw [black](.5,4.5)circle(.1);
\filldraw [black](2.5,3.5)circle(.1);
\draw[color=black, fill=gray!40, thick](4,0)rectangle(5,1);
\draw (-1,4.5) node  [align=left] {$i$};
\draw (-1,3.5) node  [align=left] {$i+1$};
\draw (-1,2.5) node  [align=left] {$i_\ell$};
\draw (-1,0.5) node  [align=left] {$a$};
\draw (.5,-.5) node [align=left]{$w_{i}$};
\draw (2.5,-.5) node [align=left]{$w_{i+1}$};
\draw (4.5,-.5) node [align=left]{$b$};
\end{tikzpicture}
\hspace{3em}
\begin{tikzpicture}[x=1.5em,y=1.5em]
\draw[step=1,gray!40, thin] (0,0) grid (7,6);
\draw[black, very thick,dashed] (0,0) rectangle (7,6);
		\draw[color=black,very thick](0,0)--(0,6);
\filldraw [black](2.5,4.5)circle(.1);
		\filldraw [black](.5,3.5)circle(.1);
\draw[color=black, fill=gray!40, thick](4,0)rectangle(5,1);
\draw (-1,4.5) node  [align=left] {$i$};
\draw (-1,3.5) node  [align=left] {$i+1$};
\draw (-1,2.5) node  [align=left] {$i_\ell$};
\draw (-1,0.5) node  [align=left] {$a$};
\draw (.5,-.5) node [align=left]{$v_{i+1}$};
\draw (2.5,-.5) node [align=left]{$v_{i}$};
\draw (4.5,-.5) node [align=left]{$b$};
\end{tikzpicture}
\]
This swap does not affect other pivots.  Therefore, $\phi(w)=\phi(v)-\{i\}$.

 Now fix $I\subseteq \phi(v)$.

\noindent Case 1: Suppose $i,i+1\in I$.

In this case, $v_I(i)>v_I(i+1)$ and so
$\pi_i(\Groth^{(\beta)}_{v_I})=\Groth^{(\beta)}_{v_Is_i}$.
By Lemma~\ref{lemma:cycles}, since $i,i+1\in I$,
\[v_Is_i=vt_{a\,b'}  c^{(a)}_{I} s_i=vt_{a\,b'} c^{(a)}_{I-\{i\}}=v_{I-\{i\}}.\]
Thus,
$\pi_i(\Groth^{(\beta)}_{v_I})=\Groth^{(\beta)}_{v_{I-\{i\}}}$.

\noindent Case 2: Assume $i\not \in I$ and $i+1\in I$.

We have $v_I(i)<v_I(i+1)$.  Thus by (\ref{eqn:dividedbetaterm}), $\pi_i(\Groth^{(\beta)}_{v_I})=-\beta \Groth^{(\beta)}_{v_I}$.

\noindent Case 3: Suppose $i+1 \not \in I$.

If $i+1\not \in I$, then $v_{I}(i+1)=1$.  Thus  $v_I$ has a descent at $i$.  Therefore, $\pi_i(\Groth^{(\beta)}_{v_I})=\Groth^{(\beta)}_{v_Is_i}$. By Lemma~\ref{lemma:cycles}, 
\[v_Is_i=ws_it_{a\,b'} c^{(a)}_{I} s_i=wt_{a\,b'}s_i c^{(a)}_{I} s_i=wt_{a\,b'} c^{(a)}_{s_i\cdot I}=w_{s_i\cdot I}.\]
Therefore, $\pi_i(\Groth^{(\beta)}_{v_I})=\Groth^{(\beta)}_{w_{s_i\cdot I}}$.
\end{proof}

\begin{lemma}
\label{lemma:technicaldivdif2}
Suppose $w\in \SymGp_n$ so that $\mc(w)=(a,b)$ and $w_{a-1}=1$.  Let $v=ws_{a-1}$.  Then the following statements hold.
\begin{enumerate}
\item  $\phi(w)=\phi(v)\sqcup\{a-1\}$.
\item If $I\subseteq \phi(v)$ then $v_I=w_{I\cup\{a-1\}}=w_Is_{a-1}$  and $\pi_{a-1}(\Groth^{(\beta)}_{v_I})=\Groth^{(\beta)}_{w_I}$.
\end{enumerate}
\end{lemma}
\begin{proof}
\noindent (1) 
First, notice that $\mc(v)=(a-1,b)$.
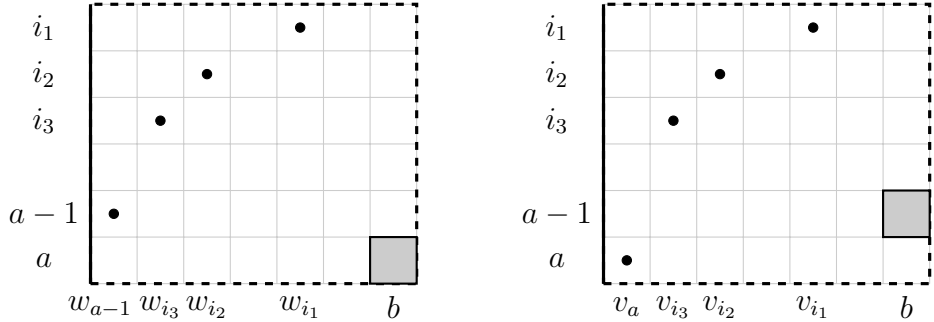
\[\begin{tikzpicture}[x=1.5em,y=1.5em]
\draw[step=1,gray!40, thin] (0,0) grid (7,6);
\draw[black, very thick,dashed] (0,0) rectangle (7,6);
		\draw[color=black,very thick](0,0)--(0,6);
		\filldraw [black](.5,1.5)circle(.1);
		\filldraw [black](1.5,3.5)circle(.1);
		\filldraw [black](2.5,4.5)circle(.1);
		\filldraw [black](4.5,5.5)circle(.1);
\draw[color=black, fill=gray!40, thick](6,0)rectangle(7,1);
\draw (-1,5.5) node  [align=left] {$i_1$};
\draw (-1,4.5) node  [align=left] {$i_2$};
\draw (-1,3.5) node  [align=left] {$i_3$};
\draw (-1,1.5) node  [align=left] {$a-1$};
\draw (-1,0.5) node  [align=left] {$a$};
\draw (.2,-.5) node [align=left]{$w_{a-1}$};
\draw (4.5,-.5) node [align=left]{$w_{i_1}$};
\draw (2.5,-.5) node [align=left]{$w_{i_2}$};
\draw (1.5,-.5) node [align=left]{$w_{i_3}$};
\draw (6.5,-.5) node [align=left]{$b$};
\end{tikzpicture}
\hspace{3em}
\begin{tikzpicture}[x=1.5em,y=1.5em]

\draw[step=1,gray!40, thin] (0,0) grid (7,6);
\draw[black, very thick,dashed] (0,0) rectangle (7,6);
		\draw[color=black,very thick](0,0)--(0,6);
		\filldraw [black](.5,.5)circle(.1);
		\filldraw [black](1.5,3.5)circle(.1);
		\filldraw [black](2.5,4.5)circle(.1);
		\filldraw [black](4.5,5.5)circle(.1);
\draw[color=black, fill=gray!40, thick](6,1)rectangle(7,2);
\draw (-1,5.5) node  [align=left] {$i_1$};
\draw (-1,4.5) node  [align=left] {$i_2$};
\draw (-1,3.5) node  [align=left] {$i_3$};
\draw (-1,1.5) node  [align=left] {$a-1$};
\draw (-1,0.5) node  [align=left] {$a$};
\draw (.5,-.5) node [align=left]{$v_{a}$};
\draw (4.5,-.5) node [align=left]{$v_{i_1}$};
\draw (2.5,-.5) node [align=left]{$v_{i_2}$};
\draw (1.5,-.5) node [align=left]{$v_{i_3}$};
\draw (6.5,-.5) node [align=left]{$b$};
\end{tikzpicture}
\]
 Since $w_{a-1}=1$ and $w_a>b$, automatically $a-1\in \phi(w)$.  Because $\mc(v)=(a-1,b)$,  $a-1\not \in \phi(v)$.  Furthermore, for any $i<a-1$, $i\in \phi_{a}(w)$ if and only if $i\in \phi(v)$.  Thus, $\phi(w)=\phi(v)\sqcup\{a-1\}$.

\noindent (2) Take $I\subseteq \phi(v)$.  Notice $b':=w^{-1}(b)=v^{-1}(b)$.    Then we have 
\[v_I=ws_{a-1}t_{a-1\, b'} c^{(a-1)}_I
=wt_{a\,b'} s_{a-1} c^{(a-1)}_I
=wt_{a\,b'} c^{(a)}_{I\cup\{a-1\}}.
\]
Thus, $v_I=w_{I\cup\{a-1\}}$.

Also, 
$s_{a-1}c_I^{(a-1)}=c^{(a)}_Is_{a-1}$ and so
$v_I=wt_{a\,b'}c_I^{(a)}s_{a-1}=w_I^{(a)}s_{a-1}.$
  We know $v_I(a)=(w_Is_{a-1})(a)=w_I(a-1)=1$.  Therefore, $v_I$ has a descent at $a-1$ and so \[\pi_{a-1}(\Groth^{(\beta)}_{v_I})=\Groth^{(\beta)}_{v_Is_{a-1}}=\Groth^{(\beta)}_{w_I}.\qedhere\]
\end{proof}

With these preliminaries, we now prove transition.

\begin{proof}[Proof of Theorem~\ref{thm:transition}]
 We will first induct on $n$. There is nothing to check for $\SymGp_1$.  Now fix $n>1$ and assume transition holds for permutations in $\SymGp_{n-1}$.

Take $w\in \SymGp_n$ and let $(a,b)=\mc(w)$.  
    As a secondary inductive hypothesis,  assume transition holds for permutations $u\in \SymGp_n$ such that $\ell(u)>\ell(w)$.  By direct computation, one may verify transition holds for the base case of $w_0\in \SymGp_n$.

If $w$ is the identity, then there is nothing to show.  So assume not.  We proceed by case analysis.

\noindent Case 1: Assume $w_i>1$ for all $i\in[a-1]$.

We have $\code_w(i)\geq 1$ for all $i\in [a-1]$.  Then there exists $v\in \SymGp_{n-1}$ so that $\code_w=\code_v+1^{a-1}$.   By induction on $n$, transition holds for $v$. 
Write $\widetilde{\mathbf x}=(x_1,\,\ldots,\,x_{n-1})$ and $\widetilde{\mathbf y}=(y_2,\,\ldots,\,y_{n})$.

\noindent Subcase 1: $\code_w(a)=1$.

In this case, $b=1$ and $(a,b)$ is in the dominant part of $D(w)$.  Therefore, $\phi(w)=\emptyset$.  Notice that $\code_{w_\emptyset}=\code_v+1^{a-1}$.  Then applying Lemma~\ref{lemma:GrothAppendBoxes}, we see
\begin{align*}
\Groth^{(\beta)}_w(\mathbf x;\mathbf y)&=\left(\prod_{i=1}^ax_i\oplus y_1\right)\mathfrak G_v(\widetilde{\mathbf x};\widetilde{\mathbf y})\\
&=(x_a\oplus y_1)\left(\prod_{i=1}^{a-1}x_i\oplus y_1\right)\Groth^{(\beta)}_v(\widetilde{\mathbf x};\widetilde{\mathbf y})\\
&=(x_a\oplus y_1) \Groth^{(\beta)}_{w_\emptyset}(\mathbf x;\mathbf y).
\end{align*}

\noindent Subcase 2: $\code_w(a)>1$.

Observe $\mc(v)=(a,b-1)$.  In particular, we have $\phi(w)=\phi(v)$.  Furthermore, for all $I\subseteq \phi(w)$, notice $\code_{w_I}=\code_{v_I}+1^a$.
 By induction on $n$, transition holds for $v$.  Applying Lemma~\ref{lemma:GrothAppendBoxes}, we see
\begin{align*}
\Groth^{(\beta)}_w(\mathbf x;\mathbf y)
&=\left(\prod_{i=1}^ax_i\oplus y_1\right)\Groth^{(\beta)}_v(\widetilde{\mathbf x}; \widetilde{\mathbf y})\\
&=\left(\prod_{i=1}^ax_i\oplus y_1\right)(x_a\oplus y_b)\Groth^{(\beta)}_{v_\emptyset}(\widetilde{\mathbf x};\widetilde{\mathbf y})\\
& \hspace{2em}+\left(\prod_{i=1}^ax_i\oplus y_1\right)(1+\beta (x_a\oplus y_b))\sum_{I\subseteq \phi(v): I\neq  \emptyset}\beta^{|I|-1}\Groth^{(\beta)}_{v_I}(\widetilde{\mathbf x};\widetilde{\mathbf y})\\
&=(x_a\oplus y_b)\Groth^{(\beta)}_{w_\emptyset}(\mathbf x;\mathbf y)+(1+\beta (x_a\oplus y_b))\sum_{I\subseteq \phi(w):I\neq\emptyset}\beta^{|I|-1}\Groth^{(\beta)}_{w_I}(\mathbf x;\mathbf y).
\end{align*}

\noindent Case 2:  Assume there some $i\in [a-2]$ such that  $w_i=1$.

 Set $v=ws_i$.  Since $\ell(v)=\ell(w)+1$, by induction, transition holds for $v$, i.e.,\
\[\Groth^{(\beta)}_{v}=(x_a\oplus y_b)\Groth^{(\beta)}_{v_\emptyset}+(1+\beta (x_a\oplus y_b))\sum_{I\subseteq \phi(v): I\neq  \emptyset}\beta^{|I|-1}\Groth^{(\beta)}_{v_I}.\]
Since $x_a\oplus y_b$ is symmetric in $x_i$ and $x_i+1$, by applying $\pi_i$ to both sides, we obtain:
\begin{equation}
\label{eq:transitioninductstep}
\Groth^{(\beta)}_{w}=(x_a\oplus y_b)\Groth^{(\beta)}_{w_\emptyset}+(1+\beta (x_a\oplus y_b))\sum_{I\subseteq \phi(v): I\neq  \emptyset}\beta^{|I|-1}\pi_i(\Groth^{(\beta)}_{v_I}).
\end{equation}

\noindent Subcase 1: Suppose $i\not \in \phi(v)$ or $i+1 \not \in \phi(v)$.  As an immediate application of Lemma~\ref{lemma:technicaldivdifA},
\[\Groth^{(\beta)}_{w}=(x_a\oplus y_b)\Groth^{(\beta)}_{w_\emptyset}+(1+\beta (x_a\oplus y_b))\sum_{I\subseteq \phi(v): I\neq  \emptyset}\beta^{|I|-1}\Groth^{(\beta)}_{w_{s_i\cdot I}}.\]
Also by Lemma~\ref{lemma:technicaldivdifA}, $\phi(w)=s_i\cdot \phi(v)$.  In particular, \[\{I:I\subseteq \phi(w)\}=\{s_i\cdot I:I\subseteq \phi(v)\}\] and so the result follows.

\noindent Subcase 2: Assume $i,i+1 \in \phi(v)$. 
Take $I\subseteq \phi(v)$ so that $i\not \in I$ and $i+1\in I$.  
By Lemma~\ref{lemma:technicaldivdifA}, 
\begin{equation}
\label{eq:zeroterms}
\beta^{|I|-1}\pi_i(\Groth^{(\beta)}_{v_{ I}})+\beta^{|I\cup\{i\}|-1}\pi_i(\Groth^{(\beta)}_{v_{I\cup\{i\}}})=\beta^{|I|-1}(-\beta)\Groth^{(\beta)}_{v_{ I}}+\beta^{|I\cup\{i\}|-1}\Groth^{(\beta)}_{v_{I}}=0.
\end{equation}
Now, suppose $I\subseteq \phi(v)$ so that $i,i+1\not \in I$.
Then
\begin{equation}
\label{eq:otherterms}
\beta^{|I|-1}\pi_i(\Groth^{(\beta)}_{v_{ I}})+\beta^{|I\cup\{i\}|-1}\pi_i(\Groth^{(\beta)}_{v_{I\cup\{i\}}})
=\beta^{|I|-1}\Groth^{(\beta)}_{w_{I}}+\beta^{|I\cup\{i+1\}|-1}\Groth^{(\beta)}_{w_{I\cup\{i+1\}}}.
\end{equation}
Since, $\phi(w)=\phi(v)-\{i\}$, applying  (\ref{eq:zeroterms}) and (\ref{eq:otherterms}) to (\ref{eq:transitioninductstep}) produces the desired equation.

\noindent Case 3: Assume $w_{a-1}=1$.

Let $v=ws_{a-1}$.  Since $\mc(w)=(a,b)$,  we know $\mc(v)=(a-1,b)$.  Furthermore, $\phi(w)=\phi(v)\cup\{a-1\}$.

Since $\ell(v)>\ell(w)$, by the inductive hypothesis, transition holds for $v$.  Therefore,
\[\Groth^{(\beta)}_v=(x_{a-1}\oplus y_b)\Groth^{(\beta)}_{v_\emptyset}+(1+\beta (x_{a-1}\oplus y_b))\sum_{I\subseteq \phi(v):I\neq  \emptyset}\beta^{|I|-1}\Groth^{(\beta)}_{v_I}.\]
We have  $\pi_{a-1}(x_{a-1}\oplus y_b)=1$ and so,
\begin{align*}
\pi_{a-1}((x_{a-1}\oplus y_b) \Groth^{(\beta)}_{v_\emptyset})
&=\Groth^{(\beta)}_{v_\emptyset}+(x_a\oplus y_b)\pi_i(\Groth^{(\beta)}_{v_\emptyset})+\beta(x_a\oplus y_b) \Groth^{(\beta)}_{v_\emptyset} &\text{(by (\ref{eqn:leibniz}))} \\
&=\left(1+\beta(x_a\oplus y_b)\right)\Groth^{(\beta)}_{w_{\{a-1\}}}+(x_a\oplus y_b)\Groth^{(\beta)}_{w_\emptyset} & \text{(by Lemma~\ref{lemma:technicaldivdif2}).}
\end{align*}
Furthermore, $\pi_{a-1}(1+\beta(x_{a-1}\oplus y_b))=0$.  By Lemma~\ref{lemma:technicaldivdif2},
\begin{align*}
\pi_{a-1}((1+\beta (x_{a-1}\oplus y_b)) \Groth^{(\beta)}_{v_I}(\mathbf x;\mathbf y))
&=(1+\beta (x_a\oplus y_b))(\pi_{a-1}(\Groth^{(\beta)}_{v_I})+\beta \Groth^{(\beta)}_{v_I}) & \text{(by (\ref{eqn:leibniz}))}\\
&=(1+\beta (x_a\oplus y_b))(\Groth^{(\beta)}_{w_I}+\beta \Groth^{(\beta)}_{w_{I\cup\{a-1\}}}).
\end{align*}
\begin{align*}
\Groth^{(\beta)}_w&
=(x_a\oplus y_b)\Groth^{(\beta)}_{w_\emptyset}+\left(1+\beta(x_a\oplus y_b)\right)\Groth^{(\beta)}_{w_{\{a-1\}}}\\
&\quad+(1+\beta x_a\oplus y_b)\sum_{I\subseteq \phi(v):I\neq  \emptyset}\beta^{|I|-1}(\Groth^{(\beta)}_{w_I}+\beta \Groth^{(\beta)}_{w_{I\cup\{a-1\}}})\\
&=(x_a\oplus y_b)\Groth^{(\beta)}_{w_\emptyset}+\left(1+\beta x_a\oplus y_b \right)\sum_{I\subseteq \phi(w):I\neq\emptyset}\beta^{|I|-1}\Groth^{(\beta)}_{w_I}  & \text{(by Lemma~\ref{lemma:technicaldivdif2})}.
\end{align*}
Therefore, transition holds for $w$.
\end{proof}

\section*{Acknowledgments}
This project was inspired by a talk of Thomas Lam which took place at the Ohio State Schubert Calculus Conference (2018).  The author is grateful to Zachary Hamaker and Oliver Pechenik for numerous discussions about bumpless pipe dreams.  The author also benefited from helpful conversations and correspondence with Sara Billey, William Fulton, Allen Knutson, Thomas Lam, Karola M\'esz\'aros, Mark Shimozono, David Speyer, and Alexander Yong.

%
%
\bibliographystyle{amsalpha} 
\bibliography{bumpless}

\end{document}